\documentclass[12pt]{article} 
\usepackage[utf8]{inputenc}
\usepackage[T1]{fontenc}
\usepackage[english]{babel}
\usepackage{ulem}

\usepackage{amsmath,amsfonts,amssymb,amsthm,mathrsfs,amsrefs, bbm}
\usepackage{tikz}

\usepackage{xfrac}
\renewcommand{\leq}{\leqslant}
\renewcommand{\geq}{\geqslant}

\usepackage[cal=euler]{mathalfa}
\usepackage[mono=false]{libertine}
\useosf
\usepackage[left=27mm,right=27mm,top=30mm,bottom=32mm]{geometry}

\usepackage[font={small,sf}, labelfont={sf,bf}, margin=1cm]{caption}
\usepackage[pdftex,colorlinks=true]{hyperref}
\usepackage[expansion]{microtype}

\usepackage{stackrel}
\usepackage{soul} %\st{} pour barrer

\theoremstyle{plain}
\newtheorem{theorem}{Theorem}
\newtheorem{corollary}[theorem]{Corollary}
\newtheorem{proposition}[theorem]{Proposition}
\newtheorem{lemma}[theorem]{Lemma}
\theoremstyle{definition}
\newtheorem{definition}[theorem]{Definition}

\newtheorem{remark}[theorem]{Remark}

\newcommand{\ind}{\textbf{1}}

\newcommand{\R}{\mathbb{R}}

\newcommand{\bT}{\boldsymbol{T}}
\newcommand{\btau}{\boldsymbol{\tau}}

\newcommand{\rmd}{\mathrm{d}}

\newcommand{\pants}{\includegraphics[width=0.6cm]{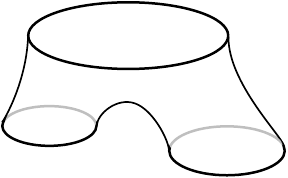}}
\newcommand{\smallpants}{\includegraphics[width=0.3cm]{images/pant}}
\newcommand{\idealtrig}{\includegraphics[width=0.5cm]{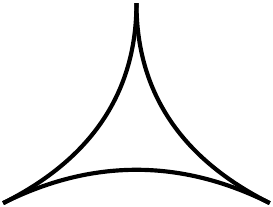}}
\newcommand{\smallidealtrig}{\includegraphics[width=0.3cm]{images/idealtrig}}
\newcommand{\reddot}{\color{red} \bullet}
\newcommand{\trig}{\triangle}
\newcommand{\trigspace}{\triangle\!\!\!\!\!\!\;\triangle} 
\newcommand{\teich}{\mathcal{T}} 
\newcommand{\redl}{\includegraphics[height=2mm]{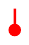}}
\newcommand{\smallsquare}{\scalebox{0.7}{$\square$}}

\hypersetup{
    pdftitle    = {}}

\usepackage{overpic, float}

\begin{document}
 
	\title{\vspace{-2cm}\bf Random punctured hyperbolic surfaces \\\& the Brownian sphere}
	\author{\textsc{Timothy Budd}\footnote{Radboud University, Nijmegen, The Netherlands. Email: \href{mailto:t.budd@science.ru.nl}{t.budd@science.ru.nl}} \, and \textsc{Nicolas Curien}\footnote{Universit\'e Paris-Saclay. E-mail: \href{mailto:nicolas.curien@gmail.com}{nicolas.curien@gmail.com}.}}
	\date{}
	\maketitle

\vspace{-1cm}
\begin{abstract}
We consider  random genus-$0$ hyperbolic surfaces $\mathcal{S}_n$ with $n + 1$ punctures, sampled according to the Weil--Petersson measure. We show that, after rescaling the metric by $n^{-1/4}$, the surface $\mathcal{S}_n$ converges in distribution to the {Brownian sphere}---a random compact metric space homeomorphic to the $2$-sphere, exhibiting fractal geometry and appearing as a universal scaling limit in various models of random planar maps. Without rescaling the metric, we establish a local Benjamini--Schramm convergence of $\mathcal{S}_n$ to a random infinite-volume hyperbolic surface with countably many punctures, homeomorphic to $\mathbb{R}^2 \setminus \mathbb{Z}^2$. Our proofs mirror techniques from the theory of random planar maps. In particular, we develop an encoding of punctured hyperbolic surfaces via a family of plane trees with continuous labels, akin to Schaeffer's bijection. This encoding stems from the Epstein--Penner decomposition and, through a series of transformations, reduces to a model of single-type Galton--Watson trees, enabling the application of known invariance principles.
\end{abstract}

\begin{figure}[!h]
  \centering
  \includegraphics[width=0.9\linewidth]{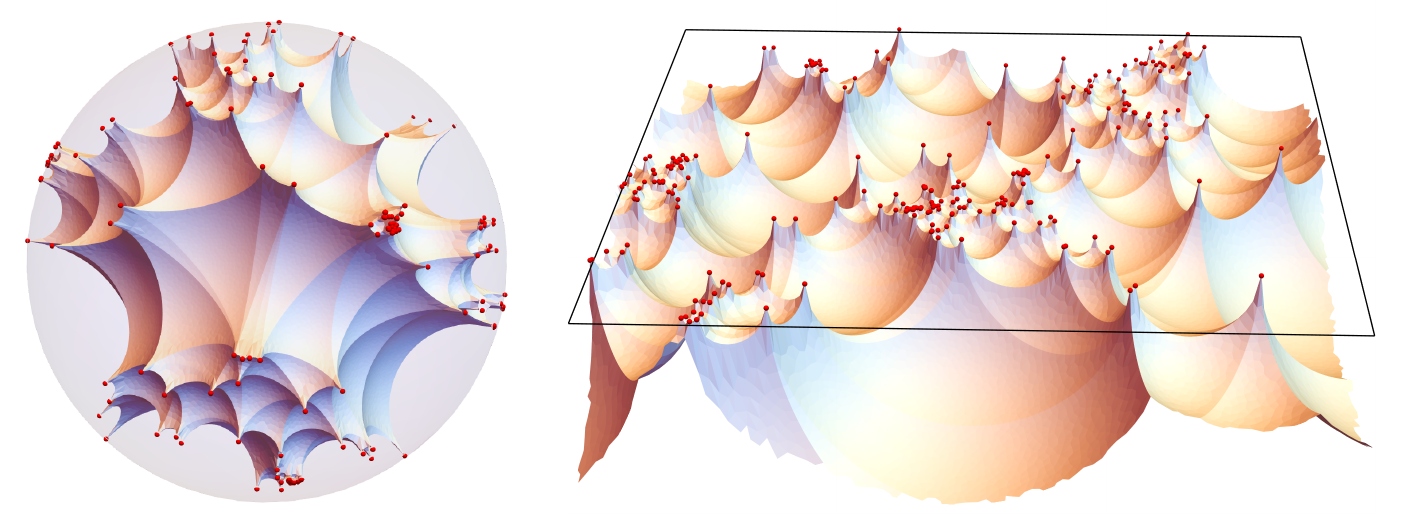}
  \captionsetup{width=1.1\textwidth}
  \caption{Simulations of Weil--Petersson random surfaces: Left $\mathcal{S}_{n} \in \mathcal{M}_{0,n+1}$ for $n=282$ shown as the boundary geometry of the unique corresponding convex ideal hyperbolic polyhedron in the Poincar\'e ball model. Right: A portion of the random infinite hyperbolic surface $\mathcal{S}_\infty$ shown as the boundary geometry of the corresponding convex ideal hyperbolic polyhedron in the Poincar\'e half space model.
   \label{fig:simulationsA}}
\end{figure}

\tableofcontents

\section{Introduction}

The study of random planar geometries has been a very fruitful theme of research over the last few decades. Motivated by 2D quantum gravity \cite{ADJ97}, it has witnessed spectacular developments in particular in the theory of random planar maps and of random hyperbolic surfaces. 
\paragraph{Weil--Petersson random surfaces.} Recall that a \textbf{hyperbolic surface} is a 2d Riemannian manifold equipped with a complete metric with constant curvature $-1$. For $2g-2+n > 0$ the moduli space  $ \mathcal{M}_{g,n}$ of all isometry classes of hyperbolic surfaces of genus $g \geq 0$ with $n$ punctures (or cusps) is non-empty and is equipped with a natural finite measure called the Weil--Petersson measure and denoted here by $ \mathrm{WP}$, see Section \ref{sec:WP} for details. After normalization, this measure enables us to sample $  \mathcal{S}_{g,n} \in \mathcal{M}_{g,n}$ at random. 

\textsc{High genus.} Many properties of such random surfaces have been established in the breakthrough works of Mirzakhani \cite{mirzakhani2013growth} in the high-genus case. In particular, she proved that a random WP surface $ \mathcal{S}_{g,0} \in  \mathcal{M}_{g,0}$ is very connected as $g \to \infty$: it has logarithmic diameter, positive Cheeger constant, large typical injectivity radius although $\mathcal{S}_{g,0}$ possesses a few short cycles \cite{mirzakhani2017lengths}. In particular, the length spectrum of $\mathcal{S}_{g,n}$ when $g,n \to \infty$ displays a surprising universality, see \cite{budd2025tight,barazer2025length,janson2023unicellular} while the spectral gap property of the Laplacian operator on  $ \mathcal{S}_{g,0}$ has been the subject of many recent works, see e.g.~\cite{lipnowski2021towards,anantharaman2024spectral}. Many of those results are based on a clever estimations of the Weil--Petersson volumes $V_{g,n} = \mathrm{WP}( \mathcal{M}_{g,n})$, see \cite{mirzakhani2007simple,mirzakhani2007weil,Ey16}. We refer the reader to the beautiful survey articles \cite{do2011moduli,wright2020tour,petri2024extremal} for more details and references. Let us also mention that the above properties are believed to be shared by many other models of random hyperbolic surfaces in high genus, e.g.~the Brooks--Makover model \cite{BM04,budzinski2019diameter,petri2017random}, random gluing of pairs of pants \cite{budzinski2019minimal} or random covers \cite{magee2020random}.

\textsc{Planar case.} On the other hand, much less is known in the planar case $g=0$, although there is a renewed interested sparked by JT gravity, \cite{mertens2023solvable}. Let us mention that Zograf computed the volumes of punctured spheres $V_{0,n} = \mathrm{WP}( \mathcal{M}_{0,n})$ in \cite{MR1234274}, see also \cite{kaufmann1996higher} (we recover those formulas in Theorem \ref{thm:zograf} below). The properties of the length and Laplacian spectrum of $ \mathcal{S}_{0,n}$ in the $n \to \infty$ limit have recently been investigated in \cite{hide2025short,hide2025large}. In this work, we shall study in details both the local and the global geometry  of a random surface $\mathcal{S}_{0,n} \in \mathcal{M}_{0,n}$ using an adaptation of the approach used in the theory of random planar maps.

\paragraph{Random planar maps.} A \textbf{map} is a finite connected graph properly embedded in an orientable surface of genus $g \geq 0$ so that its faces are all homeomorphic to disks. They have widely been used (in particular in the physics literature) as discretization of continuous surfaces.  To fix ideas, let us focus on the case of rooted triangulations, i.e.~when all faces of the map are triangles and where an oriented edge is distinguished. We denote by $ \trigspace_{g,n}$ the set of all (rooted) triangulations in genus $g \geq 0$ with $n$ triangles. For each $n,g \geq 0$, we can then sample $\trig_{g,n} \in \trigspace_{g,n}$ uniformly at random in this finite set (provided it is non empty). The enumeration of $ \trigspace_{g,n}$  has its roots both in combinatorics (Tutte's equation) and physics (matrix integrals). We refer to \cite{Ey16} for references on this vast subject still under development. The local geometric properties of a random triangulation $\trig_{g,n} \in \trigspace_{g,n}$ are now well understood thanks to a recent breakthrough of  Budzinski \& Louf \cite{budzinski2019local} generalizing the case $g=0$ of Angel \& Schramm \cite{AS03}. They proved that  if $g=g(n)$ is such that $g(n)/n \to \alpha \in [0,1/4)$ then we have the following convergence in distribution   \begin{eqnarray}   \label{eq:convlocaltrig} \trig_{g(n),n} \xrightarrow[n\to\infty]{(d)} \trig_{\infty}^{(\alpha)},  \end{eqnarray}
   for the local topology, also called Benjamini--Schramm topology \cite{BS01}. The random triangulations $  \trig_{\infty}^{(\alpha)}$ are infinite random triangulations of the plane defined in \cite{CurPSHIT}, they generalize the UIPT of  Angel \& Schramm \cite{AS03} and have  hyperbolic flavor when $\alpha >0$  (mean degree strictly larger than $6$, exponential volume growth, positive speed for the random walk...), see \cite{CurStFlour} for details.  This result is tightly connected to the enumeration of $ \trigspace_{g(n),n}$. Down to earth, the convergence \eqref{eq:convlocaltrig} means that for any $r$, the  (random) combinatorial ball of radius $r$ around the root edge in $\trig_{g(n),n}$ converges in distribution as $n \to \infty$ towards the combinatorial ball of radius $r$ around the root edge in $ \trig_{\infty}^{( \alpha)}$. Although the above convergence describes the local geometry near the root edge in $\trig_{g(n),n}$, it does not say much on the global geometry of $\trig_{g(n),n}$, e.g.~about its diameter (see \cite{Louf20,budzinski2023distances} for very recent progress in that direction). This global geometry is however well understood in the low genus case, and let us present the planar case $g=0$ to simplify the exposition. For this, we shall consider $\trig_{0,n}$ as a random compact metric space equipped with its graph distance $ \mathrm{d_{gr}}$. Le Gall \cite{LG11} (see also Miermont \cite{Mie11} for the case of quadrangulations) then showed the following convergence  in distribution 
 \begin{eqnarray} \label{eq:convBMtrig} \big( \mathrm{Vertices}(\trig_{0,n}) , n^{{-1/4}} \cdot \mathrm{d_{gr}}\big) \xrightarrow[n\to\infty]{(d)} (  \mathbf{m}_{\infty}, 6^{- \frac{1}{4}} \cdot D^{*}), \end{eqnarray}
 for the Gromov--Hausdorff topology where $ ( \mathbf{m}_{\infty}, D^{*})$ is a random compact metric space almost surely homeomorphic to the $2$-sphere  \cite{LGP08} but with fractal dimension $4$ \cite{LG07} called the \textbf{Brownian sphere}.  Contrary to \eqref{eq:convlocaltrig} where the triangulations $\trig_{g,n}$ are seen as  combinatorial maps, in \eqref{eq:convBMtrig} they are interpreted as (random) compact metric spaces and the limit is now a (random) continuous metric space. This convergence has since then been extended to various classes of planar maps showing the universality of the Brownian sphere as a scaling limit of discrete planar geometries, see \cite{curien2019first,le2019brownian} for references. The  Brownian sphere has also been constructed from the Gaussian Free Field in a series of breakthroughs \cite{miller2015liouville,MS15,DMSgluingoftrees}. Contrary to the local limit \eqref{eq:convlocaltrig}, convergences towards the Brownian sphere like \eqref{eq:convBMtrig} are usually established using encoding  of planar maps via labeled trees. Scaling limits for random trees are by now well established, going back at least to the works of Aldous in the 90's \cite{Ald91a} on the Brownian Continuum Random Tree (CRT). In particular, many classes of random labeled trees are known to converge towards to the Brownian CRT  with a Brownian motion living on top, the so-called Brownian snake of Le Gall \cite{LeG99}. The Brownian Sphere is then constructed as a quotient space of the Brownian CRT for an equivalence relation determined by the Brownian snake. See Part III for background on the Gromov--Hausdorff topology and for the definition of the Brownian sphere.
 
 \paragraph{Main results.} In the following, for $n \geq 2$ we denote by $\mathcal{S}_n \equiv \mathcal{S}_{0,n+1} \in \mathcal{M}_{0,n+1}$ a random hyperbolic sphere with $n+1$ punctures  whose distribution is the normalized Weil--Petersson measure. Our main results are the analogs of \eqref{eq:convlocaltrig} and \eqref{eq:convBMtrig} in the case of the random hyperbolic surface $ \mathcal{S}_{n}$. Thanks to the hyperbolic structure, the random surface $  \mathcal{S}_{n}$ has a metric $ \mathrm{d_{ \mathrm{hyp}}}$ and carries a natural finite measure $ \mu_{ \mathrm{hyp}}$ whose total mass is given by the Gauss--Bonnet formula
 $$ \mu_{ \mathrm{hyp}}( \mathcal{S}_{n}) = 2\pi( n-1).$$
 This mass measure enables to sample, conditionally on $ \mathcal{S}_{n}$, a uniform point $\circ_{n} \in \mathcal{S}_{n}$ and root the surface at this point. This point plays the same role as the root edge in the theory of random maps and we have the analog of the local limit \eqref{eq:convlocaltrig} in the planar case:
 
\begin{theorem}["Local" Benjamini--Schramm convergence] \label{thm:BS}
The random hyperbolic surface $ \mathcal{S}_n$ converges in the Benjamini--Schramm sense, 
$$( \mathcal{S}_{n}, \circ_{n}) \xrightarrow[n\to\infty]{(d)} ( \mathcal{S}_{\infty}, \circ_\infty),$$ 
where  $ \mathcal{S}_\infty$ is a random hyperbolic surface with a distinguished point $\circ_\infty$ and  an infinite discrete set of punctures.
\end{theorem}
 In particular, $\mathcal{S}_{\infty}$ is a hyperbolic surface with genus $0$, countably many punctures and no other boundaries, hence  homeomorphic to $ \mathbb{R}^{2} \backslash \mathbb{Z}^{2}$, see \cite{richards1963classification}. A simulation of a piece of $  \mathcal{S}_{\infty}$ is on displayed on Figure \ref{fig:simulationsA}. Those surfaces are sometimes called flute surfaces, see \cite{pandazis2024parabolic}. The reader may wonder which type of local convergence we used in the above result: the local Gromov--Hausdorff--Prokhorov topology \cite{abraham2013note,bowen2025benjamini} or the stronger topology  on Riemannian manifolds introduced in \cite{abert2016unimodular} would work, but our proof actually yields a stronger result which states that for any $r>0$ we can couple the ball of radius $r$ around $\circ_{n}$ in $ \mathcal{S}_{n}$ with the ball of radius $r$ around $\circ_\infty$ in $ \mathcal{S}_{\infty}$ so that they coincide with a probability tending to $1$ as $n \to \infty$.  In the terminology of \cite{abert2016unimodular}, the random pointed hyperbolic surface $( \mathcal{S}_{\infty}, \circ_\infty)$ is a unimodular random Riemannian manifold, in particular its law is invariant under the geodesic flow along a uniform angle or along the Brownian motion.
 
We now move on to global scaling limit results for $ \mathcal{S}_{n}$.  Since $ \mathcal{S}_{n}$ is unbounded due to the cusps, some precautions are in order. It is well-known  \cite[Theorem 4.4.6]{BuserBook} 
 that the horocycle neighborhoods of length $1$ around each of the $n+1$ cusps are disjoint.
 We can thus consider the random metric space $\mathcal{S}_n^\circ$ obtained by removing around each puncture the open horocycle neighborhood of length $1$ which is now a metric space of finite diameter, more precisely a compact hyperbolic surface of genus $0$ with $n+1$ holes with geodesic boundary of length $1$. Although the cusps of the surface make it unbounded, they contain very little mass for the hyperbolic measure $\mu_{ \mathrm{hyp}}$ and typical points sampled from $\mu_{ \mathrm{hyp}}$ do not enter deep into the cusps. One way to formalize this is to use the Gromov--Prokhorov convergence of random measured metric spaces instead of using the Gromov--Hausdorff convergence (see Section \ref{sec:gromov} for a reminder about those notions of convergence). The analog of \eqref{eq:convBMtrig} in our case becomes:
\begin{theorem}[``Global'' Gromov--Hausdorff/Prokhorov scaling limits] \label{thm:GH} Recall that $( \mathbf{m}_\infty, D^*)$ denotes the normalized Brownian sphere and denote by $\mu$ its normalized mass measure. Introduce the scaling constant
$$c_{ \mathrm{wp}}= \frac{2\pi}{\sqrt{3 c_0}} = 2.3392\ldots,$$
where $c_0$ is the first zero of the Bessel function $J_0$.
Then we have the convergences in distribution 
 \begin{eqnarray} \label{eq:cvGH} \left( \mathcal{S}_n^\circ, n^{-1/4} \cdot \mathrm{d_{hyp}} \right) &\xrightarrow[n\to\infty]{(d)} &( \mathbf{m}_\infty, c_{ \mathrm{wp}} \cdot  D^*),\\
 \label{eq:cvGP} \left( \mathcal{S}_n, n^{-1/4} \cdot \mathrm{d_{hyp}},  (2 \pi (n-1))^{-1}\cdot \mu_{ \mathrm{hyp}} \right) &\xrightarrow[n\to\infty]{(d)} & ( \mathbf{m}_\infty, c_{ \mathrm{wp}} \cdot D^{*} , \mu),  \end{eqnarray} respectively for the Gromov--Hausdorff and the Gromov--Prokhorov topology.
\end{theorem}
See Figure \ref{fig:simulationsB} for simulations of the large scale structure of $\mathcal{S}_{n}$.
 It is worth noting that the scaling constant $c_{\mathrm{wp}}$ is similar to the scaling constant $\alpha_2^{-1} = 5^{-1/4} \tfrac{\pi}{\sqrt{3 c_0}}$ observed in \cite[Cor.~1.4]{budd2025tight} for the length of the non-separating systole of Weil-Petersson random surfaces of large genus with many cusps normalized by $(n/g)^{-1/4}$.
We direct the reader to \cite[Sec.~1.7]{BuddZonneveld24} for a discussion of how hyperbolic distance statistics in $\mathcal{S}_n$ might be linked directly to the topological recursion of Weil-Petersson volumes.
  
  \paragraph{When Penner meets Schaeffer.} The proof of our main results is based on a construction of the random surface $ \mathcal{S}_{n}$ from a certain decorated random tree. This can be seen as the analog of the classic ``Schaeffer-type'' constructions in the theory of random planar maps \cite{Sch98,BDFG04}. Our starting point is  the Penner--Bowditch--Epstein construction (due to Penner  \cite{penner1987decorated} and Bowditch-Epstein  \cite{epstein1988euclidean}). Let us recall those encodings of hyperbolic surfaces by starting with the classical pants-decomposition.   
  
  It is well-known that any hyperbolic surface $ X \in \cup \mathcal{M}_{g,n}$ can be constructed by gluing together pair of pants (possibly degenerate pair of pants having $1$ or $2$ cusps). Fixing a pant decomposition $\pants_{n,g}$,  the  boundary lengths and twists parameters, the so-called Fenchel--Nielsen coordinates, then produces a mapping 
  $$  \mathrm{FenNie}(\pants_{g,n}) : \big(\mathbb{R} \times \mathbb{R}_{+}\big)^{3g-3+n}  \equiv \teich_{g,n} \to \mathcal{M}_{g,n},$$ where $ \teich_{g,n}$ is the so-called Teichm\"uller space of hyperbolic surfaces, which is the universal cover of $ \mathcal{M}_{g,n}$. A different construction, restricted to the case $n \geq 1$ and  obtained by gluing ideal hyperbolic triangles  with some shearing  along the combinatorics of an ideal triangulation  $ \idealtrig_{g,n}$ has been conceived by Penner \cite{penner2012decorated}. This produces an encoding of the Teichm\"uller space $ \teich_{g,n}$ where additionally each puncture carries a decoration $ \epsilon \in \mathbb{R}_{>0}$, interpreted as a horocycle length, and where the coordinates are given by the (signed) hyperbolic lengths between horocycles:
  $$  \mathrm{Pen}(\idealtrig_{g,n}) :  \big(\mathbb{R} \times \mathbb{R}\big)^{3g-3+n} \times (\mathbb{R}_{>0})^{n}  \equiv  \widetilde{ \teich}_{g,n} \to \mathcal{M}_{g,n}.$$
After passing to the quotient by the mapping class group, the natural measure on $\widetilde{ \teich}_{g,n}$ descends to $ \mathcal{M}_{g,n}$ and produces the Weil--Petersson measure which is finite. However finding a fundamental domain in the (decorated) Teichm\"uller space for $ \mathrm{FenNie}( \pants)$ or  $ \mathrm{Pen}(\idealtrig)$ is a very hard problem. To find a ``unique'' encoding of a hyperbolic surface, the idea is to impose a further combinatorial constraint on the pants decomposition or the ideal triangulation which is mapping class group invariant, or in other words, prescribed by the isometry class of the surface only. In the case of the pants decomposition,  this is at the hear of MacShane's identities and Mirzakhani's recursion and will be developed in depth in  our forthcoming work \cite{BBCP25+}. In the case of Penner's mapping, this can be done by requiring that the underlying ideal triangulation is constructed from loci equidistant to specified horocycles (this is the spine construction of Bowditch and Epstein  \cite{epstein1988euclidean}). This set, called the \textbf{spine} or cut-locus, has a combinatorial structure of a map with $k$ faces where $k$ is the number of horocycles from which we measure the distances, and is in particular a tree when $k=1$. Such constructions are also at the heart of Schaeffer's bijection \cite{schaeffer1997bijective} and its variants in the realm of combinatorial maps. Specifically, the classical Cori--Vauquelin--Schaeffer's construction \cite{CV81,schaeffer1997bijective,CS04} is an encoding of planar pointed quadrangulations by a labeled plane tree, where the tree can be interpreted as the discrete cut-locus of the quadrangulation with respect to the pointed vertex and the labels are the distances to it. This was notably extended to the case of bipartite maps by Bouttier--Di Francesco--Guitter \cite{BDFG04} and to multi-pointed maps in arbitrary genus by Miermont \cite{Mie09} (the delays in Miermont are then analogs of the decorations on punctures in Penner, see also \cite{AB14}). These constructions have been pivotal in the development of the theory of random planar maps and their scaling limits, notably leading to the introduction of the Brownian sphere \cite{LG11,Mie11} and their stable analogs \cite{CRM22}.

\paragraph{Trees with angle assignments.} In our case of plane hyperbolic surfaces with cusps, by distinguishing an origin puncture and sending the other horocycle lengths to $0$, we revisit  Penner--Bowditch--Epstein construction to encode a generic surface $ \mathcal{S}_{n} \in \mathcal{M}_{0,n+1}$ by a plane cubic tree\footnote{i.e. each vertex has either degree $1$ or $3$. Those trees are also called binary in the following.} $ \tau_{n}$ with $n$ leaves together with angle assignments, thus providing us with an analog of Schaeffer's classical construction. More precisely, an allowed assignment of angles to a cubic tree is a labeling of the corners of its internal vertices by $(0, \pi)$ so that the sum of angles around a given internal vertex is $2 \pi$ and such that the sum of two angles ``opposite'' of an internal edge must be larger than $\pi$, see Figure \ref{fig:opposite} below.

\begin{figure}[!h]
 \begin{center}
 \includegraphics[width=8cm]{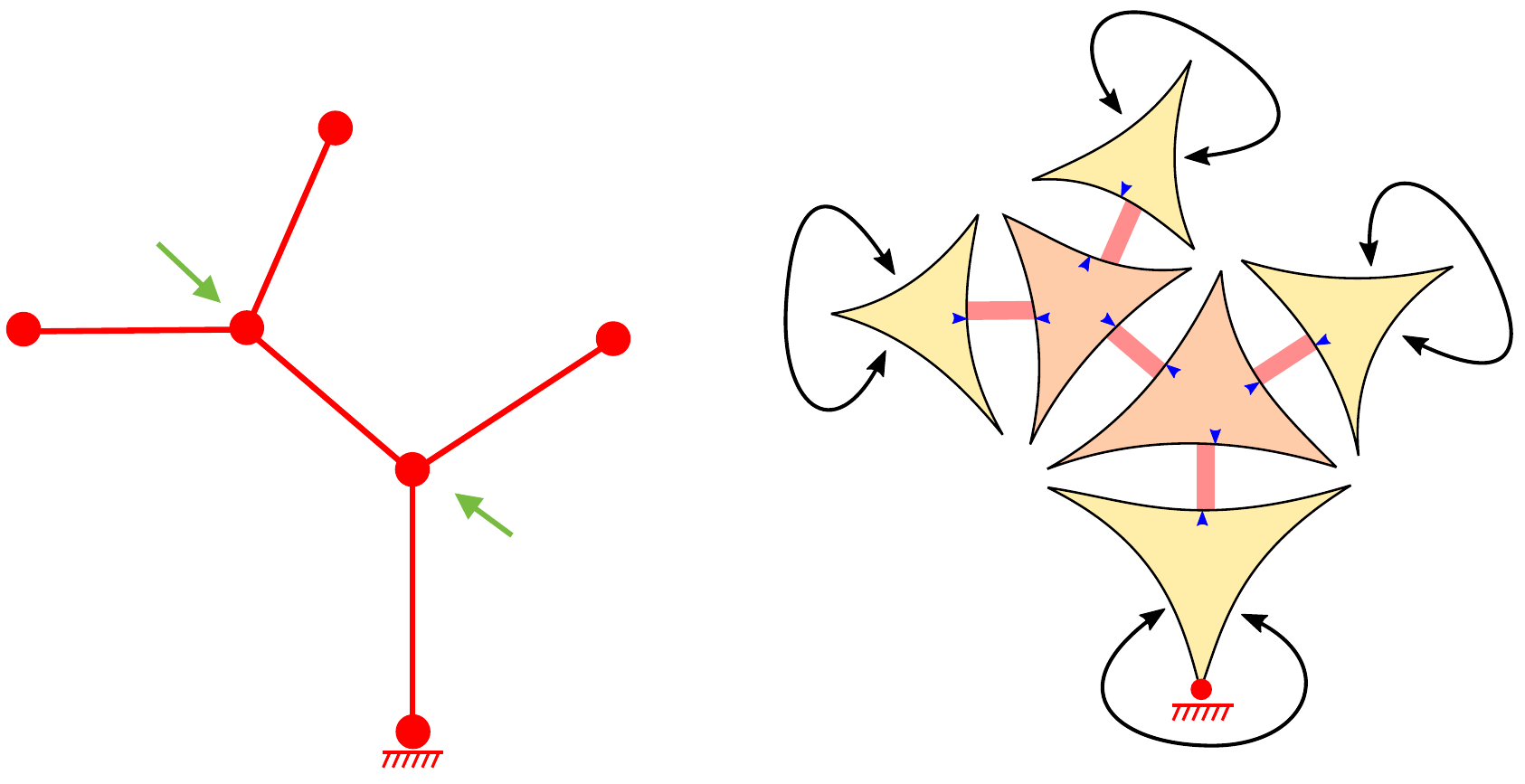}
 \caption{One can associate to a plane binary with $n$ leaves and  allowed angle assignment a hyperbolic surface with $n+1$ cusps. To be allowed, an angle configuration must belong to $(0,\pi)^{ 3n-6}$ and the sum of any two angles opposite of an edge (in green) must be larger than $\pi$. \label{fig:opposite}}
 \end{center}
 \end{figure} 
 
 The push-forward of the normalized Weil--Petersson measure by this encoding then yields the natural Lebesgue measure on such configurations (Theorem \ref{thm:treeencoding}). The technical crux consists then in understanding the behavior of such a random labeled tree $ \bT_{n} = (  T_{n}, ( \theta_{c} : c \in \mathrm{Corners}( T_{n}))$ sampled uniformly on trees with $n$ leaves with an allowed angle assignment. More precisely, we prove that when $n \to \infty$ such random trees admits a local limit when seen from a typical leaf or seen from a typical vertex (Theorem \ref{thm:localleaf} and  Corollary \ref{cor:aldous}). Our Theorem \ref{thm:BS} then follows by adapting the arguments of \cite{CMMinfini} to our hyperbolic setting: in a sense the local limit of $ \mathcal{S}_n$ is deduced from the local limit of its encoding labeled tree $\bT_n$. The same high-level picture is true for the scaling limit. We associate to the angle assignment $\theta$ of $\bT_n$ a labeling $\ell$ carried by the edges of the tree and prove that after renormalizing the tree by $ n^{{-1/2}}$ and  labels by $n^{{-1/4}}$ we get convergence towards Aldous' Brownian tree decorated with Le Gall's Brownian snake process (Theorem \ref{thm:scaling}). A central technical challenge is that the distribution of $\bT_n$ is more intricate than that of the labeled Galton--Watson trees typically encountered in the study of random planar maps via Schaeffer-type encodings. Moreover, the conditioning we impose on the number of leaves is less conventional in the existing literature. Last but not least, the precise calculation of the scaling constants $ c_{ \mathrm{wp}}$ is highly complex, as it necessitates determining the asymptotic variance of a real Markov chain whose transition probabilities involve Bessel functions.   With Theorem \ref{thm:scaling} at hands, the proof of Theorem \ref{thm:GH} proceeds along the same lines as Le Gall's approach in \cite{LG11}, provided we establish the appropriate analogs in our setting of key results from random planar map theory most notably, the so-called Schaeffer bound (see Proposition \ref{prop:Doupper} below). \bigskip 
 
To help the reading, we divided this work into three parts: The first one develops the labeled tree encoding of hyperbolic surfaces based on the Bowditch--Epstein--Penner construction. The second establishes the local and scaling limit of the random labeled trees found in the first part. It uses many techniques from the theory of random trees. Finally the third part translates back to random hyperbolic surfaces the results of the second part. For this, we import and adapt many ingredients developed to study local and scaling limits of random planar maps. \bigskip 
 
 \textbf{Acknowledgments.} The first author is supported by the VIDI programme with project number VI.Vidi.193.048, which is financed by the Dutch Research Council (NWO). The second author is supported by SuPerGRandMa, the ERC CoG no 101087572. 
 
 \section*{Index of notation}
To help the reader navigate through these pages, we gather here the most important notation used. 

\noindent \textbf{Tree bijection with surfaces}. \ \\
\noindent \begin{tabular}{ll}
$ \mathcal{M}_{g,n}$& moduli space of hyperbolic surfaces with genus $g$ and $n$ (labeled) punctures.\\
$ \mathcal{T}_{g,n}$& Teichm\"ulller space of hyperbolic surfaces with genus $g$ and $n$ (labeled) punctures.\\
$X \in \mathcal{M}_{g,n}$& a generic hyperbolic surface.\\
$ \mathrm{WP}$& Weil--Petersson measure.\\
$ V_{g,n}$  & $ = \mathrm{WP}( \mathcal{M}_{g,n})$.\\
$ \mathcal{S}_{n}$ & $ \in \mathcal{M}_{0,n+1}$ random punctured spheres with $n+1$ cusps.\\
$ \widetilde{\mathcal{T}}_{g,n}$& Decorated Teichm\"ulller space.\\
$  \epsilon_{1}, ... , \epsilon_{n+1}$ & horocycle lengths.\\
$ \ell_{\gamma}$ & (signed) distance between punctures.\\
$ \lambda_{\gamma}$ & $ = \sqrt{2 \mathrm{e}^{\ell_{\gamma}}}$ Penner's lambda-length.\\
$ \Sigma(X)$ & spine of $X$.\\
$ \btau $ &$ = (\tau, \ell / \theta /\lambda)$ generic binary labeled (red) tree.\\
$ \mathrm{Glue}( \btau)$ & surface obtained from the labeled tree $ \btau$.
\end{tabular} 

\medskip 

\noindent \textbf{Enumeration}. \ \\
\noindent \begin{tabular}{ll}
$J_{0},J_{1}$ & Bessel functions of the first kind.\\
$S(y)$& $ = \frac{\sqrt{y}}{\pi} J_1(2\pi \sqrt{y})$.\\
$Z(\cdot) $ & generation function of punctured surfaces, see \eqref{eq:treevolumeZ}.\\
$ F( x ; \theta) $& generating function of tree with angle constraint at the beginning, see \eqref{def:Fxtheta}.\\
$c_{0}$ & first zero of $J_{0}$.\\
$x_{c}$ & $=  \frac{c_{0}}{2\pi^{2}} J_{1}(c_{0})$.\\
$Z(x_{c})$ & $=  (\frac{c_{0}}{2\pi})^{2}$.\\
$ B_{x}( \cdot), B^{\redl}(\cdot)$ & generating functions for blobs, see before Lemma \ref{lem:decomposition}.\\
$F_{\infty}(\theta)$ &$=J_{0}( \frac{\theta}{\pi}c_{0})$.\\
$q$ &$=F(x_{c}; \pi/2) = \frac{c_{0}}{\pi^{2}} J_{1}(c_{0}/2)$.\\
$ \mathbb{P}_{x, \theta}, \mathbb{P}_{\theta}$ & Boltzmann law on labeled binary trees.\\
$ \mathbb{P}_{ \theta}^{n}, \mathbb{P}^{\infty}_{\theta}$ & conditioned laws.\\
\end{tabular}
\medskip 

\noindent \textbf{Limits of random trees}. \ \\
\noindent \begin{tabular}{ll}
$ |\tau|$& number of leaves of a plane tree.\\
$\pi( \btau)$ & blob tree.\\
$ \mathfrak{t}$& generic (bicolored) blob tree.\\
$ (p_{k})_{k \geq 0}$ & offspring law of variance $ \sigma^{2}$, see Definition \ref{def:pk}.\\
$ (r_{k})_{k \geq 0}$ & probability to add a red leaf to a black vertex of degree $k$, see Definition \ref{def:pk}.\\
$ (p^{\redl}_{k})_{k \geq 0}$ & modified offspring distribution, see \eqref{eq:distribroot}.\\
$  \mathrm{Red}, \mathrm{Red}^{\redl}$ & algorithm to add red leaves to a black tree, see Definition \ref{def:reds}.\\
$ \mathbf{Blob}$ & function transforming a blob tree by blowing-up each blob, see Definition \ref{def:blob}.\\
$ \mathrm{T}$ & $p$-Galton--Watson tree.\\
$ \mathfrak{T}$ & $p$-Galton--Watson tree with modified offspring distribution at the root.\\
$ \bT_{n}$ & tree of law $ \mathbb{P}^{n}= \mathbb{P}_{0}^{n}$ associated with $ \mathcal{S}_{n}$. \\
$ \bT_{\infty}$ & local limit seen from a leaf, of law $ \mathbb{P}_{0}^{\infty}$.\\
$ \bT_{\infty}^{*}$ & local limit seen from a vertex, see Corollary \ref{cor:aldous}.\\
$ \mathbf{C}_{\btau}, \mathbf{Z}_{ \btau}, \mathbf{R}_{\btau}$ & contour, label and leaf-counting process in the contour exploration.\\
\end{tabular}

\medskip 

\noindent \textbf{Scaling limit construction}. \ \\
\noindent \begin{tabular}{ll}
$( \mathbf{e},Z)$ & Brownian excursion and the head of the Brownian snake driven by $ \mathbf{e}$.\\
$  \mathbf{m}_{\infty}$ & Brownian sphere with metric $ D^{*}$.\\
$ \mathrm{Dev}(X)$ & development of the surface $X$ obtained by cutting along the spine.\\
$ \mathfrak{c}_{1}, ... , \mathfrak{c}_{n+1}$ & canonical horocycle neighborhoods of length $1$.\\
$  \mathcal{S}_{n}^{\circ}$ &  $ = \mathcal{S}_{n} \backslash \{ \mathfrak{c}_{1}, ... , \mathfrak{c}_{n+1}\}$.\\
$ D^{(n)}$ & distance between closest canonical horocycles along the sped-up contour.\\
$ D^{\circ}$ & distance in the $Z$-tree.\\
$  \mathfrak{D}$ & sub sequential scaling limits of $D^{(n)}$.\end{tabular}

\part{Tree encoding of  punctured spheres}
In this first part, we develop the tree encoding of plane hyperbolic surfaces with punctures via labeled plane binary trees by building upon the Penner--Epstein--Bowditch construction. Via successive transformations we describe the push forward of the Weil--Petersson law on $ \mathcal{M}_{0,n}$ to the space of labeled binary trees. In passing, this enables us to perform asymptotic enumeration of the Weil--Petersson volumes $ V_{0,n}$ which will be needed in the next part. 

\section{Background on hyperbolic surfaces} \label{sec:WP}
We start with some background on hyperbolic surfaces and Weil--Petersson volumes, we refer to the beautiful survey \cite{wright2020tour} or \cite{BuserBook,penner2012decorated} for details and references.
Let $g\geq 0$ and $n\geq 1$ (and $n\geq 3$ if $g=0$) and fix $S$ a genus-$g$ orientable surface with $n$ points labeled $\{1,2,... , n\}$ removed, viewed as a smooth two-dimensional manifold.
The Teichm\"uller space of $S$, denoted by $\teich_{g,n}$, consists of pairs $ \mathcal{X}=(X,f)$ where $X$ is a hyperbolic surface and $f : S \to X$ an orientation preserving diffeomorphism, and where two pairs $(X,f)$ and $(Y,g)$ are considered equivalent if there exists an isometry $h : X\to Y$ such that $g^{-1} \circ h \circ f : S \to S$ is homotopic to the identity.
If instead we only require the homeomorphism $g^{-1} \circ h \circ f$ to preserve the orientation and each of the punctures of $S$, then we obtain the moduli space $\mathcal{M}_{g,n}$ of $S$. 
They are related by the action of the mapping class group $\mathrm{MCG}_{g,n}$ of $S$, which is the group of orientation and puncture preserving diffeomorphisms of $S$ modulo those that are homotopic to the identity.
Then
\begin{equation*}
  \mathcal{M}_{g,n} = \teich_{g,n}\, /\, \mathrm{MCG}_{g,n}.
\end{equation*}
This representation of $\mathcal{M}_{g,n}$ is often convenient because the Teichm\"uller space has a much simpler topological structure. To illustrate this let us present two classical parameterizations of Teichm\"uller space:
\paragraph{Pants decomposition and Fenchel--Nielsen coordinates.} We recall the Fenchel--Nielsen coordinates that determine a homeomorphism $\teich_{g,n} \to  \R_+^{3g-3+n} \times \R^{3g-3+n}$ as follows.
Let $\pants$ be a pants decomposition of $S$, i.e.\ a set of (homotopy class of) disjoint simple closed curves drawn on $S$ such that the complement is a collection of spheres with three holes or punctures.
Then to each point $ \mathcal{X} = (X,f) \in \teich_{g,n}$ one may associate positive real numbers $((\ell_\gamma)_{\gamma\in \smallpants},(\tau_\gamma)_{\gamma\in\smallpants}) \in \R_+^{3g-3+n} \times \R^{3g-3+n}$, where $\ell_\gamma$ is the hyperbolic length of the unique shortest geodesic in $X$ homotopic to $f\circ\gamma$ and $\tau_\gamma$ a twist parameter describing how the hyperbolic pairs of pants in $X$ are glued\footnote{Defining the twist parameters needs more curves crossing the pants decomposition $\smallpants$, see \cite[Chapter 6]{BuserBook}.}, see Figure \ref{fig:pants}. 

 \begin{figure}[!h]
  \begin{center}
  \includegraphics[width=9cm]{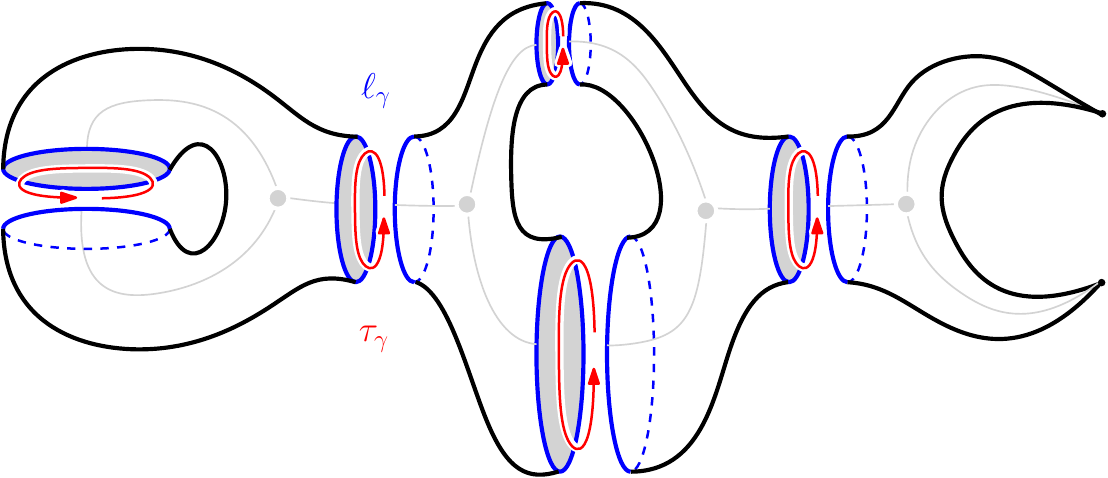}
  \caption{Illustration of a pants decomposition which gives rise to a coordinate system for $\teich_{2,2}$. \label{fig:pants}}
  \end{center}
  \end{figure}
The Teichm\"uller space $\teich_{g,n}$ comes with a natural symplectic structure, a closed and non-degenerate $2$-form $\omega$ called the Weil--Petersson symplectic form.
In Fenchel--Nielsen coordinates it is given simply by
\begin{equation*}
  \omega_{\mathrm{WP}} = \frac{1}{2}\sum_{\gamma\in\smallpants} \rmd\ell_\gamma \wedge \rmd \tau_\gamma.
\end{equation*}
It is invariant under $\mathrm{MCG}_{g,n}$ and therefore descends to a symplectic structure on $\mathcal{M}_{g,n}$. 
Since $\teich_{g,n}$ and $\mathcal{M}_{g,n}$ are of dimension $6g-6+2n$ the Weil--Petersson symplectic form induces a measure on $\teich_{g,n}$ and $\mathcal{M}_{g,n}$ given by 
\begin{equation*}
  \mathrm{WP} = \frac{\omega_{\mathrm{WP}}^{3g-3+n}}{(3g-3+n)!}.
\end{equation*}
The total measure of $\mathcal{M}_{g,n}$ is finite and known as the Weil--Petersson volume $V_{g,n} = \mathrm{WP}(\mathcal{M}_{g,n})$.
In terms of Fenchel--Nielsen coordinates on Teichm\"uller space it follows from the formulas above that $\mathrm{WP}$ is nothing but $2^{3-3g-n}$ times the standard Lebesgue measure on $\R_+^{3g-3+n} \times \R^{3g-3+n}$.
Hence, up to the power of two, the Weil--Petersson volume $V_{g,n}$ is also the Lebesgue volume of a fundamental domain for the action of $\mathrm{MCG}_{g,n}$ on $\teich_{g,n}$.
Unfortunately most mapping class elements correspond to complicated transformations in Fenchel--Nielsen coordinates\footnote{Recall that $\mathrm{MCG}_{g,n}$ is generated by Dehn twists: opening up the surface along a simple closed geodesic $\gamma$ and twisting a full turn before regluing.
If $\gamma\in \smallpants$ this corresponds to a simple transformation $\tau_\gamma \mapsto \tau_\gamma \pm \ell_\gamma$ but one also needs to take into account Dehn twists on geodesics $\gamma$ that do not belong to the pants decomposition.}. {A way out is to adapt the pant decomposition $\pants$ to the surface, see our forthcoming work \cite{BBCP25+} where we reinterpret Mirzakhani's famous recursion this way. We will use a similar idea in this work but with another parametrization of  Teichm\"uller space due to Penner \cite{penner1987decorated}.}
\section{Bowditch--Epstein--Penner \& Schaeffer}
 \label{sec:bowditch-epstein-penner}

\paragraph{Gluing of ideal triangles and Penner's lambda-lengths.} {In the following, it is mandatory to consider  surfaces with punctures. For $g\geq 0, n \geq 1$ with $3g-3+n \geq 0$ we consider the \textbf{decorated Teichm\"uller space}  $\widetilde{\teich}_{g,n}$ which is the trivial bundle over $ \teich_{n,g}$ with fiber $ \mathbb{R}_{>0}^{n}$. Geometrically, the space $\widetilde{\teich}_{g,n}$ can be interpreted as the space of hyperbolic metrics on $S$ together with a choice of horocycle on each puncture:  Each horocycle is uniquely specified by its hyperbolic length\footnote{Rigorously, this is defined by  taking the universal cover of the surface minus its punctures: the horocycle is then a projection of standard horocycle around one of the countably many preimages of a puncture. Horocycles are immersed closed curves in general but not necessarily embedded, see \cite[Section 2.1]{penner2012decorated}. However, by Lemma  \cite[Lemma 2.3]{penner2012decorated}, horocycles of length less than one are simple disjoint curves on the surface} $ \epsilon \in \mathbb{R}_{>0}$ implying that $\widetilde{\teich}_{g,n}$ is indeed homeomorphic to $\teich_{g,n} \times \R_{>0}^n$, see \cite[Chapter 2]{penner2012decorated}. In the following if $(X,f,(\epsilon_i)_{i=1}^n) \in \widetilde{\teich}_{g,n}$ we shall use the coherent short-hand notation 
$$ \widetilde{ \mathcal{X}} = (X,f,(\epsilon_i)_{i=1}^n), \quad \widetilde{X} = (X,(\epsilon_i)_{i=1}^n), \mbox{ and } \quad { \mathcal{X}} = (X,f) \in \teich_{g,n}.$$}

Fix an  ideal triangulation $\idealtrig$ of $S$, that is a homotopy class of non-crossing curves in $S$ starting and ending at a puncture such that the connected components   are all triangles.  Penner introduced in  \cite[Chapter 2]{penner2012decorated} a bijection $\widetilde{\teich}_{g,n} \to \R_{>0}^{6g-6+3n}$ via the introduction of \textbf{$\lambda$-length} coordinates.
Let $\widetilde{\mathcal{X}} \in \widetilde{\teich}_{g,n}$ be a hyperbolic metric decorated with horocycles $c_1,\ldots,c_n$.
For each curve $\gamma\in\idealtrig$ one may find a unique complete geodesic in $X$ homotopic to $f\circ\gamma$ connecting, say, punctures $i$ and $j$.
Then Penner's lambda-length is defined by   \begin{eqnarray} \label{eq:lambda-lengths} \lambda_\gamma = \sqrt{2 \mathrm{e}^{\ell_\gamma}}  \end{eqnarray} where $\ell_\gamma$ is the signed hyperbolic distance along this geodesic between the horocycles $c_i$ and $c_j$ (such that $\ell_\gamma < 0$ if the horocycles intersect). See Figure \ref{fig:horocycles} for an illustration.

\begin{figure}[!h]
 \begin{center}
 \includegraphics[width=12cm]{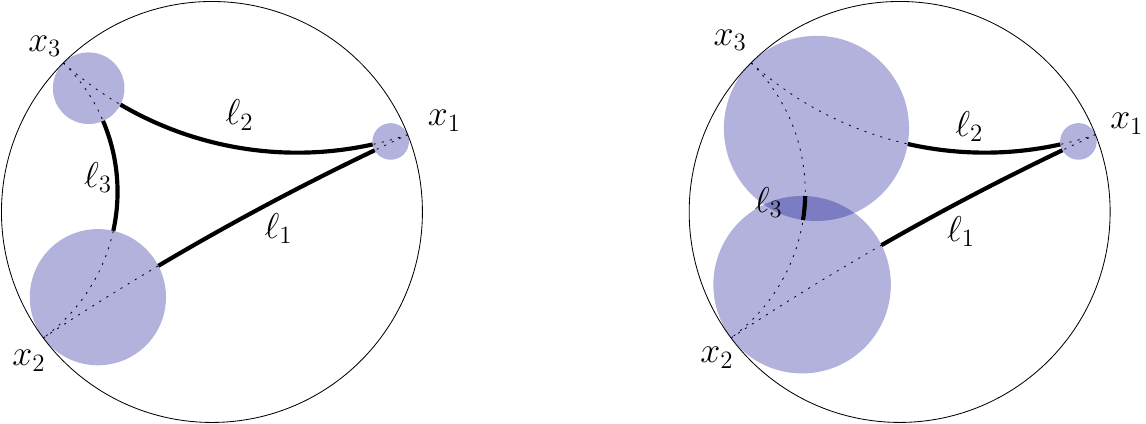}
 \caption{Illustration of Penner's lambda-lengths: given an ideal triangle decorated in $ \mathbb{D}$ with horocycles (represented by the blue circles), the lambda lengths are the exponential of the signed distances $\ell_i$ between the horocycles along  the geodesics linking the ideal points. In particular, on the right-hand side $\ell_{3}$ is negative. In fact, by \cite[Corollary 4.8, Chapter 1]{penner2012decorated}, ordered triples of lambda lengths give a parametrization of M\"obius-orbits of ordered triples of horocycles with distinct centers, so that we can glue together decorated ideal triangles along a combinatorial triangulation with prescribed lambda-lengths. \label{fig:horocycles}}
 \end{center}
 \end{figure}
Penner has shown \cite[Theorem 2.5, Chapter 2]{penner2012decorated} that $\mathcal{X}$ is uniquely determined by $(\lambda_{\gamma})_{\gamma\in \smallidealtrig}$ and that any tuple $(\lambda_{\gamma})_{\gamma\in \smallidealtrig} \in \R_{>0}^{6g-6+3n}$ is obtained in this way. The construction of $  \widetilde{\mathcal{X}} = ( X,f, ( \epsilon_{i})_{i=1}^{n})$ from $ (\idealtrig, (\lambda_{\gamma} : \gamma \in \idealtrig))$ is clear: just glue decorated hyperbolic triangles along the combinatorics of $\idealtrig$ to get $X$ (one can always glue decorated triangles along sides with compatible lambda-lengths), the marking $f$ being then produced by the embedding $\idealtrig \subset S$. The horocycle lengths $ ( \epsilon_{i})_{i=1}^{n}$ are finally recovered from the lambda-lengths: 
\begin{equation}
  \epsilon_i = \sum \frac{\lambda_{\gamma_c}}{\lambda_{\gamma_a}\lambda_{\gamma_b}},
\end{equation} 
where the sum is over all corners at the $i$th puncture and $\gamma_a$, $\gamma_b$ are the adjacent sides of the triangle containing the corner and $\gamma_c$ the opposite side, see  \cite[Lemma 4.9, Chapter 1]{penner2012decorated}. In particular, the horocycle lengths are scaled by $\delta$ if all lambda-lengths are scaled by $\delta^{-1}>0$. Importantly for us, using the canonical projection $\varphi : \widetilde{\teich}_{g,n} \to {\teich}_{g,n}$, the Weil--Petersson symplectic form $\omega_{_{\mathrm{WP}}}$ on $\teich_{g,n}$ pulls back to the $2$-form $\tilde{\omega}_{_\mathrm{WP}} = \varphi^* \omega_{_{\mathrm{WP}}}$ on $\widetilde{\teich}_{g,n}$ given explicitly by
\begin{align}   \label{eq:WPpenner}
\tilde{\omega}_{_\mathrm{WP}} = -2 \sum_{(i,j)}\rmd\log\lambda_{\gamma_i}\wedge\rmd\log\lambda_{\gamma_j}
\end{align}
where the sum is over all pairs $i\neq j$ such that $\gamma_i$ preceeds $\gamma_j$ in $\idealtrig$ in clockwise order around a puncture, see \cite[Theorem 3.1]{penner1987decorated}.
The Weil--Petersson volume form  on $\widetilde{\teich}_{g,n}$ may then be computed via $\varphi^* \mathrm{WP} = \tilde{\omega}_{_{\mathrm{WP}}}^{3g-3+n} / (3g-3+n)!$. As announced above, we shall now adapt the choice of ideal triangulation $\idealtrig$ to the isometry class of the surface. Notice that there are infinitely many rooted\footnote{i.e. with one oriented ideal arc distinguished} ideal triangulations of $S$, but the quotient by the action of $ \mathrm{MCG}_{g,n}$ is precisely the finite set of $  \trigspace_{g,n}$ of combinatorial rooted triangulations of genus $g$ with $n$ punctures. In the following, we shall denote embedded triangulations with the sign $\idealtrig$, whereas the underlying combinatorial map will be denoted by $\Delta$. 

\paragraph{Bowditch--Epstein spine construction.}
To any decorated surface $\widetilde{\mathcal{X}} = (X,f,(\epsilon_i)_{i=1}^n) \in \widetilde{\teich}_{g,n}$ one may associate a homotopy class of ideal arcs $\idealtrig(\widetilde{\mathcal{X}})$ in $S$. Those arcs will be furthermore independent under homogeneous rescaling of the horocycle lengths $(\epsilon_i)_{i=1}^n$, in particular we will always suppose that $ \epsilon_{i} \leq 1$ so that the horocycles are disjoint simple closed curves on the surface. Besides, the combinatorics of $\idealtrig(\widetilde{\mathcal{X}})$ will only depend upon $\widetilde{{X}}$ and not upon the marking $f$. In other words, we will have the compatibility $$\idealtrig (X,f,(\epsilon_i)_{i=1}^n) = f^{-1}\circ g \big(\idealtrig(X,g,(\epsilon_i)_{i=1}^n)\big).$$
 
Let $x$ be a point in $X$ and assume that $(\epsilon_i)_{i=1}^n$ is small enough that none of the horocycles surrounds $x$.
In this setting we can meaningfully identify the shortest geodesics starting at $x$ and ending at any of the horocycles\footnote{in fact, those geodesics are targeting the associated punctures. This is easily seen in the universal cover of $X$ minus its punctures. In particular, they are coherent as we scale the horocycles to points and the decorations $\epsilon_i$ can just be thought of as "delays". The shrinking of horocycles lengths thus corresponds to shifting the delays altogether.}.
The number of these is independent under homogeneous rescaling of the horocycle lengths.
The \textbf{spine} $\Sigma(\widetilde{X})\subset X$ is the collection of points that possess two or more such shortest geodesics. 
See Figure \ref{fig:spine} for an illustration and \cite{bowditch1988natural} for details. 
\begin{figure}[!h]
 \begin{center}
 \includegraphics[width=12cm]{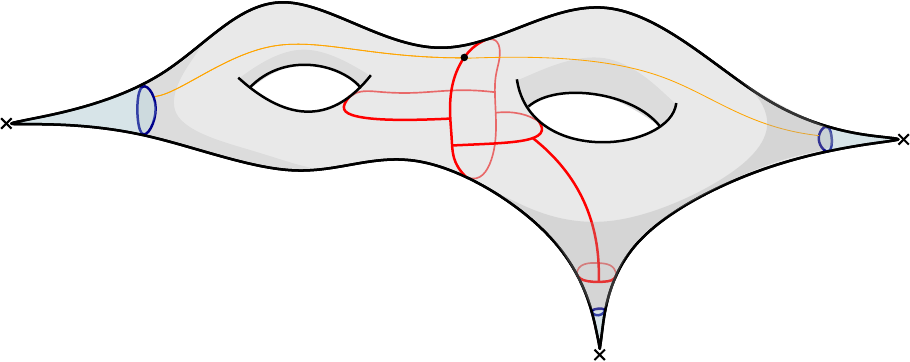}
 \caption{Construction of the spine (in red) as the loci of all points having more than one geodesic (in orange) going to the horocycles (in blue).}
 \end{center}
 \end{figure}

According to  \cite[Lemma 2.2.1]{bowditch1988natural} the spine $\Sigma(\widetilde{X})$ consists of the union of a finite number of geodesic segments that meet at vertices of degree at least three (corresponding to the points in $X$ that have at least three shortest geodesics).
When all $ \epsilon_i$ are positive, the complement $X \setminus \Sigma(\widetilde{X})$ is a disjoint union of $n$ once-punctured disks, one for each puncture of $X$.
Hence, to each geodesic segment in $\Sigma( \widetilde{X})$ we may naturally assign a dual ideal arc starting at a puncture on one side and ending at possibly another puncture on the other side. This produces a collection of non-crossing geodesics arcs on $X$ that we denote by $\idealtrig(\widetilde{X})$. Pushing through the marking $f^{-1}$, this defines $\idealtrig(\widetilde{\mathcal{X}})$ the desired homotopy class of disjoint ideal arcs in $S$. This construction is clearly compatible with the action of $\mathrm{MCG}_{g,n}$ and considering equivalence classes by the action of the the mapping class group 
provides a cell decomposition of $\widetilde{\teich}_{g,n}$ into finitely many cells that is invariant under $\mathrm{MCG}_{g,n}$.
This is quite convenient because $\idealtrig(\widetilde{\teich}_{g,n}) / \mathrm{MCG}_{g,n}$ is canonically identified with the set of (combinatorial) genus-$g$ maps with $n$ labeled vertices, with the top-dimensional cells corresponding to triangulations \cite[p154]{penner2012decorated}, which have $6g-6+3n$ edges by Euler's formula. In particular, the only cells of positive Weil--Petersson measure are those associated to triangulations and we shall always restrict to this case in the following. In other words, the combinatorial triangulation $\Delta( \widetilde{\mathcal{X}})$ associated with $\idealtrig( \widetilde{\mathcal{X}})$ only depends on the  isometry class of $X$ with the horocycle decorations. This construction is equivalent to the convex hull construction in Minkovski space of Penner \cite[Theorem 1.6 in Chapter 4]{penner2012decorated}.

The above construction is particularly well-adapted to Penner's parametrization: Fix $\trig$ an unrooted combinatorial triangulation, then $ \idealtrig(\widetilde{X}) \equiv \trig$ can be checked using Penner's lambda-lengths. Indeed, by \cite[Chapter 1, Definition 4.18 and Chapter 4, Lemma 1.7]{penner2012decorated} it holds if and only if  $\lambda_{1},\lambda_{2},\lambda_{3}$ satisfy strict triangle inequalities whenever the edges $e_1,e_2,e_3$ bound a triangle in $\triangle$, and $\lambda_{1},\ldots,\lambda_{5}$ satisfy the \textbf{Delaunay condition} 
\begin{equation} \label{eq:delaunay}
  \frac{\lambda_{2}^2+\lambda_{3}^2 - \lambda_{1}^2}{\lambda_{2}\lambda_{3}} + \frac{\lambda_{4}^2+\lambda_{5}^2 - \lambda_{1}^2}{\lambda_{4}\lambda_{5}} > 0
\end{equation}
whenever $e_1,e_2,e_3$ and $e_1,e_4,e_5$ bound two neighboring triangles, see Section \ref{sec:euclidean} for a geometric interpretation. Knowledge of $\triangle$ and the lambda-lengths $(\lambda_e : e \in \mathrm{Edges}(\triangle))$ is thus sufficient to glue the appropriate ideal decorated triangles along the combinatorics of the map to recover the surface up to isometry.

With all these ingredients at hand, one can describe the law of a random surface $ \mathcal{S}_{g,n} \in \mathcal{M}_{g,n}$ sampled according to the Weil--Petersson measure in terms of its underlying combinatorial map $ \trig $ and its lambda-lengths $( \lambda_{i} : 1 \leq i \leq 6g-6 + 3n)$. For this, we first fix arbitrarily the horocycle lengths $( \epsilon_{i})_{i=1}^{n}$ and consider the inclusion $ \teich_{g,n} \to \widetilde{ \teich}_{g,n}$ which amounts to decorate a surface with horocycles of fixed lengths $\epsilon_1,\ldots,\epsilon_n > 0$. Modding out by the action of $ \mathrm{MCG}_{g,n}$, and taking the $(3g-3+n)$th power of \eqref{eq:WPpenner}, we deduce that the Weil--Petersson measure on $ \mathcal{M}_{g,n}$ is, in terms of the $\log$-lambda-lengths, just a multiple of the measure 
  \begin{eqnarray} \label{eq:abstractlaw} \propto \sum_{|I|=n}\sum_{ \trig \in \trigspace_{g,n}}  \mathrm{cst}_{\trig,I} \prod_{i}  \mathrm{d} \log \lambda_{1} \wedge \cdots \wedge \widehat{ \mathrm{d} \log \lambda_I} \wedge \cdots \wedge  \mathrm{d}\lambda_{6g-6 + 3n}  \mathbf{1}_{ \left\{\begin{subarray}{c}  \mathrm{Triangle \ \& \ Delaunay \ \& }\\
\mathrm{Horocycle \ lengths \ conditions}
\end{subarray} \right\} },  \end{eqnarray} where the indicator encapsulates the above constrains and the constant $ \mathrm{cst}_{\trig,I} \in \mathbb{R}$ a priori depends on the combinatorics on the subset $I$ of indices removed from the product and of the map $\trig$, see \cite[Proposition 3.1]{grushevsky2001explicit}.
 This strategy has been implemented in  \cite{Penner} to compute the first few Weil--Petersson volumes $V_{1,1}$ and $V_{1,2}$ by picking $ \epsilon_{1} = \epsilon_{2} = ... = \epsilon_{n}=1$, but a systematic treatment appears difficult.
 
 \begin{figure}[!h]
  \centering
  \includegraphics[width=\linewidth]{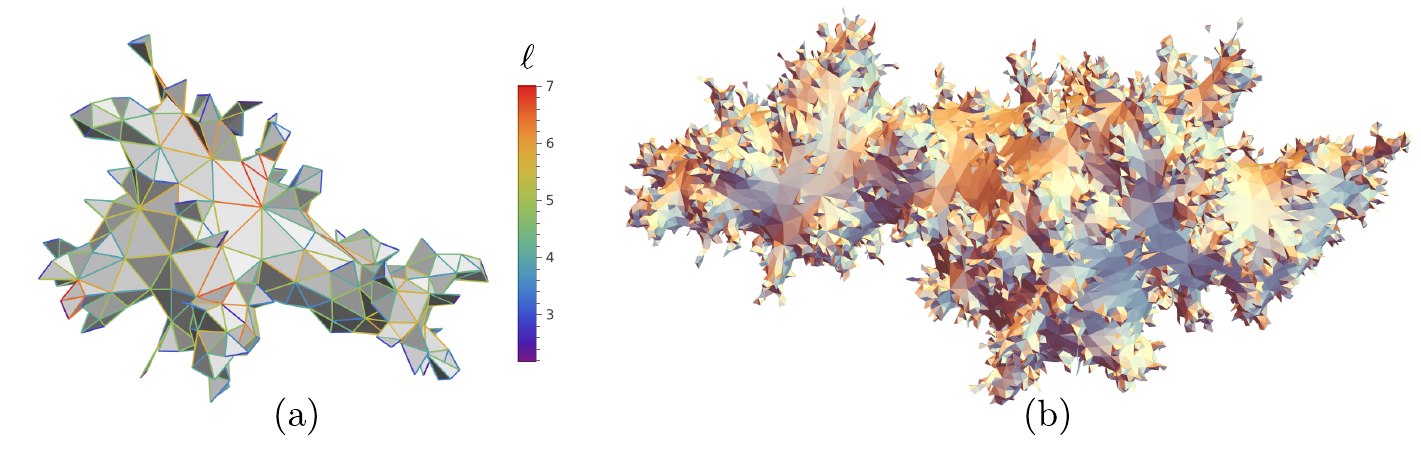}
  \captionsetup{width=\textwidth}
  \caption{ (a) This figure illustrates the combinatorial triangulation $\Delta(\mathcal{S}_n)$ associated to the ideal Delaunay triangulation $\idealtrig(\mathcal{S}_{n})$ where $\mathcal{S}_{n} \in \mathcal{M}_{0,n+1}$ with $n=332$ is sampled according to the Weil--Petersson measure and when all horocycles are taken of unit length. Each curve $\gamma$ is colored by its signed hyperbolic length $\ell_\gamma$. (b) Similar but for $n=11202$ and curves are not drawn here.\label{fig:simulationsB}}
\end{figure}

\paragraph{The planar case.} In the following, we shall focus on the case $g=0, n \geq 3$ and choose to use the horocycle lengths  $\epsilon_1 = \cdots = \epsilon_{n-1} = 0$ and $\epsilon_n = 1$. The $n$th puncture carrying the non-trivial horocycle is called the \textbf{origin puncture} below. Since this is a degenerate case (horocycle lengths a priori need to be positive in the above discussion), let us elaborate a little. Fix $\mathcal{X} = (X,f) \in \mathcal{T}_{0,n}$ and consider the horocycles associated with the lengths $ \epsilon^{\varepsilon}_n =1$ and $ \epsilon_1^{\varepsilon} = \dots = \epsilon^{\varepsilon}_{n-1} = \varepsilon>0$. Write then $ \widetilde{X}^{ \varepsilon} = ( X, ( \epsilon_{i}^{ \varepsilon})_{i=1}^{n})$ the decorated surface. Since  the surface $X$ deprived of its horocycle neighborhoods  is of finite diameter, the Epstein--Bowditch spine  $\Sigma( \widetilde{X}^{ \varepsilon})$ converges towards $\Sigma( \widetilde{X}^{ 0})$, which is the cut-locus of $X$ seen from the origin cusp which is  a tree by classical results \cite{myers1935connections}.

\begin{figure}[!h]
 \begin{center}
 \includegraphics[width=12cm]{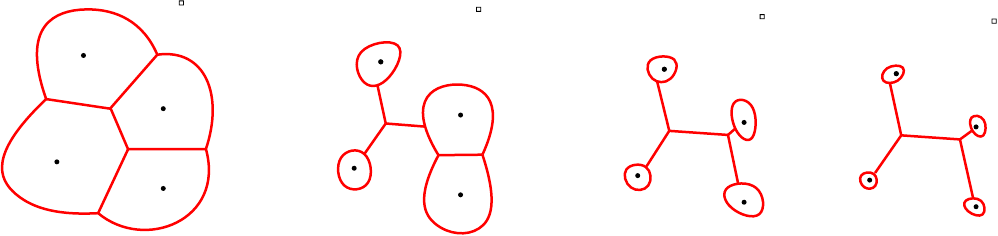}
 \caption{Illustration of the convergence of $\Sigma(  \widetilde{X}^{ \varepsilon})$ (in red on the figure) as $ \varepsilon \to 0$.  The origin puncture is the white box, whereas the horocycles lengths shrink around the other punctures (the black dots). \label{fig:looptree}}
 \end{center}
 \end{figure}

More precisely, the spine $\Sigma(  \widetilde{X}^{ \varepsilon})$ is eventually constant on $X$ deprived of (any) horocycle neighborhoods and is eventually  a ``loop-tree'', i.e. a  tree where each leaf is replaced by a loop, those loops surrounds the non-origin punctures, see Figure \ref{fig:looptree}. By the above discussion, this tree is furthermore cubic in generic situations, and we shall implicitly restrict to those. When $ \varepsilon \to 0$, the lambda-length coordinates associated to the loops of the spine then disappear, while the lambda-lengths associated to the other edges  are eventually constant;  together with the combinatorics of the tree, they suffice to reconstruct the surface $X$ up to isometry. Indeed, notice that the components in $\idealtrig( \widetilde{X}^{ 0})$ obtained by cutting $X$ along the geodesics dual to edges of the limiting spine are of two types: either a triangle folded on itself carrying a puncture (see left of Figure \ref{fig:triangles-type}), or an ideal triangle decorated with horocycles. We denote by $\tau_n$ the combinatorial plane tree structure associated with $\Sigma(\widetilde{X}^{ 0})$. The compatibility of the lambda-lengths enables then to glue all these pieces along the plane tree structure of $\tau_n$ and recover $X$ up to isometry.

 \begin{figure}[!h]
  \begin{center}
  \includegraphics[width=12cm]{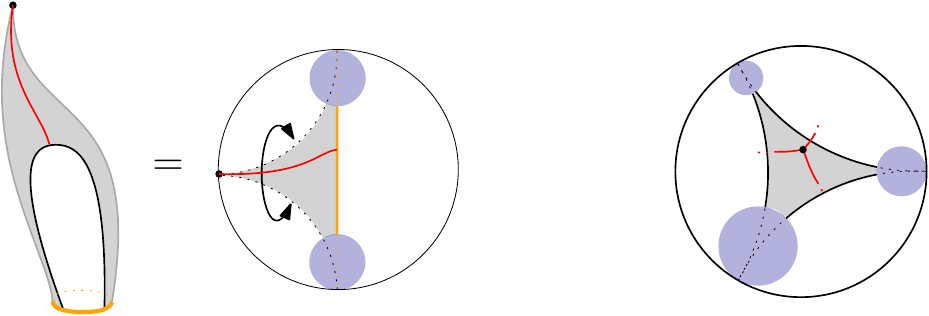}
  \caption{The two types of decorated triangles in $\idealtrig( \widetilde{X}^{0})$. Left: a triangle carrying a single puncture whose structure is prescribed by its unique lambda-length, or Right: a decorated ideal triangle with $3$ lambda-lengths. The spine is represented in red on these figures.\label{fig:triangles-type}}
  \end{center}
  \end{figure}

 In order to break symmetry, we shall furthermore distinguish a puncture of $X$ different from the origin puncture, and called the \textbf{root puncture} which we use to root  $\tau_n$.

\begin{figure}[!h]
 \begin{center}
 \includegraphics[width=3cm]{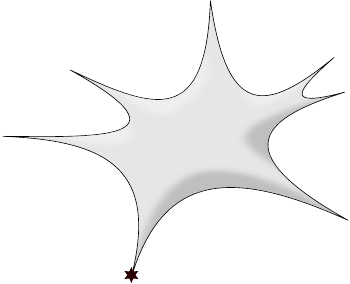}  \includegraphics[width=3cm]{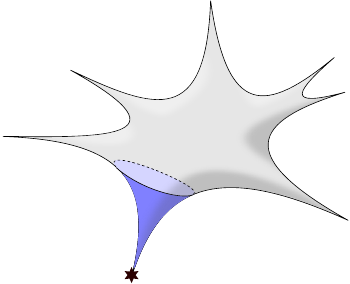} \includegraphics[width=3cm]{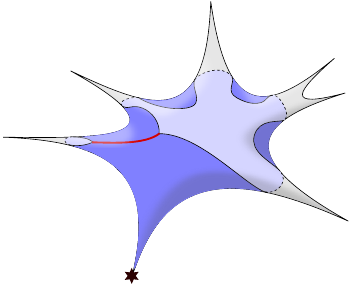}  \includegraphics[width=3cm]{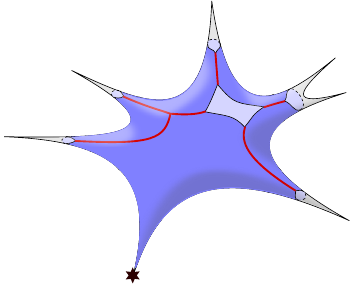} \includegraphics[width=3cm]{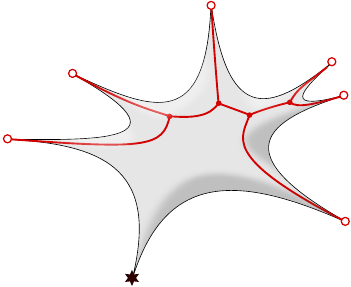}
 \caption{Illustration of the spine construction in the case of a plane punctured hyperbolic surface and where distances are measured from a single puncture, called the \textbf{origin puncture} (i.e. $ \epsilon_{1} = ...  = \epsilon_{n-1}=0$ and $\epsilon_{n}=1$). The blue regions are the balls of growing radius seen from the distinguished horocycle, and the spine (in red) is its cut-locus. \label{fig:spine}}
 \end{center}
 \end{figure}

 \begin{figure}[!h]
 \begin{center}
 \includegraphics[width=13cm]{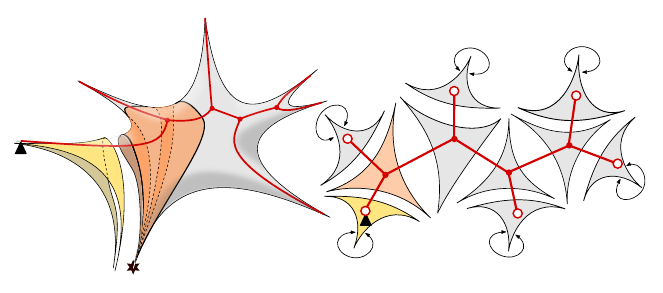}
 \caption{Left: Illustration of the spine in the case of a plane punctured surface with trivial horocycles (i.e. length $0$) except at the origin puncture (bottom of the figure). Another puncture (the left triangle) is distinguished in order to root the combinatorial tree obtained form the spine. Right: the surface is reconstructed from the spine and the remaining lambda-lengths.}
 \end{center}
 \end{figure}
 
 \medskip

 We will perform a change of variable so that in the new coordinate system, Eq. \eqref{eq:abstractlaw} becomes much simpler.  
Rather than using the coordinates $( \lambda_{e} : e \in \mathrm{Edges}( \tau_{n}))$, we shall use angles $\theta_{c}$ indexed by the corners of $\tau_{n}$. A geometric interpretation will be given in Section \ref{sec:euclidean}, but for the moment let us take it as a change of variable. Recall that the decorated ideal triangles associated to inner vertices of $\tau_n$ have lambda-lengths $\lambda_{1}, \lambda_{2}, \lambda_{3}$ satisfying the triangle condition. Consider  the \textit{Euclidean} triangle of side lengths $\lambda_{1}, \lambda_{2}, \lambda_{3}$. The three bisectors then meet at the center of the circumscribed circle and define three angles $\theta_{1} +\theta_{2} + \theta_{3} = 2\pi$ which by the law of sines satisfy 
 \begin{eqnarray} \label{eq:length-angle} \frac{ \lambda_{1}}{\sin \theta_{1}} =\frac{ \lambda_{2}}{\sin \theta_{2}}=\frac{ \lambda_{3}}{\sin \theta_{3}}.  \end{eqnarray}
 
 \begin{figure}[!h]
  \begin{center}
  \includegraphics[width=12cm]{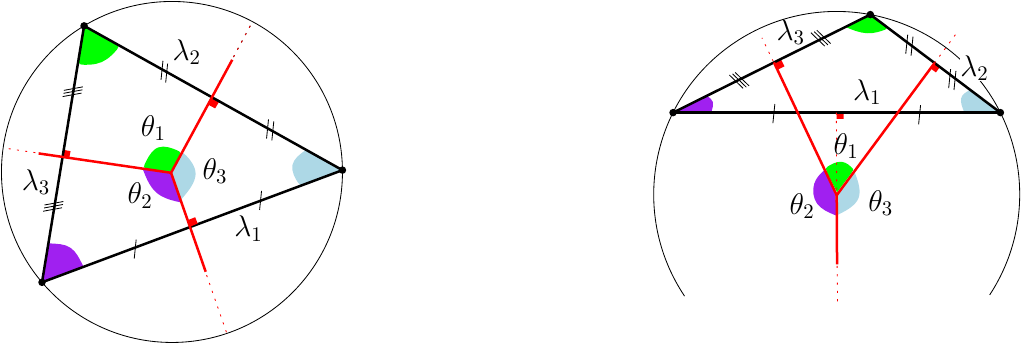}
  \caption{Illustration of the change of variables from lambda-lengths to angles using Euclidean geometry, see Proposition \ref{prop:hypeucl} for  details. \label{fig:length-angle}}
  \end{center}
  \end{figure}
Using this relation at each inner vertex of $\tau_{n}$ gives a labeling $\theta_{c}$ of the corners of $\tau_{n}$, with the convention that the angle associated to leaf-corners are taken to be equal to $\pi$. Notice that the angles are equivalent to the data of all lambda-lengths, up to a global multiplicative scaling;  this scaling being fixed by the condition $ \epsilon_{n}=1$. 
The Delaunay constraints \eqref{eq:delaunay} are then easily expressed in terms of the angles $ (\theta)$, and they amount to asking that the Euclidean triangles built from the lambda-lengths satisfy the associated Euclidean Delaunay condition: the sum of opposite angles to an internal edge be larger than $\pi$. See Figure\ \ref{fig:conditionangles} for a diagrammatic illustration and Section \ref{sec:euclidean} for more details.
 \begin{figure}[!h]
  \begin{center}
  \includegraphics[width=9cm]{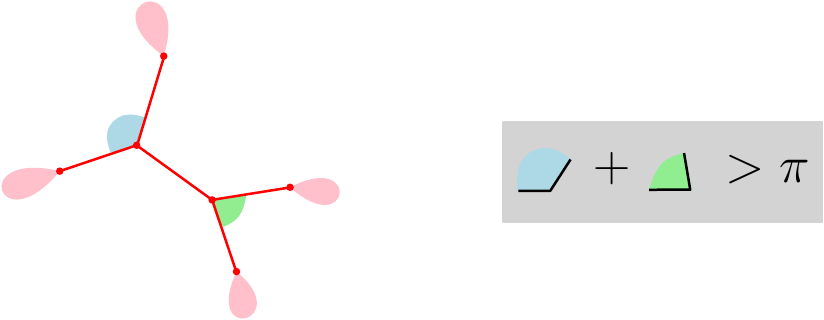}
  \caption{The angle condition equivalent to the Delaunay condition. \label{fig:conditionangles}}
  \end{center}
  \end{figure}

Let us sum-up: We consider $ X \in \mathcal{M}_{0,n+1}$ to be (the isometry class of) a generic surface with punctures labeled $1,\ldots,n+1$. We take the punctures with label $n+1$ and $1$ to be the origin and root puncture, respectively. We can then encode $X$ by a plane binary tree $\tau_{n}$ whose internal corners carry angles $\theta_{c} \in (0,\pi)$, whereas leaf corners carry by convention the angle $\theta = \pi$, satisfying the condition depicted in Figure \ref{fig:conditionangles}. Those are called \textbf{allowed angle configurations}. The leaves of the tree $\tau_{n}$ also carry the labels $1,\ldots,n$ of their corresponding punctures, where the root puncture is the first one. The crux is that in this new system of coordinates, the Weil--Petersson measure is simply a multiple of the Lebesgue measure:

\begin{theorem}[A labeled-tree encoding of WP-punctured spheres] \label{thm:treeencoding} The push-forward of the Weil--Petersson measure on $ \mathcal{M}_{0,n+1}$ via the above encoding is equal to $2^{n-2}$ times the Lebesgue measure on binary trees with $n$ leaves with an allowed angle configuration.
\end{theorem}

\begin{proof} 
Let us fix $n \geq 2$ and positive horocycle lengths $ (\epsilon_{i})_{i=1}^{n+1}$. The spine of a generic decorated surface $ \widetilde{X} = (X, ( \epsilon)_{i=1}^{n+2})$ with $X \in \mathcal{M}_{0,n+1}$ is then combinatorially a plane trivalent map with $n+1$ faces, $2n-2$ vertices, and whose $3n-3$ edges are labeled by the associated lambda-lengths (given by \eqref{eq:lambda-lengths} where the $\ell_{i}$ are the lengths of the dual geodesics between horocycles). Applying the change of coordinates \eqref{eq:length-angle} in each triangle, one can label the corners of the triangulation with angles $\theta_{c}$ (beware those angles do not come  necessarily from an embedding of the graph in the Euclidean plane) and if one further imposes the horocycle length conditions, this is a proper change of variables\footnote{At first sight there are more angle-variables than lambda-lengths: $3(2n-2)$ corners for $3n-3$ edges. However, the angle-variables have more constraints: there is $1$ constraint per vertex (the sum is equal to $2\pi$) and $1$ constraint per face (except for one) of the trivalent map, which gives $3(2n-2)-(2n-2)-n = 3n-4$ degrees of freedom. If we add the scaling variable, we thus have $3n-3$ degrees of freedom as for the lambda-lengths.}. Recalling the expression of the Weil--Petersson symplectic form in Penner's lambda-length coordinates 
\begin{align*}
\tilde{\omega}_{_\mathrm{WP}} = -2 \sum_{\triangle} \left( \rmd\log\lambda_1\wedge\rmd\log\lambda_2 + \rmd\log\lambda_2\wedge\rmd\log\lambda_3 + \rmd\log\lambda_3\wedge\rmd\log\lambda_1\right),
\end{align*}
where the sum is over all triangles $\triangle$ and $\lambda_1$, $\lambda_2$, $\lambda_3$ are the lambda-lengths of its sides in counter-clockwise order.

Note that a triangle $\triangle$ that is glued to itself does not actually contribute to $\tilde{\omega}_{_\mathrm{WP}}$, since two of its lambda-lengths are identical so that the three terms in the summand cancel.   
Otherwise, using that $\lambda_{1} = \lambda_{3} \sin \theta_{1} / \sin(2\pi-\theta_{2}-\theta_{1})$ and $\lambda_{2} = \lambda_{3} \sin \theta_{2} / \sin(2\pi-\theta_{2}-\theta_{1})$ we find
\begin{align*}
& \rmd\log\lambda_{1}\wedge\rmd\log\lambda_{2} + \rmd\log\lambda_{2}\wedge\rmd\log\lambda_{3} + \rmd\log\lambda_{3}\wedge\rmd\log\lambda_{1}%\big) 
\\
&= \rmd \log \frac{\sin \theta_{1}}{\sin(2\pi-\theta_{2}-\theta_{1})} \wedge \rmd \log \frac{\sin \theta_{2}}{\sin(2\pi-\theta_{2}-\theta_{1})}\\
&= -\big(\cot \theta_{2} \cot \theta_{1} + \cot \theta_{1} \cot(2\pi-\theta_{2}-\theta_{1}) + \cot(2\pi-\theta_{2}-\theta_{1})\cot \theta_{2}\big) \rmd\theta_{2} \wedge \rmd \theta_{1}\\
&= - \rmd\theta_{2} \wedge \rmd \theta_{1}.
\end{align*}
Let us now focus on the case where the combinatorial triangulation is dual to a looptree as in the right of Figure\ \ref{fig:looptree}.
Then only the $n-2$ triangles dual to the internal vertices contribute to $\tilde{\omega}_{_\mathrm{WP}}$.
Choosing a pair of angles at each internal vertex, we thus find that the normalized $(n-2)$th exterior power $\tilde{\omega}^{n-2} / (n-2)!$ is (up to an overall sign) $2^{n-2}$ times the $(2n-4)$-dimensional Lebesgue measure on these angles. 
Reincorporating the triangle and Delaunay conditions (which are inequalities), this amounts to the Lebesgue measure on allowed angle configurations on the combinatorial tree. Taking $\varepsilon\to 0$, the probability under WP that $\Delta( \widetilde{\mathcal{X}})$ is not dual to a looptree tends to $0$, and we deduce that the law of the tree with allowed angles $( \tau_{n} : ( \theta_{c} : c \in \mathrm{Corners}( \tau_{n}))$ is the normalized Lebesgue measure. 
\end{proof}

\begin{figure}[!h]
	\begin{center}
		\includegraphics[width=15cm]{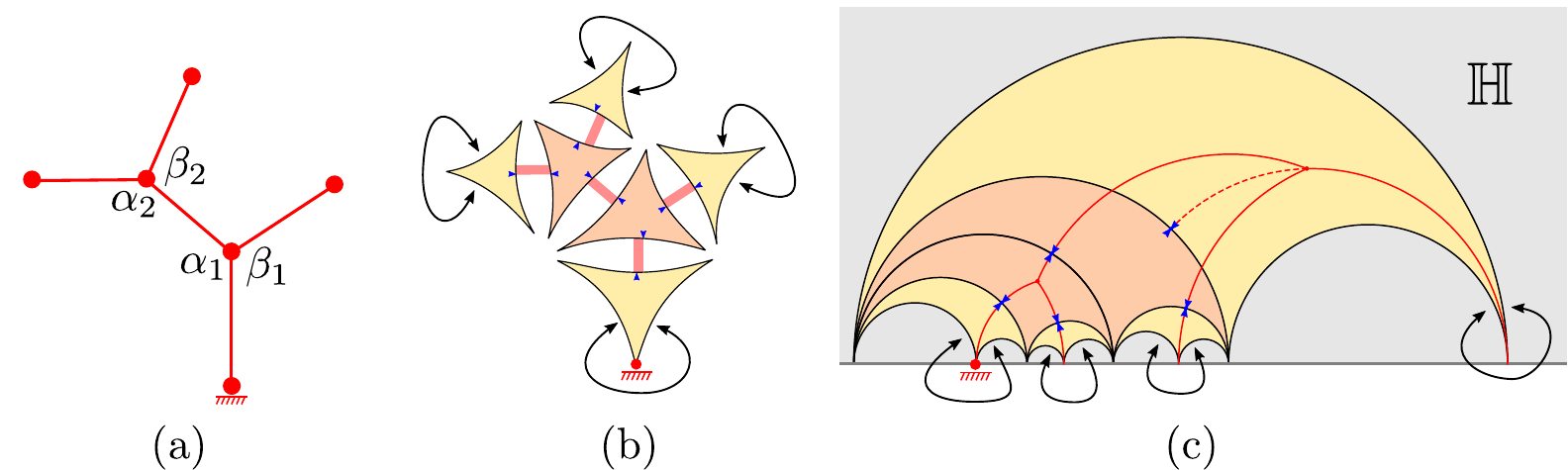}
		\caption{Summary of the construction of a random WP surface with $n+1$ punctures as described by Theorem \ref{thm:treeencoding}. (a) Sample a rooted binary tree $ \mathbf{T}_{n}$ together with an allowed angle assignment according to the normalized Lebesgue measure. (b) To each vertex of ${T}_{n}$ we assign an ideal triangle, and these are glued according to some shear which are deduced from the $\lambda$-lengths themselves recovered from the angles (up to scaling). (c) The gluing yields an explicit fundamental polygon in the hyperbolic plane (represented as the upper-half plane here) and performing the resulting side identification creates the desired punctured surface. \label{fig:fundamentaldomain}}
	\end{center}
\end{figure}

Let us remark that these results crucially depend on us working with surfaces of genus $0$. 
The situation in higher genus is more complicated, because in the case $\epsilon_1 = \cdots = \epsilon_{n-1} = 0$ and $\epsilon_n = 1$ the spine is not a tree but a genus-$g$ map with one face, see Figure~\ref{fig:highergenusspine}.
The main difficulty one then encounters is that there are $2g$ extra non-linear conditions on the angles in order for them to correspond to valid lambda-lengths: for each cycle of the spine the ratios of lambda-lengths of consecutive edges crossed by the cycle are expressed in terms of angles (via the law of sines) and these ratios must multiply to $1$.
Therefore the set of allowed angle configurations is a lot more complicated than the polytope we encounter for genus $0$.
 \begin{figure}[!h]
  \begin{center}
  \includegraphics[width=7cm]{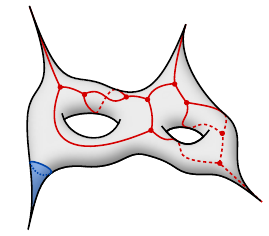}
  \caption{The spine of a genus-$2$ surface with a single horocycle is a genus-$g$ map with one face.\label{fig:highergenusspine}}
  \end{center}
  \end{figure}

\section{Geometric constructions and enumeration}
In this section we give an equivalent geometric description of a label tree with an allowed angle configuration. Surprisingly, this description is based on Euclidean geometry where the lengths are Penner's lambda-length. This sheds light on the triangle and Delaunay condition \eqref{eq:delaunay} and on the change of variables $ \ell \to \theta$ which was key in the formulation of Theorem \ref{thm:treeencoding}. This will also be particularly useful to decompose our labeled trees into blocks and perform  enumeration.
\subsection{Relation to Euclidean geometry and labeled trees}

\label{sec:euclidean}

\paragraph{Space of equidistant triangles.} Recall that a decorated ideal triangle is triplet of points $x_1,x_2,x_3\in \partial \mathbb{H}$ together with three horocycle neighborhoods of boundary lengths $  \epsilon_1, \epsilon_2, \epsilon_3 >0$ defining three signed lengths $\ell_1, \ell_2, \ell_3 \in \mathbb{R}$. By \cite[Lemma 4.12, Chapter I]{penner2012decorated}, there exists a unique point equidistant to the three horocycles if and only if the lambda-lengths $\lambda_i = \sqrt{2 \mathrm{e}^{\ell_i}}$ satisfy the strict triangle inequalities. Those decorated  triangles will be called \textbf{equidistant} in the following.

 \begin{figure}[!h]
  \begin{center}
\includegraphics[width=10cm]{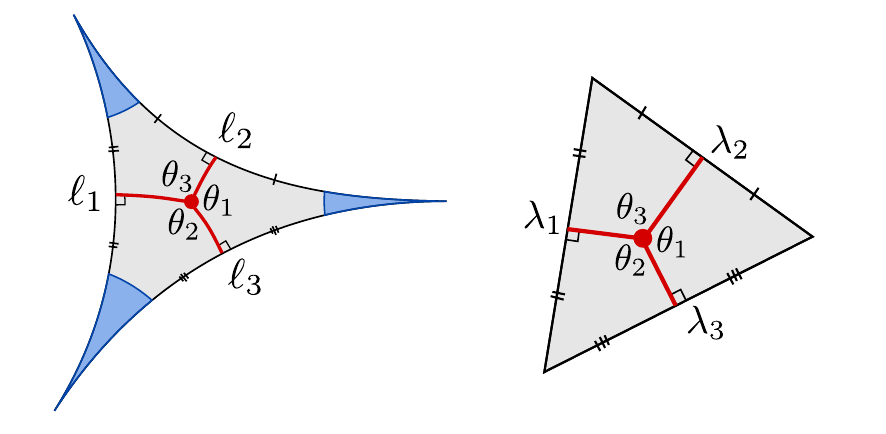}
  \caption{ Left: a decorated hyperbolic triangle having a point equidistant to the three horocycles. Right: Its Euclidean version.   \label{fig:hypeucl1}}
  \end{center}
  \end{figure}

More precisely, one can then  draw the three hyperbolic bisectors (in red on Figure \ref{fig:hypeucl1}) with respect to the horocycles and they meet at a point $\zeta \in \mathbb{H}$ which may or may not belong to the interior of the hyperbolic triangle and define three angles $\theta_1, \theta_2, \theta_3 \in (0,\pi)$ as in Figure \ref{fig:hypeucl1}. Consider now the Euclidean analog of the hyperbolic triangle with side lengths $\lambda_1, \lambda_2, \lambda_3$ (this is possible since this triplet obeys the triangle inequalities), one can similarly consider the common point of the three bisectors which turns out to define \textbf{the same} angles $\theta_1, \theta_2, \theta_3$. 
Such a relation between hyperbolic and Euclidean triangles was already at the root of Penner's work \cite{penner1987decorated}, and has featured in the study of discrete conformal mappings, see \cite[Sec.~5]{Bobenko15} and \cite{Springborn20}.
We summarize here the relations we will use later.

\begin{proposition}[Hyperbolic--Euclidean-Angles change of variables]  \label{prop:hypeucl} The above construction is a bijection between
\begin{itemize}
\item triplet of signed lengths $\ell_1, \ell_2, \ell_3 \in \mathbb{R}$, up to additive constant, of a decorated hyperbolic triangle admitting a point equidistant to the horocycles,
\item  triplet of lambda-lengths $\lambda_1, \lambda_2, \lambda_3 \in \mathbb{R}_{>0}$, up to multiplicative constant, satisfying the triangle inequalities,
\item angles $\theta_1,\theta_2,\theta_3 \in (0,\pi)$ which sum to $2 \pi$.
\end{itemize}
The relations between those triplets is given by 
    \begin{eqnarray} \label{eq:relationanglesbis} \lambda_{i} = \sqrt{2 \mathrm{e}^{\ell_{i}}}, \quad \mbox{ and } \quad \frac{   \sqrt{\mathrm{e}^{\ell_{1}}}}{\sin \theta_{1}} =\frac{ \sqrt{\mathrm{e}^{\ell_{2}}}}{\sin \theta_{2}}=\frac{ \sqrt{\mathrm{e}^{\ell_{1}}}}{\sin \theta_{3}}.  \end{eqnarray}
The angles $\theta_1, \theta_2, \theta_3$ formed by the bisectors in the hyperbolic and Euclidean setting are the same.

Furthermore, in the hyperbolic triangle, for $i=1,2,3$ the signed hyperbolic length of the bisector that connects the $i$th side with $\zeta$ (taken to be negative if $\zeta$ is outside the triangle on this side) is $\operatorname{arctanh} \cos(\pi-\theta_i)$.
\end{proposition}
\begin{proof} The two changes of variables is justified in \cite[Lemma 4.12, Chapter I]{penner2012decorated}. Let us now prove that the angles $\theta$ are the same in the Euclidean and hyperbolic setting. To see this, let us view the decorated ideal triangle in the hemisphere model, in which geodesics are intersections of vertical planes with the hemisphere and horocycles are intersections of spheres that touch the equatorial plane at the equator. 
If the point $\zeta$ is positioned at the pole of the hemisphere then  the horocycles correspond to a triple of spheres of identical (Euclidean) size.
Projecting the ideal triangle onto the plane gives the corresponding Euclidean triangle inscribed in the equator.
The perpendicular bisectors of the ideal triangle project to the bisectors of the Euclidean triangle.
Since the tangent plane at the pole is parallel to the equatorial plane, the angles at which the bisectors meet are identical for both triangles.
\begin{figure}[!h]
  \begin{center}
\includegraphics[width=8cm]{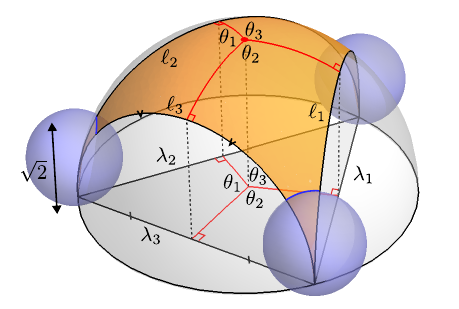}
  \caption{From hyperbolic to Euclidean world. \label{fig:euclhyptriangle}}
  \end{center}
  \end{figure}

The statement about the signed hyperbolic length of the $i$th bisector follows from standard hyperbolic trigonometry. If $\theta_i \geq \pi/2$, half of the $i$th side up to the midpoint (it does not matter which half) and the $i$th bisector are the sides of a right-angled hyperbolic triangle with an ideal vertex and remaining angle equal to $\pi-\theta_i$. Indeed, the ideal triangle is partitioned into six of these right-angled triangles with angles $\phi_1,\phi_1,\phi_2,\phi_2,\phi_3,\phi_3$ satisfying $\theta_3 = \phi_1 + \phi_2,  \cdots$, so $\phi_3 =  \frac{1}{2}(\theta_{1}+\theta_{2}-\theta_{3})= \pi - \theta_3$. Therefore the length of the bisector is $\operatorname{arctanh} \cos(\pi - \theta_i)$. The same holds for $\theta < \pi/2$, but now the right-angled hyperbolic triangle is on the outside of the ideal triangle and the remaining angle is $\theta_i$, so the signed length of the bisector is $-\operatorname{arctanh} \cos \theta_i$. 
\end{proof}

\paragraph{Tree of Delaunay triangles.} Recall furthermore from the previous section that on top of having the equidistant property, the ideal decorated triangles appearing in the decomposition of a (generic) plane punctured surface satisfy the Delaunay condition. It turns out that if one interprets the lambda-lengths as the lengths of Euclidean triangles, then the Delaunay condition is equivalent to the standard Euclidean Delaunay condition: Recall that two Euclidean triangles $ABC$ and $BDC$ sharing an edge $BC$ satisfy the  Delaunay condition if $D$ (resp.~A) does not belong to the circumscribed disk of $ABC$ (resp. $BCD$) and vice-versa. If we denote by $\lambda_{1}, \lambda_{2}, ... , \lambda_{5}$ the lengths of those triangles (the common edge being $\lambda_{1}$) this is indeed equivalent to the condition \eqref{eq:delaunay}, see Figure \ref{fig:delaunay-eucli} and its caption. 

\begin{figure}[!h]
 \begin{center}
 \includegraphics[width=17cm]{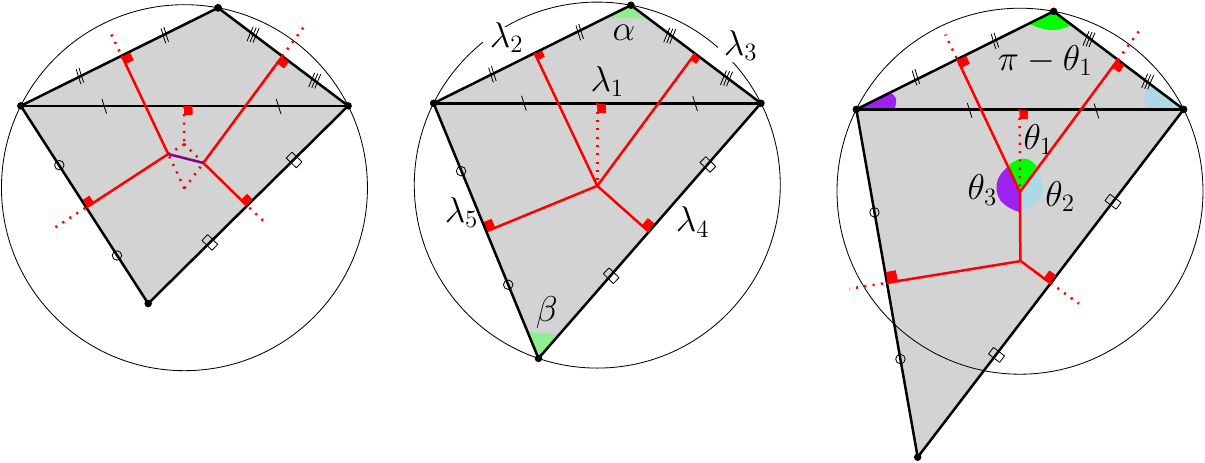}
 \caption{(Middle) Two Euclidean triangles in the equality case of Delaunay. In such case, the angles $\alpha$ and $\beta$ must sum-up to $\pi$ by the angle interception theorem. Since by the law of cosine (a.k.a.~Al-Kashi theorem) we have $\frac{\lambda_{2}^{2}+\lambda_{3}^{2}-\lambda_{1}^{2}}{\lambda_{2} \lambda_{3}} = 2\cos(\alpha)$ and $\frac{\lambda_{4}^{2}+\lambda_{5}^{2}-\lambda_{1}^{2}}{\lambda_{4} \lambda_{5}} = 2\cos(\beta)$, the former condition is equivalent to the equality case in \eqref{eq:delaunay}. (Left) When the neighboring triangles do not satisfy the Delaunay condition, the Vorono\"\i  ~diagram is not dual to the combinatorial triangulation. (Right) When the neighboring triangles do satisfy the Delaunay condition, the Vorono\"\i  ~diagram is combinatorially dual to the triangulation, and  one can define three angles at the branch points, which by the sine law satisfy \eqref{eq:relationanglesbis}.
\label{fig:delaunay-eucli}}
 \end{center}
 \end{figure}

Consider then a binary tree $\tau$ with an allowed angle configuration. If one considers the Euclidean triangles associated to the inner vertices of $\tau$ and glue them together, we  obtain a \textbf{tree of Euclidean triangles} satisfying the Delaunay conditions, which we will consider up to multiplicative scaling. 
Let us emphasize that we may not be able to draw such tree of triangles in the Euclidean plane without intersection.  The spine or Voronoi diagram of such a tree of Euclidean triangles is combinatorially equivalent to $\tau$ (in red on Figure \ref{fig:delaunay-eucli}). Let us insist that $\tau$ is a plane (i.e.~there is a cyclic orientation around each vertex)  and binary (i.e. whose vertices have either degree $1$, the leaves, or $3$) tree. It is furthermore rooted at a distinguished leaf. We refer to \cite{Nev86} for formalism about plane trees which are seen as subset of the Ulam tree $ \mathbb{U} = \bigcup_{k \geq 0} ( \mathbb{Z}_{>0})^k$, in particular the \textbf{root vertex} of a plane tree is always denoted by the empty word $ \varnothing$. By Proposition \ref{prop:hypeucl}, any such tree harbors  three notions of labelings: 
 \begin{itemize}
 \item the labeling of its edges by the signed hyperbolic distance between punctures $$\ell_{e}, \quad e \in \mathrm{Edges}(\tau) ,$$
 \item  the labeling of its edges by the $\lambda$-lengths which satisfy the triangle and Delaunay condition $$\lambda_{e} = \sqrt{2 \mathrm{e}^{{\ell_{e}}}},\quad  e \in \mathrm{Edges}(\tau),$$
 \item the allowed angle configuration 
 $$\theta_{c}, \quad c \in \mathrm{Corners}(\tau).$$
\end{itemize}
The $\theta$-labeling is equivalent to the $\lambda$-labeling up to multiplicative constant or to the $\ell$-labeling up to additive constant.  In the following, we shall fix the last two by declaring that $\lambda_{e_{\varnothing}} = 1$ and $\ell_{e_{\varnothing}}=0$ where $e_{\varnothing}$ is the edge incident to the root leaf $\varnothing$ in $\tau$. Such trees will always be represented in red in our illustrative figures, and we shall sometimes called them ``{\color{red}red}'' trees to differentiate them from other type of labeled trees used in these pages. We will then speak of {\color{red} red labeled tree} to refer to a rooted planar binary tree equipped with one of its equivalent labeling $\theta,\lambda$ or  $\ell$  (the context making clear which one we use). To ease the presentation, we use the bold symbols $ \boldsymbol{ \tau}, \bT...$ for the labeled trees and the standard ones $ \tau, {T},...$ for the unlabeled versions.
We shall also write $|\tau|$ for the number of leaves of $\tau$ and $\# \tau$ for its number of vertices. In the coming sections, we shall shift from the labeled tree to its tree of Euclidean triangles back and forth without further notice.

\subsection{Blobs} \label{sec:blobdef}
Let $\boldsymbol{\tau}$ be a binary tree with an allowed angle assignment. In its associated tree of Euclidean triangles, consider  two neighboring Euclidean triangles with lengths $\lambda_{1}, ... , \lambda_{5}$ sharing the edge of length $\lambda_{1}$. Notice that if both triangles contain the center of their circumscribed circles, that is if 
$$ \frac{\lambda_{2}^{2} + \lambda_{3}^{2}-\lambda_{1}^{2}}{\lambda_{2}\lambda_{3}} \geq 0, \quad \mbox{ and }\quad \frac{\lambda_{4}^{2} + \lambda_{5}^{2}-\lambda_{1}^{2}}{\lambda_{4}\lambda_{5}} \geq 0,$$ then they obviously satisfy the Delaunay condition \eqref{eq:delaunay}. Geometrically, under this condition the Voronoi diagram in fact intersect the edge of length $ \lambda_{1}$. We shall call such edges \textbf{cut-edges}. Those cut-edges decompose a tree of Euclidean triangles into so-called \textbf{blobs}. 

\begin{figure}[!h]
 \begin{center}
 \includegraphics[width=14cm]{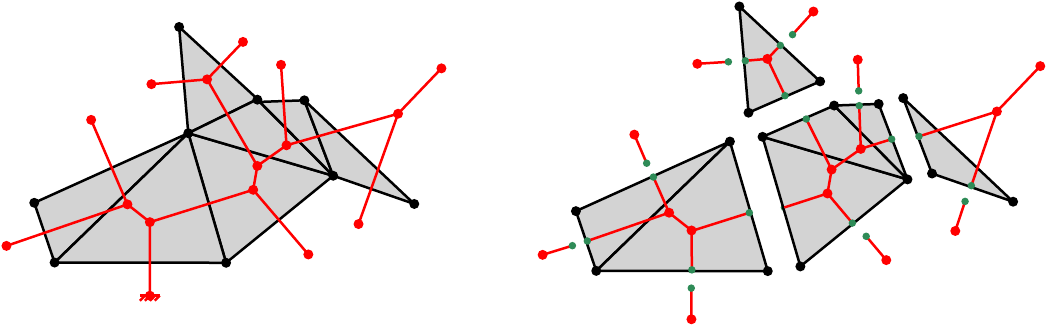}
 \caption{ \label{fig:decomposition} Decomposition of a tree of Euclidean triangles into blobs by cutting at cut-edges (in green on the figure). Notice the trivial blobs made of a single red leaves and remark that a blob may contain a read leaf (e.g.~the right-most blob in the right-hand figure). \label{fig:blob-1}}
 \end{center}
 \end{figure}
 
A blob is thus a tree of Euclidean triangles (satisfying the
Delaunay conditions) that does not contain cut-edges, see Figure \ref{fig:blob-1}. By convention, we also consider the trivial blob made of a single red leaf. 
The red edges inside blobs have a canonical  orientation pointing from the center of the circumscribed circle towards the middle of its corresponding edge.
Since there are no cut-edges those orientations are coherent and are all emanating from a single triangle (in dark gray on Figure \ref{fig:blob-2}). In particular, a blob  can have at most one red leaf attached to it. See Figure \ref{fig:blob-2} for two  examples of blobs. 
The \textbf{degree} of a blob is its number of cut-edges.

  \begin{figure}[!h]
  \begin{center}
  \includegraphics[width=13cm]{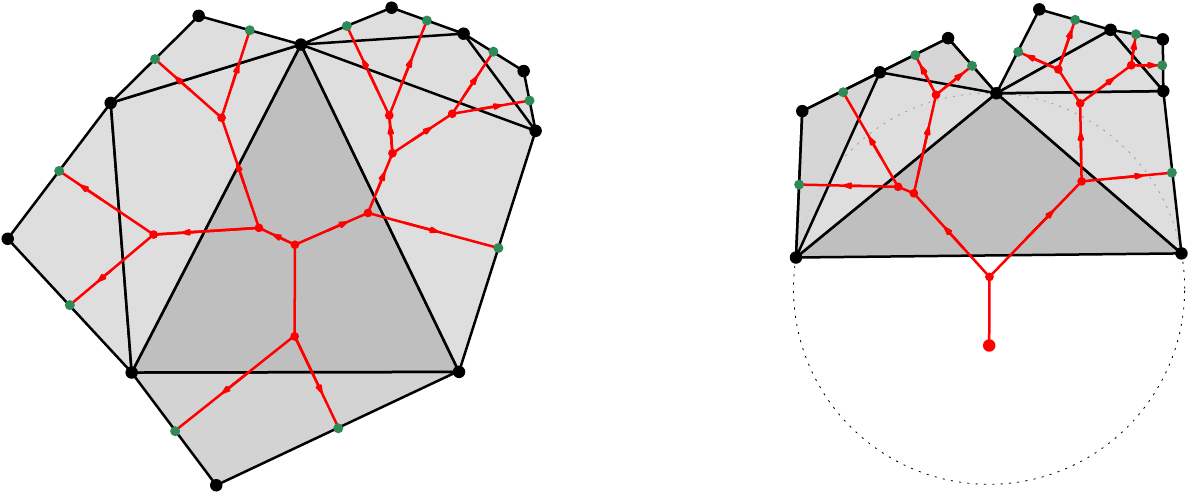}
  \caption{Two blobs. On the left, a blob of degree $11$, and on the right, a blob of degree $9$ carries a red leaf. The degree (number of cut-edges) is the number of green dots.  \label{fig:blob-2}}
  \end{center}
  \end{figure}

\subsection{Enumeration} \label{sec:enumeration}
In this section we derive a few enumerative consequences on the Weil--Petersson volumes from Theorem \ref{thm:treeencoding}. We shall also enumerate the volumes of the blobs defined in the preceding section.
Since Bessel functions will make a prominent appearance, we start by reminding us of a few of their properties, while referring the reader to \cite{watson1922treatise} for details.

\paragraph{Bessel functions}
The Bessel functions of the first kind, denoted by $J_{n}(x)$ are the solutions that are non-singular at the origin of the differential equation
  \begin{eqnarray} \label{eq:diffbessel} x^{2} \frac{ \mathrm{d}^{2} y}{ \mathrm{d}x^{2}} + x \frac{ \mathrm{d} y}{ \mathrm{d}x} + (x^{2}-n^{2})y =0.  \end{eqnarray} See Figure \ref{fig:bessel} for a plot of the first few functions $J_{0},J_{1}, J_{2}...$. 
\begin{figure}[!h]
 \begin{center}
 \includegraphics[width=10cm]{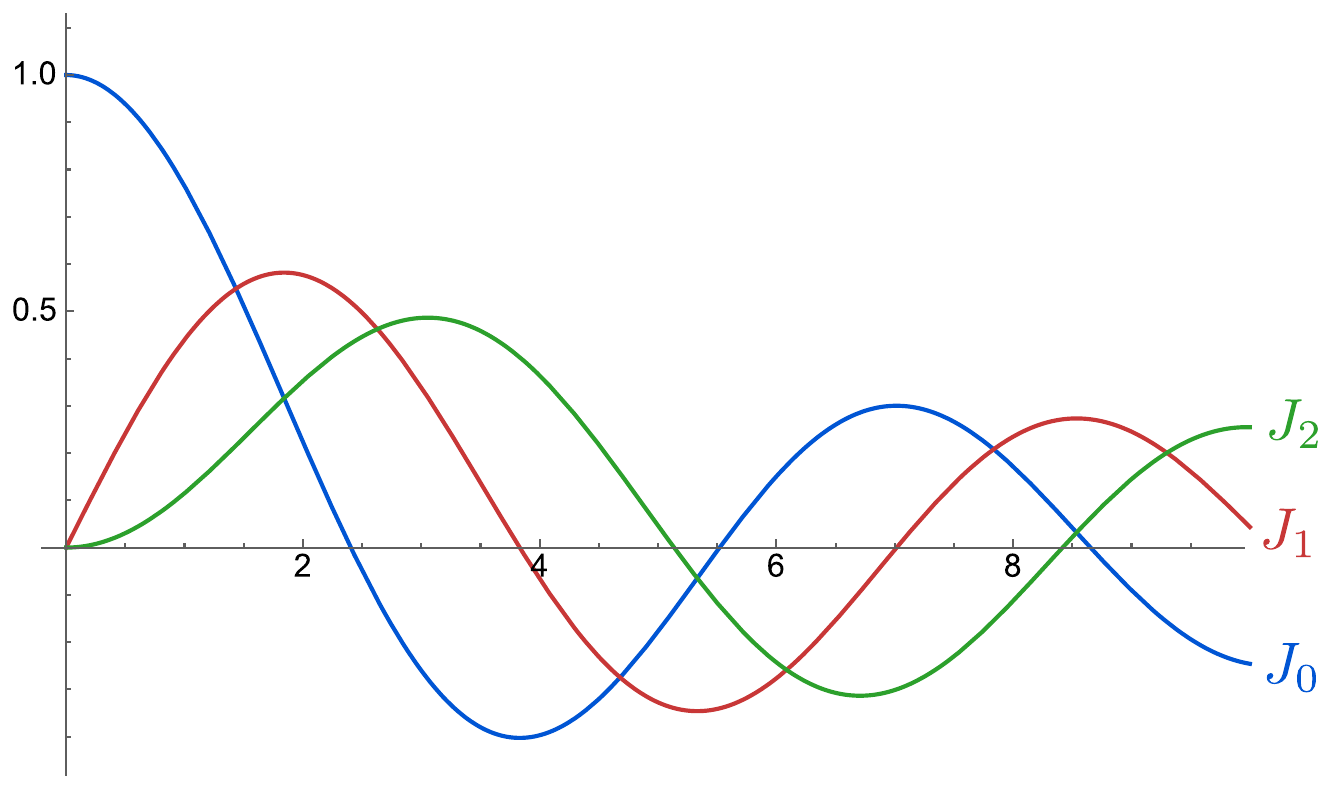}
 \caption{A plot of the first three Bessel functions of the first kind. \label{fig:bessel}}
 \end{center}
 \end{figure}

Using \eqref{eq:diffbessel} we see that $J_{0}'=-J_{1}, J_{0}''= \frac{1}{2}(J_{2}-J_{0})$. Combining those equations we deduce that if we define $S(y) = \frac{\sqrt{y}}{\pi} J_1(2\pi \sqrt{y})$ then we have    \begin{eqnarray} \label{eq:ZJ}   \frac{ \partial}{ \partial y} S(y) = \frac{ \partial}{ \partial y}\frac{\sqrt{y}}{\pi} J_{1}(2 \pi \sqrt{y}) = J_{0}( 2 \pi \sqrt{y}).  \end{eqnarray}
 Below we shall primarily use $J_{0}$ and $J_{1}$ which will appear in solutions to integral equations via their Laplace transform:
  \begin{eqnarray} \int_{0}^{\infty} \mathrm{d}x \, \mathrm{e}^{-tx} J_{0}(x) = \frac{1}{ \sqrt{1+t^{2}}}, \quad \mbox{ for }t >0. \label{eq:laplace}  \end{eqnarray}

\paragraph{Zograf's formula.}

Let $\mathsf{Bin}_n$ be the set of full (plane) binary trees with $n$ (unlabeled) leaves rooted on a leaf, i.e.\ the set of plane trees with $n-2$ cubic vertices and $n$ univalent vertices (including the root).
To any tree $ \tau \in\mathsf{Bin}_n$ we associate an open subset $A_{\tau} \subset \R^{2n-4}$ of allowed \textbf{angle assignments} as follows.
Fix an ordering $u_1, \ldots, u_{n-2}$ of the cubic vertices of $\tau$.  Then we let $(\alpha_1,\beta_1,\ldots,\alpha_{n-2},\beta_{n-2}) \in A_{\tau}$ if $\alpha_i,\beta_i\in (0,\pi)$ and $\alpha_i + \beta_i > \pi$ for all $i$ and 
\begin{align}
\begin{cases}
\alpha_i + \beta_i < \beta_j + \pi &\text{if }u_i\text{ is a left child of }u_j,\\
\alpha_i + \beta_i < \alpha_j + \pi &\text{if }u_i\text{ is a right child of }u_j.\\
\end{cases}\label{eq:treeangleconstraints}
\end{align}
This is clearly a translation of the constraint \eqref{fig:conditionangles} if we imagine that $\alpha_{i}, \beta_{i}$ are the angles associated to first and last corner of the vertex $i$ in the contour order, the third angle being prescribed by $2 \pi - \alpha_{i} - \beta_{i}$. Let $\mathcal{A}_n$ be the disjoint union of angle assignments when $\tau$ ranges over all rooted cubic plane trees,
\begin{align} \label{def:An}
\mathcal{A}_n = \bigsqcup_{\tau\in\mathsf{Bin}_n} A_{\tau},
\end{align}
 and let us equip $\mathcal{A}_n$ with the measure $\mu$ arising from the standard Lebesgue measure on $A_{\tau}$.
In case $n=2$, the set $\mathcal{A}_n$ consists of a unique element which by convention receives unit measure. 

Since the leaves of $\tau$ can be labeled $1,\ldots,n$ such that the root receives label $1$ in precisely $(n-1)!$ ways, Theorem~\ref{thm:treeencoding} relates the total measure $\mu(\mathcal{A}_n)$ to the Weil--Petersson volumes $V_{g,n} =  \mathrm{WP}( \mathcal{M}_{g,n})$ via 
\begin{align*}
  \mu(\mathcal{A}_n) = 2^{2-n} \frac{V_{0,n+1}}{(n-1)!}.
\end{align*}
When $g=0$, the Weil--Petersson volumes were computed by Zograf \cite{MR1234274}. Let us show how to recover his formula using Theorem~\ref{thm:treeencoding}. For this we introduce the generating series of WP volumes
\begin{align}\label{eq:treevolumeZ}
Z(x) := \sum_{n=2}^\infty \mu(\mathcal{A}_n) x^{n-1},
\end{align} 
and the more general formal power series
\begin{align} \label{def:Fxtheta}
F(x;\theta) &:= \sum_{n=2}^\infty x^{n-1} \sum_{\tau \in \mathsf{Bin}_n} \mu(A_{\theta,\tau}),\qquad \theta\in[0,\pi]\\
&= x + \tfrac12(\pi^2-\theta^2)x^2 + \cdots\nonumber,
\end{align}
where $A_{\theta,\tau} \subset A_{\tau}$ is the set of angle assignments such that $\alpha_1 + \beta_1 + \theta < 2\pi$, assuming $u_1$ is the first cubic vertex encountered from the root, see Figure \ref{fig:treeequation}.
If $\tau$ has no cubic vertices, i.e.\ when $n=2$, then we set $\mu(A_{\theta,\tau}) = \mu(A_{\tau}) = 1$ by convention.

\begin{theorem}[Zograf, \cite{MR1234274}]\label{thm:treevolume}  \label{thm:zograf} For $n \geq 3$ we have
\begin{align}
V_{0,n} &= \frac{(2\pi^2)^{n-3}}{(n-3)!}a_n, \qquad a_3=1,\nonumber\\
a_n &= \frac{1}{2}\sum_{k=1}^{n-3} \frac{k(n-k-2)}{n-1} \binom{n-4}{k-1}\binom{n}{k+1}a_{k+2}a_{n-k}\qquad \text{for }n\geq 4. \label{eq:zografrecursion}
\end{align}
Equivalently, its generating function $Z(x)$ satisfies
\begin{align}
Z(x) = \sum_{n=3}^\infty \frac{2^{3-n}}{(n-2)!} V_{0,n}\, x^{n-2}, \qquad \frac{\sqrt{Z(x)}}{\pi} J_1(2\pi \sqrt{Z(x)}) = x.\label{eq:zografgenfun}
\end{align}
More generally, we have 
\begin{eqnarray} \label{eq:Fxtheta} F(x;\theta) = \frac{\sqrt{Z(x)}}{\theta} J_1(2\theta \sqrt{Z(x)}).  \end{eqnarray}
\end{theorem}
\begin{proof}

We claim that $F(\theta) \equiv F(x;\theta)$ satisfies the formal integral equation
\begin{equation}\label{eq:Frecurrence}
F(\theta) = x + \int_0^\pi \!\rmd\alpha\int_0^\pi\!\rmd\beta\, F(\alpha)F(\beta)\, \ind_{\{\theta<\alpha+\beta<\pi\}}, \qquad \theta\in[0,\pi]. 
\end{equation}

\begin{figure}
			\centering
			\includegraphics[width=.9\linewidth]{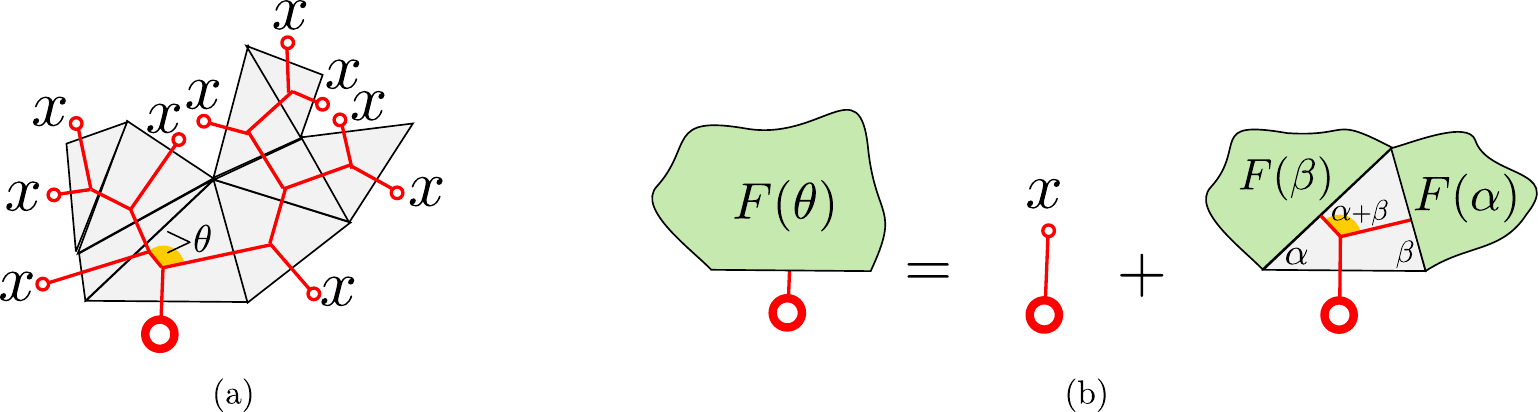}
			\caption{(a) Example of a tree contributing to $F(x;\theta)$ together with the tree of Euclidean triangles corresponding to the cubic vertices. (b) Pictorial representation of the recurrence equation \eqref{eq:Frecurrence}. \label{fig:treeequation} }
		\end{figure}
		
The first term is simply the contribution for $n=2$.
The second term corresponds to the Lebesgue integration over the angles $\alpha_1 = \pi - \alpha$ and $\beta_1 = \pi - \beta$ at the first cubic vertex (when $n\geq 3$), see Figure \ref{fig:treeequation}.
By definition of $A_{\theta,\tau}$ the variables $\alpha$ and $\beta$ are restricted to the region $\theta < \alpha + \beta < \pi$.
If $\alpha$ and $\beta$ are fixed and the left and right subtree at the first cubic vertex are denoted $\tau_{\text{L}}$ and $\tau_{\text{R}}$, then the conditions \eqref{eq:treeangleconstraints} imply that the remaining angles of $\tau$ lie in $A_{\beta,\tau_{\text{L}}} \times A_{\alpha,\tau_{\text{R}}}$ (up to reordering the cubic vertices).
Integrating over these angles and summing over $\tau_{\text{L}}$ and $\tau_{\text{R}}$ exactly gives \eqref{eq:Frecurrence}.

By differentiating \eqref{eq:Frecurrence} we find that $F(x;\theta)$ is also uniquely determined by
\begin{equation}\label{eq:Frecurrencederiv}
F'(\theta) = - \int_0^\theta \rmd\alpha\, F(\alpha)F(\theta-\alpha), \qquad \theta\in[0,\pi]
\end{equation}
and the boundary condition $F(x;\pi)=x$.
A solution to \eqref{eq:Frecurrencederiv} for all $\theta \geq 0$ is given by
\begin{equation}
F(x;\theta) = \frac{\sqrt{F(x;0)}}{\theta} J_1(2\theta \sqrt{F(x;0)}) \label{eq:Ftheta}
\end{equation}
for any positive power series $F(x;0)$.
This can be easily checked from the Laplace transform, see \eqref{eq:laplace}, 
\begin{equation}
f(y)=\int_0^\infty \rmd\theta\, \mathrm{e}^{-y\theta} \frac{\sqrt{F(x;0)}}{\theta} J_1(2\theta \sqrt{F(x;0)}) = \frac{-y + \sqrt{y^2+4F(x;0)}}{2},
\end{equation}
which solves the Laplace transformed version of \eqref{eq:Frecurrencederiv}, namely $y f(y) - F(x;0)  = - f(y)^2$.
Comparing the boundary condition $F(x;\pi)=x$ with \eqref{eq:zografgenfun} shows that $$F(x;0) = Z(x).$$
By construction $F(x;0) = \sum_{n=2}^\infty \mu(\mathcal{A}_n) x^{n-1}$, thus verifying \eqref{eq:treevolumeZ}. Using the fact that $J_{1}(x) \sim \frac{x}{2}$ near $0$ in \eqref{eq:Ftheta} completes the proof. The equivalence between Zograf's recursion \eqref{eq:zografrecursion} and  \eqref{eq:zografgenfun} is derived in the opening section of \cite{kaufmann1996higher}.
\end{proof}

\paragraph{Asymptotics.} By the analytic inversion theorem, the radius of convergence of $x \mapsto Z(x)$ or equivalently that of $x \mapsto F(x; \theta)$ is equal to the first critical point of the function 
\begin{align}\label{eq:Finvfunction}
  S(y) = \frac{\sqrt{y}}{\pi} J_1(2\pi \sqrt{y}),
\end{align}
which after recalling \eqref{eq:ZJ}, is equal to 
\begin{align} \label{eq:xc}
x_c = \frac{c_0}{2\pi^2} J_1(c_0) = 0.06324\ldots \mbox{ at which } Z(x_c) = \left(\frac{c_0}{2\pi}\right)^{2} =  0.14648\ldots,
\end{align}
where $c_0=2.40482\ldots$ is the first positive zero of $J_0$.  A standard singularity analysis yields:

\begin{lemma}\label{lem:asymptotics} For each $\theta \in [0,\pi)$, we have as $n \to \infty$
 \begin{align*}
 \frac{x_c^{n-1} [x^{n-1}] F(x;\theta)}{F(x_c,\theta)} \sim \frac{\theta}{\sqrt{2\pi^3}}\, \frac{J_0( \tfrac{\theta}{\pi}c_0)}{J_1( \tfrac{\theta}{\pi}c_0)}\, n^{-3/2}, \quad \mbox{ as } n \to \infty.
 \end{align*}
 \end{lemma}
\begin{proof} 
By the singular inversion theorem \cite[Thm.~VI.6]{Flajolet:analytic} applied to the aperiodic function $\phi : y \mapsto y / S(y)$, the solution $Z(x)$ to $y = x \phi(y)$ can be analytically extended to a neighborhood of the disk of radius $x_c$ around $0$ slit along $[x_c, \infty)$ with an expansion as $x\uparrow x_c$ given by
\begin{align*}
 Z(x) &= Z(x_c) - \frac{c_0}{\sqrt{2} \, \pi^2}\sqrt{1 - x/x_c} + o(\sqrt{x_c-x}).
\end{align*}
Since for every $\theta\in [0,\pi)$, the function $y \mapsto \tfrac{\sqrt{y}}{\theta} J_1(2\theta\sqrt{y})$ is entire and has non-vanishing derivative at $Z(x_c)$, the series $F(x;\theta)$ is analytic in the same domain with expansion
\begin{align*}
  F(x;\theta) &= F(x_c;\theta) - J_0(\tfrac{\theta}{\pi}c_0) \frac{c_0}{\sqrt{2} \, \pi^2}\sqrt{1 - x/x_c} + o(\sqrt{x_c-x}).
\end{align*}
The transfer theorem thus gives
\begin{align*}
x_c^{n-1} [x^{n-1}]F(x;\theta) &\sim -\frac{n^{-3/2}}{\Gamma(-1/2)} \frac{c_0}{\sqrt{2} \, \pi^2}J_0(\tfrac{\theta}{\pi}c_0)  = \frac{c_0}{2\sqrt{2} \, \pi^{5/2}} J_0(\tfrac{\theta}{\pi}c_0)n^{-3/2} \quad \text{as }n\to\infty.
\end{align*} 
Dividing by $F(x_c;\theta) = \frac{c_0}{2\pi \theta} J_1(\tfrac{\theta}{\pi} c_0)$ gives the stated asymptotic equivalence.
\end{proof}
 
 \paragraph{Enumeration of blobs.}  Let us now turn to enumeration of labeled trees associated with blobs, as described in the previous section. We let 
 \begin{align*}
  B_{x}(z) = x + \tfrac14 \pi^2 x z + \tfrac18 (\pi^2 + \tfrac58 \pi^4 x) z^2 + \cdots
 \end{align*}
  be the generating function of blobs rooted on a cut-edge and with a weight $z$ per cut-edge different from the root, the weight is obtained by the weight of the dual labeled tree and in particular includes a factor $x$ if the blob carries a red leaf, see Figure \ref{fig:blobenu1}. We shall also separately enumerate those blobs with a dangling red leaf. When rooted at the side of this dangling leaf (which is therefore not a cut-edge) we denote by 
 \begin{align*}
 {B}^{\redl}(z;\theta) = z + \tfrac18 (\pi^2 - 4\theta^2) z^2 + \cdots 
 \end{align*}
 the generating function of trees counted with a weight $z$ per cut-edge, with angle assignments associated to blobs carrying one red leaf, such that the root of the tree correspond to that leaf, and such that $\alpha_{1}+\beta_{1}+\theta < 2\pi$ (there is no weight $x$ for the root leaf as in \eqref{eq:treevolumeZ}). In particular, we have 
 $$  \partial_z{B}^{\redl}(z;0) = [x^1]B_{x}(z).$$

 \begin{figure}[!h]
  \begin{center}
  \includegraphics[width=15cm]{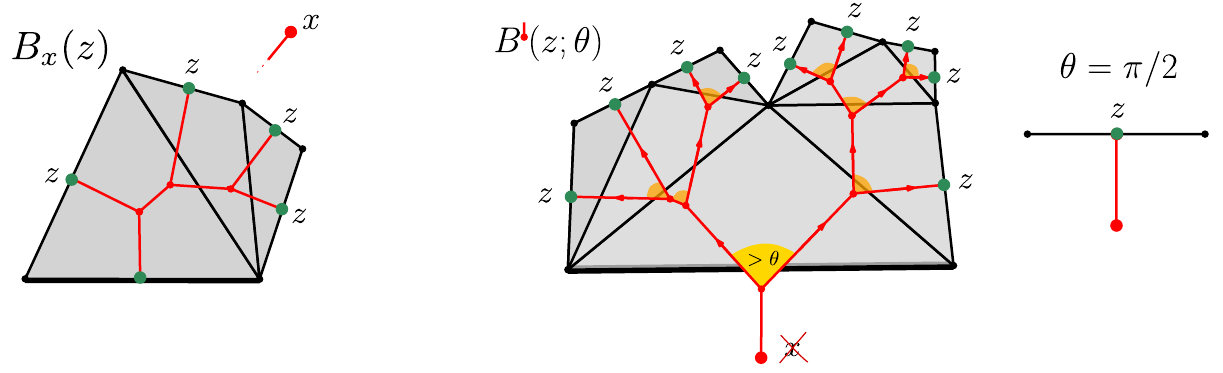}
  \caption{   Illustration of the definition of the generating series $B_{x}(z)$ and ${B}_{x}^{\redl}(z;\theta)$ of labeled trees associated with blobs rooted on a cut-edge and blobs having one red leaf rooted at that side. The degree of $z$ is the number of non-root cut-edges. Notice that in the second case all internal angles (in orange on the figure) of such trees must belong to $(0,\pi/2)$ by definition of a blob. \label{fig:blobenu1}}
  \end{center}
  \end{figure}

By convention, the only such tree having angle $\pi/2$ corresponds to the trivial blob in the center of Figure \ref{fig:blobenu1}, so that ${B}_{x}^{\redl}(z;\pi/2) = z$. 
The decompositions presented in Section \ref{sec:blobdef} then yields the following simple relations:
 
 \begin{lemma}[Decomposition relations] \label{lem:decomposition} We have for any $x \in (0,x_{c}]$
 $$ B_{x}(F(x;\pi/2))= F(x; \pi/2) \quad \mbox{ and }\quad  {B}^{\redl}( F(x; \pi/2); \theta) = F(x; \theta), \quad \mbox{ for } 0 \leq \theta \leq \pi/2.$$
 \end{lemma}

Although those equations fully characterize the generating functions $B$ and $ B^{\redl}$ (implicitly), the second one is amenable to the same analysis as in Theorem \ref{thm:zograf}. The main remark is that the labeled trees corresponding to blobs with one red leaf, when rooted at that leaf, correspond to those trees in which the ``internal'' angles, in orange in Figure \ref{fig:blobenu1}, belong to $(0, \pi/2]$, instead of $(0, \pi]$ as it was the case before. 
  \begin{proposition}[Enumeration of blobs with one red leaf] \label{prop:blobredleaf} We have 
$$ {B}^{\redl}(z; \theta) =  4F\left( \frac{z}{4} ; 2 \theta \right).$$
\end{proposition}
  \begin{proof}The exact same decomposition yielding  \eqref{eq:Frecurrence} shows that ${B}^{\redl}(\theta) := B^{\redl}(z; \theta)$ satisfies the same integral equation where $\pi$ has been replaced by $\pi/2$, i.e. 
  \begin{equation*}\label{eq:Frecurrencebis}
{B}^{\redl}(\theta) = z + \int_0^{\pi/2} \!\rmd\alpha\int_0^{\pi/2}\!\rmd\beta\, {B}^{\redl}(\alpha){B}^{\redl}(\beta)\, \ind_{\{\theta<\alpha+\beta<\pi/2\}}, \qquad \theta\in[0,\pi/2]. 
\end{equation*}
Proceeding along the lines of the proof of Theorem \ref{thm:zograf}, and plugging in the boundary condition ${B}^{\redl}(z ;\pi/2) =z$ we deduce using \eqref{eq:Frecurrence} and \eqref{eq:Ftheta} that 
$${B}^{\redl}(z; \theta) =  \frac{ \sqrt{ {Z}^{\redl}(z)}}{\theta} J_{1}\left( 2 \theta \sqrt{{Z}^{\redl}(z)}\right)$$
where ${Z}^{\redl}(z) = {B}^{\redl}(z; 0)$ solves
$$ \frac{\sqrt{{Z}^{\redl}(z)}}{\pi/2} J_1(\pi \sqrt{{Z}^{\redl}(z)}) = z.$$
Comparing the previous display with the definition of $Z$ in \eqref{eq:zografgenfun} we see that ${Z}^{\redl}(4z) = 4 Z(z)$.
Finally, comparing ${B}^{\redl}(z; \theta)$ with \eqref{eq:Fxtheta} we find the claimed relation.\end{proof}

\part{Limits of random labeled trees}

In this part, we study the random labeled trees appearing in the coding of random surfaces in Theorem \ref{thm:treeencoding}, paving the way for our main results Theorems \ref{thm:BS} and \ref{thm:GH} on the geometry of random hyperbolic surfaces.  We first define various probability distributions on the set of (possibly infinite) binary plane labeled trees (i.e.~with allowed angle configuration).

\paragraph{Boltzmann laws.}  A generic labeled tree will always be written $\boldsymbol{\tau}=(\tau, ( \theta_{c} : c \in \mathrm{Corners}(\tau)))$ and its size $|\tau|$ is its number of leaves.  First, for $n \geq 2$, we denote by $ \mathbb{P}^{n}$ the law of a random labeled tree sampled according to the Lebesgue measure on $ \mathcal{A}_{n}$ see \eqref{def:An}. To ease notation we shall also use $\bT_{n}$ for a random variable of this law under the generic probability $ \mathbb{P}$. \\

  More generally, recall from Section \ref{sec:enumeration} the definition of the generating functions $F(x ; \theta)$ for any $\theta \in (0,\pi)$, and denote by $ \mathbb{P}_{x, \theta}$ the probability measure on labeled trees 
$$ \mathbb{P}_{x, \theta} =   \frac{1}{F(x ; \theta)}\sum_{n \geq 2} x^{n-1} \sum_{\tau \in \mathrm{Bin}_{n}} \mathrm{Leb}_{ \mathcal{A}_{\theta,\tau}}.$$  We next introduce the fixed-size versions  of those measures: for any $n \geq 2$, let 
$$\mathbb{P}_{x,\theta}^{n} := \mathbb{P}_{x, \theta}( \cdot  \mid |\tau| =n),$$  where the definition again does not depend upon $x \in (0, x_{c}]$. In particular $\mathbb{P}_{ x,0}^{n} = \mathbb{P}^{n}$. Finally, recall from \eqref{eq:xc} that  the radius of convergence of  $Z(\cdot) = F(\cdot ; 0)$ is $x_c = \frac{c_0}{2\pi^2} J_1(c_0)$. We shall often take $x=x_c$ and if so, drop $x$ from notation, e.g. we use 
$$ F(\theta) \equiv F(x_c ; \theta), \quad \mathbb{P}_\theta \equiv \mathbb{P}_{x_c, \theta},\quad \ldots . $$ 
 We note  from \eqref{eq:Frecurrence} that the function $F(\theta) = F(x_{c}; \theta)$ is decreasing in particular bounded from above and below by $F(0)$ and $F(\pi)$ which are both finite and positive.  \begin{figure}[!h]
 \begin{center}
 \includegraphics[width=8cm]{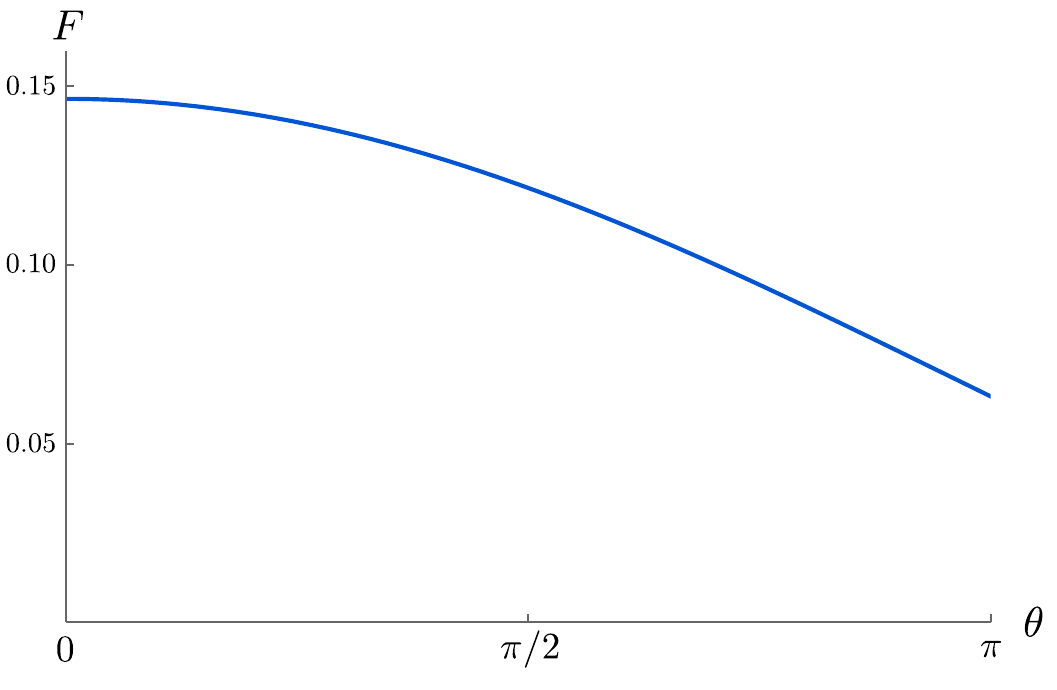}
 \caption{Plot of the function $\theta \mapsto F(\theta)\equiv F(x_{c};\theta)$. In particular, the function is bounded above and below on the whole interval $[0,\pi]$. \label{fig:F}}
 \end{center}
 \end{figure}

  Finally to simplify notation we introduce 
 $$\mbox{for all }\theta \in (0, \pi), \quad F_\infty(\theta) := J_0(\tfrac{\theta}{\pi}c_0), $$ so that Lemma  \ref{lem:asymptotics} in particular rewrites as  \begin{align} \label{eq:asymptgeneral}
 \mathbb{P}_\theta(|\tau|=n) \sim \frac{c_0}{2\sqrt{2}\,\pi^{5/2}}\frac{F_\infty(\theta)}{F(x_c;\theta)}\, n^{-3/2}.
 \end{align}

\section{Local limits}
Let us first make precise the setting for the local convergence. We shall be interested in local limits of the labeled tree $ \boldsymbol{{T}_{n}}$ of law $ \mathbb{P}^{n}$, both seen from a uniform leaf (its original rooting) or from a uniform vertex (useful for establishing Theorem \ref{thm:BS}). Those local limits shall be described by infinite one-ended trees of the type of Kesten's trees \cite{Kes86} or a variant of Aldous' sin trees \cite{Ald91c} for mono-type Galton--Watson trees. We refer the reader to \cite{AD14,Yvette} for background.  Let us first define what we mean by local convergence in our setting of planar rooted labeled and possibly bicolored trees. \\

Let $ \boldsymbol{\tau}$ and $ \boldsymbol{\tau'}$ be locally finite possibly infinite plane (non-necessarily binary) trees with a real labeling on their edges. In the following, the vertices of $\tau$ and $\tau'$ may furthermore be colored either black or red. The local distance between $\boldsymbol{\tau}$ and $\boldsymbol{\tau'}$ is defined as 
$$ \mathrm{d_{loc}}( \boldsymbol{\tau}, \boldsymbol{\tau'}) = 2^{- \sup\{r \geq 0: [ \boldsymbol{\tau}]_{r} = [ \boldsymbol{\tau'}]_{r}\}},$$
where $[\boldsymbol{\tau}]_{r}$ is the plane tree obtained by keeping only the first $r$ generations around the root vertex, keeping the same colors and the same labels on the edges. Notice that we \textbf{do not} allow for small variations of the labels: they should match exactly in $[\boldsymbol{\tau}]_{r}$ and $[\boldsymbol{\tau'}]_{r}$. This is easily seen to be a complete distance on the space of plane locally finite labeled trees (but beware, this space is not separable!) and we shall use the associated convergence in distribution. Compared to usual settings, the convergence we impose on the labels corresponds to the total variation distance and is much stronger than a standard weak convergence in $ \mathbb{R}$. Equivalently, a sequence of random labeled trees $ \bT_n$ converges to $\boldsymbol{T}$ in this sense, if for every $r \geq 0$, we have   \begin{eqnarray} \label{eq:defloc}	 \mathrm{d_{TV}}([\bT_n]_{r}, [\boldsymbol{T}]_{r}) \to 0, \quad \mbox{ as } n \to \infty,  \end{eqnarray} or yet otherwise said, that we can couple $\bT_n$ and $\boldsymbol{T}$ on the same probability space so that $[\bT_n]_{r}$ and $[\boldsymbol{T}]_{r}$ are the same objects with a probability tending to $1$ as $n \to \infty$. We shall write 
$$  \bT_n  \xrightarrow[n\to\infty]{(loc)} \boldsymbol{T}.$$

\subsection{Multi-type Galton--Watson}
We will establish local limits of the coding trees of random WP surfaces in two ways: first by interpreting them as multi-type Galton--Watson tree (with types in $[0,\pi]$) and then by finding a mono-type Galton--Watson tree underneath them. The second approach requires more notation but will also be useful when studying scaling limits. We start with the first approach which will give us precise estimates on the labels.
\subsubsection{Kesten's limit}
 Notice that \eqref{eq:Frecurrence}, translated in probabilistic language, shows that under $\mathbb{P}_{x, \theta}$ the random labeled tree $ \boldsymbol{\tau}$ is a multi-type Galton--Watson with types in $[0,\pi]$. More precisely, under $ \mathbb{P}_\theta=\mathbb{P}_{x_c, \theta}$ the probability that the tree is reduced to a single edge (labeled $\ell_{e_{\varnothing}}=0)$ is equal to $$ \mathbb{P}_{\theta}(\btau = \raisebox{-2mm}{\includegraphics[width=0.3cm]{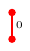}}) =  \frac{x_c}{F(\theta)},$$ and otherwise if $\btau_{\text{L}}$ and $\btau_{\text{L}}$ are the two labeled subtrees obtained by pruning the first vertex after the root (we duplicate this vertex so that $\btau_{\text{L}}$ and $\btau_{\text{L}}$ are rooted on a leaf), we have 
 \begin{eqnarray} \label{eq:GWsplit} \mathbb{E}_{\theta}\left[ \varphi(\btau_{ \text{L}}) \psi( \btau_{ \text{R}})\right] &=&  \iint_{0}^{\pi} \mathrm{d} \alpha \mathrm{d}\beta  \mathbf{1}_{\theta < \alpha+\beta < \pi} \frac{F( \alpha) F( \beta)}{F( \theta)}\mathbb{E}_{\alpha}\left[\varphi(\btau_{\text{L}})\right] \mathbb{E}_{\beta}\left[\psi(\btau_{\text{R}})\right].  \end{eqnarray} Note in passing that the number of leaves of $\tau_{\text{L}}$ and the number of leaves of $\tau_{\text{R}}$ is equal to the number of leaves of $\tau$ plus $1$. 
Our first convergence is then :

\begin{theorem}[Local limit from a leaf] \label{thm:localleaf}We have the following local convergence in distribution 
$$ \mathbb{P}_{\theta}^{n}  \xrightarrow[n\to\infty]{(d)} \mathbb{P}_{\theta}^{\infty},$$
where the distribution $\mathbb{P}_{\theta}^{\infty}$ is supported by infinite trees with one spine which are characterized by the Markov property at the root described in Figure \ref{fig:splitting}.
\end{theorem}

\begin{figure}[!h]
 \begin{center}
 \includegraphics[width=13cm]{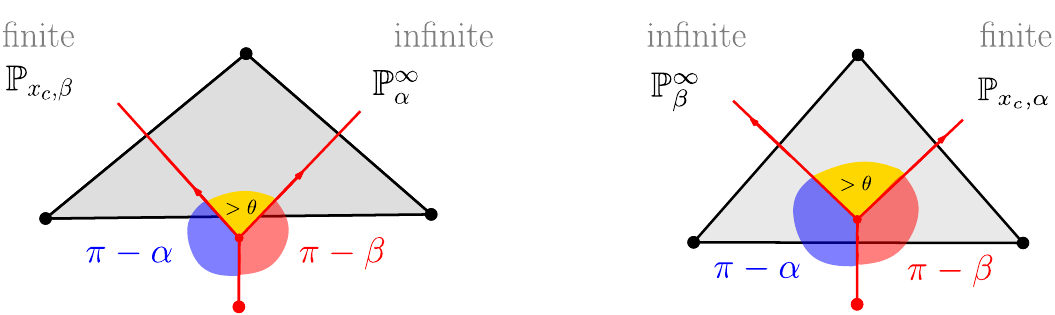}
 \caption{Fix $\theta \in (0,\pi)$ and denote by $E_{(\alpha, < \infty),(\beta, \infty)}$, respectively $E_{(\alpha, \infty),(\beta, <\infty)}$ the event in which the random labeled tree under $\mathbb{P}_{\theta}^{\infty}$ splits into two trees (one finite, one infinite) according to the internal angles $\pi-\alpha, \pi-\beta$. Then those events are the only two possibilities and conditionally $E_{(\alpha, < \infty),(\beta, \infty)}$ (resp. $E_{(\alpha, \infty),(\beta, <\infty)}$) the two remaining trees are independent and of law $ \mathbb{P}_{x_{c}, \alpha}$ and $ \mathbb{P}_{\beta}^{\infty}$ (resp. $ \mathbb{P}_{\alpha}^{\infty}$ and $ \mathbb{P}_{x_{c}\beta}$). The density of $E_{(\alpha, < \infty),(\beta, \infty)}$  under $\mathbb{P}_{\theta}^{\infty}$ is $\mathrm{d}\alpha \mathrm{d}\beta \mathbf{1}_{\theta < \alpha+\beta < \pi} \frac{F(\alpha)F_{\infty}(\beta)}{F_{\infty}(\theta)}$, and similarly for $E_{(\alpha,  \infty),(\beta, < \infty)}$ by exchanging $\alpha$ and $\beta$.
  \label{fig:splitting}}
 \end{center}
 \end{figure}
 \begin{proof} This is a classical result for mono-type or discrete multi-type Galton--Watson trees once we are in possession of the asymptotics \eqref{eq:asymptgeneral}. Let us sketch the proof in our case and explain the subtlety of our definition of local convergence. We start  by proving that there exists $C>0$ such that \begin{eqnarray} \sup_{\theta} \mathbb{P}_{\theta}( |\tau|=n) \leq C\, n^{-3/2}\label{eq:uniftheta}. \end{eqnarray}
Indeed, by \eqref{eq:GWsplit} we have 
  \begin{eqnarray*} \mathbb{P}_{\theta}(|\tau|= n) &=& \iint_{0}^{\pi} \mathrm{d} \alpha \mathrm{d}\beta  \mathbf{1}_{\theta < \alpha+\beta < \pi} \frac{F( \alpha) F( \beta)}{F( \theta)} \sum_{n_{1} + n_{2} = n+1} \mathbb{P}_{\alpha}(|\tau| = n_{1})\cdot \mathbb{P}_{\beta}(|\tau| = n_{2}) \\
  & \leq &  \frac{F(0)}{F(\theta)}\iint_{0}^{\pi} \mathrm{d} \alpha \mathrm{d}\beta  \mathbf{1}_{0 < \alpha+\beta < \pi} \frac{F( \alpha) F( \beta)}{F(0)}  \sum_{n_{1} + n_{2} = n+1} \mathbb{P}_{\alpha}(|\tau| = n_{1})\cdot \mathbb{P}_{\beta}(|\tau| = n_{2})\\  &=& \frac{F(0)}{F(\theta)} \mathbb{P}_{0}(|\tau|= n),  \end{eqnarray*}
  and the claim follows from \eqref{eq:asymptgeneral} and the fact that $F$ is bounded away from $0$. Now we use \eqref{eq:uniftheta} to prove that it is unlikely that a random tree under $\mathbb{P}_{\theta}( \cdot \mid |\tau|=n)$ splits at the root in two subtrees, both of size larger than $A >0$ since this event is bounded above using \eqref{eq:GWsplit} by 
 $$ \frac{1}{\mathbb{P}_{\theta}(|\tau|=n)} \sum_{m_{1},m_{2} \geq A} \iint_{0}^{\pi} \mathrm{d} \alpha \mathrm{d}\beta  \mathbf{1}_{\theta < \alpha+\beta < \pi} \  \frac{F( \alpha) F( \beta)}{F( \theta)} \mathbb{P}_{\alpha}(|\tau| = m_{1}) \cdot \mathbb{P}_{\beta}(|\tau| = m_{2}).$$
Notice that at least one of $m_{1}, m_{2}$ is larger than $n/2$  so that using \eqref{eq:uniftheta} and the bounds on $F$ the above display is bounded above by some constant time $\sum_{m \geq A} m^{-3/2}$ which tends to $0$ as $A \to \infty$. Once this is granted, any potential subsequential limits of $\mathbb{P}_{\theta}^{n}$ must then be one-ended. If $\varphi$ is a bounded measurable function on the space of finite labeled tree which is zero as soon as the tree has size larger than $A$ one can write using \eqref{eq:GWsplit}

\begin{eqnarray*} \mathbb{E}_{\theta}^{n}\left[ \varphi( \btau_{ \text{L}})\right] 
&=& \frac{1}{\mathbb{P}_{\theta}(|\tau|=n)} \iint_{0}^{\pi} \mathrm{d} \alpha \mathrm{d}\beta  \mathbf{1}_{\theta < \alpha+\beta < \pi} \  \frac{F( \alpha) F( \beta)}{F( \theta)}\mathbb{E}_{ \alpha}\left[ \varphi( \btau_{ \text{L}}) \mathbb{P}_{ \beta} \left( |\tau|  = n+1 - |\tau_{ \text{L}}| \right) \right] . \end{eqnarray*} Using the bounds on $F$ as well as  $\eqref{eq:asymptgeneral}$ and $\eqref{eq:uniftheta}$ one gets by dominated convergence theorem that the above tends as $n \to \infty$ towards 	
$$\iint_{0}^{\pi} \mathrm{d} \alpha \mathrm{d}\beta  \mathbf{1}_{\theta < \alpha+\beta < \pi} \  \frac{F( \alpha) F_{\infty}( \beta)}{F_{\infty}( \theta)}\mathbb{E}_{ \alpha}\left[ \varphi( \btau_{ \text{L}})  \right].$$
This indeed shows that when $\btau_{\text{L}}$ is finite, it converges in total variation distance under $\mathbb{P}_{\theta}^{n}$ towards the desired law. By iterating the argument we get the statement of the theorem.  \end{proof}

\subsubsection{Labels along the spine} \label{sec:spine}
Recall from Section \ref{sec:euclidean} that we can define the $\ell$-labeling along the edges of $\btau$. Although the variation of $\ell$ are a priori unbounded, we first provide rough bounds  showing that they are uniformly controlled: We denote by $\Delta$ the variation of $\ell$ between the root edge $e_\varnothing$ and the edge immediately on its left.

 \begin{lemma}[Rough exponential bound] \label{lem:roughexpo} There exists $ \mathrm{Cst}>0$ such that for all $\theta \in [0, \infty)$
 $$  \mathbb{P}_{\theta}(|\Delta|>x) \leq \mathrm{Cst}\ \mathrm{e}^{-x/4}.$$
 \end{lemma}
 \begin{proof} Recalling Figure \ref{fig:length-angle}, with the notation of Figure \ref{fig:treeequation} we have 
 $$ \Delta = 2 \log \left( \frac{\sin \beta}{ \sin \alpha+\beta}\right),$$ where $\alpha$ and $\beta$ are the left and right angles at the first splitting in the tree $\btau$. Since the function $F(\cdot)$ is bounded above  and below (see Figure \ref{fig:F}), there exists a constant $C>0$ such that we have 
  \begin{eqnarray} \label{eq:densitybound} \sup_{\theta \in [0,\pi]}\mathbf{1}_{\theta < \alpha + \beta < \pi} \frac{F(\alpha)F(\beta)}{F(\theta)} \leq  \frac{C}{\pi^{2}}.  \end{eqnarray} We deduce in particular (a fact that will be used later in Proposition \ref{prop:density}) that under $ \mathbb{P}_\theta$, the density of  the pair $\alpha, \beta$ is bounded above on $[0, \pi]^{2}$.  So that if $U,V$ are independent uniform variables over $[0,\pi]$ we can write
  \begin{eqnarray*} \sup_{\theta} \mathbb{P}_{\theta}( |\Delta| > x) &\underset{ \eqref{eq:GWsplit}}{\leq}&  C \cdot \mathbb{P}\left(\left|2 \log \left( \frac{\sin V}{ \sin U+V}\right)\right| > x\right) \\
  & \leq& C \left( \mathbb{P}\left(| \log  \sin V| > x/4\right) +  \mathbb{P}\left(| \log  \sin (U+ V) | > x/4\right)\right).  \end{eqnarray*}
 Using the fact that $\sin(V)$ and $\sin(U+V)$ have a density over $[0,1]$ which is bounded in the vicinity of $0$ (but not near $1$), for large $x\gg 1$ one can write that  $ \mathbb{P}(|\log(\sin(V))|> x/4) \leq  C'\mathbb{P}(|\log (U)|>x/4) \leq C' \mathrm{e}^{-x/4}$  for some $C'>0$ and similarly for the other term. Gathering-up the pieces we proved that for some $ \mathrm{Cst}>0$ we have 
 $$ \mathbb{P}_{\theta}(|\Delta|>x) \leq \mathrm{Cst}\, \mathrm{e}^{-x/4}.$$ \end{proof}

 Since a random labeled tree under $ \mathbb{P}_{0}^{\infty}$ has a unique spine, we can define the angles along the spine $(\alpha_{k}, \beta_{k})_{k \geq 0}$ with the convention that the angle $\alpha$ is always associated to the infinite part of the tree, see Figure \ref{fig:anglesspine}.
 
 \begin{figure}[!h]
  \begin{center}
  \includegraphics[width=10cm]{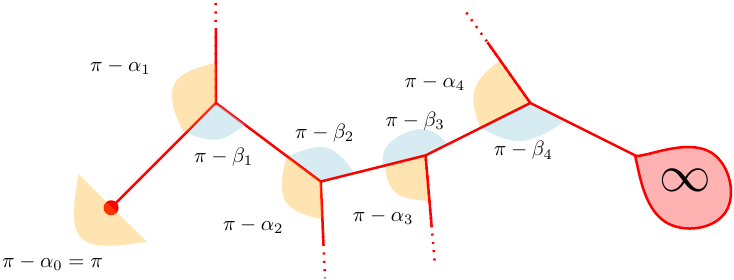}
  \caption{The angles along the unique spine under the law $\mathbb{P}_0^\infty$. \label{fig:anglesspine}}
  \end{center}
  \end{figure}
 
 \begin{proposition}[Invariant distribution of angles/labels along the spine] \label{prop:spine} Let $0=\alpha_0, \alpha_1, \alpha_2, ...$ the Markov chain made of the angles along the spine under the law $\mathbb{P}_0^\infty$. Then $(\alpha_k)_{k\geq 0}$ is a Markov chain on $[0,\pi)$ with transition probabilities $$ p_{\theta \to \beta}  = \mathrm{d}\beta  \int_{0}^{\pi} \mathrm{d}{\alpha} \, \mathbf{1}_{\theta < \alpha+ \beta < \pi} \frac{2 F(x_{c}; \alpha) F_\infty(\beta)}{F_\infty(\theta)}.$$
 The chain is irreducible and ergodic with invariant distribution $\nu(\theta)\rmd\theta$ given by 
 $$ \nu(\theta) = \frac{F_{\infty}(\theta) F_{\infty}'(\pi-\theta)}{\cos c_0} = \frac{-c_0}{\pi \cos c_0} J_0(\tfrac{\theta}{\pi}c_0) J_1(\tfrac{\pi-\theta}{\pi}c_0).$$
 \end{proposition}

 \begin{proof}
 The law of the labels along the spine is directly read off from the Markov property of Theorem~\ref{thm:localleaf}. 
 One may easily check that $p_{\theta \to \beta}$ is bounded from below by a strictly positive continuous density on $(0,1)$ uniformly in $\theta$, showing that Markov chain is irreducible and ergodic.
 For later reference, we record in particular that
 \begin{align}\label{eq:Finftyintegral}
  \int\int_0^\pi \rmd\alpha\rmd\beta \mathbf{1}_{\theta < \alpha+ \beta < \pi} 2F(\alpha) F_\infty(\beta) = F_\infty(\theta).
 \end{align}

 To verify the claimed formula for the stationary distribution $\nu(\theta)\rmd \theta$, we first check that it is a probability measure.  
 Note that $F_\infty(\theta)$ is positive and decreasing on $\theta \in (0,\pi)$ and that $\cos c_0 < 0$, so $\nu(\theta) \geq 0$. 
 By \eqref{eq:laplace}, the Laplace transforms of $F_\infty(\theta)$ and its derivative are 
 \begin{align*}
  \int_0^\infty  \mathrm{e}^{-t\theta} F_\infty(\theta) \rmd \theta = \frac{1}{\sqrt{t^2+ \frac{c_0^2}{\pi^2}}}, \quad \int_0^\infty \mathrm{e}^{-t\theta} F_\infty'(\theta) \rmd \theta = \frac{t}{\sqrt{t^2+ \frac{c_0^2}{\pi^2}}}-1.
 \end{align*}
 Since we also have
\begin{align*}
  \int_0^\infty \cos( \tfrac{\alpha}{\pi} c_0) \mathrm{e}^{-t\alpha}\rmd \alpha = \frac{t}{t^2+ \frac{c_0^2}{\pi^2}},
\end{align*}
it follows from the convolution property of the Laplace transform that
 \begin{align*}
 \int_0^{\alpha} F_\infty(\theta)F'_\infty(\alpha - \theta) \rmd \theta = \cos( \tfrac{\alpha}{\pi} c_0) - F_\infty(\alpha). 
 \end{align*}
 Setting $\alpha = \pi$ and using that $F_\infty(\pi) = 0$, this verifies the normalization $\int_0^\pi \nu(\theta)\rmd\theta = 1$.

 For the invariance of $\nu(\theta)\rmd\theta$, we note that $F_\infty$ obey
 \begin{align}\label{eq:Finftyprime}
  F_\infty'(\theta) = -2 \int_0^{\theta} F(\alpha) F_\infty(\theta - \alpha)\rmd \alpha,
 \end{align}
 which follows from computing the $\theta$-derivative of the relation \eqref{eq:Finftyintegral}.
 Then we compute
 \begin{align*}
    \int_0^\pi F_\infty(\theta)F'_\infty(\pi-\theta)\, p_{\theta \to \beta}\, \rmd \theta &= \rmd \beta \int_0^{\pi-\beta} \rmd \alpha \int_0^{\alpha+\beta} \rmd \theta \,2 F(\alpha)F_\infty(\beta) F'_\infty(\pi-\theta) \\
    &=-2 F_\infty(\beta) \rmd \beta \int_0^{\pi-\beta} \rmd \alpha F(\alpha) F_\infty(\pi-\beta-\alpha) \\
    &= F_\infty(\beta)F'_\infty(\pi-\beta), 
 \end{align*}
 showing that $\nu(\theta)\rmd\theta$ is indeed invariant.
 \end{proof}
 
For our computation of scaling constants later in Section~\ref{sec:scalingcst} it will be useful to consider a related Markov process on $[0,\pi)$.

 \begin{figure}[!h]
  \begin{center}
  \includegraphics[width=.7\linewidth]{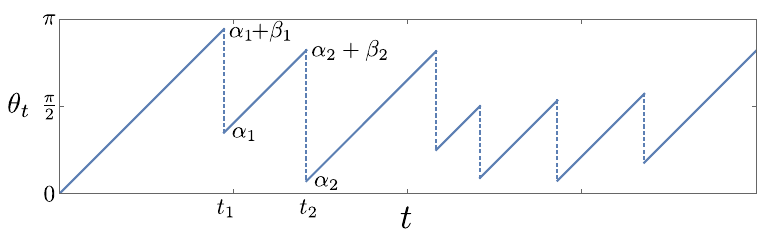}
  \caption{The process $(\theta_t)_{t\geq 0}$ of Proposition~\ref{prop:angleprocess}.\label{fig:angleprocess}}
  \end{center}
  \end{figure}

\begin{proposition}\label{prop:angleprocess}
  Let $(\theta_t)_{t\geq 0}$ be the process determined by the angles along the spine under the law $\mathbb{P}_0^\infty$ as follows.
  It starts at $0$, has constant drift $1$ and negative jumps of size $-\beta_k$ occurring at time $t_k = \alpha_k + \sum_{j=1}^k \beta_j$ for each $k \geq 1$, so that $\theta_{t_k} = \alpha_k$ and $\theta_{t_k-} = \alpha_k + \beta_k$.
  Then $(\theta_t)_{t\geq 0}$ is a Markov process on $[0,\pi)$ with generator
  \begin{align*}
    (L\phi)(\theta) = \phi'(\theta) + \int_0^\theta \rmd \alpha 2 F(\theta - \alpha) \frac{F_\infty(\alpha)}{F_\infty(\theta)} (\phi(\alpha) - \phi(\theta))
  \end{align*}
  and stationary distribution $\tilde{\nu}(\theta)\rmd \theta$ given by
  \begin{align*}
    \tilde{\nu}(\theta) = \frac{c_0}{\pi \sin c_0}\,F_\infty(\theta)F_\infty(\pi - \theta).
  \end{align*}

\end{proposition}
\begin{proof}
  For a Markov process with this generator started at $\theta \in [0,\pi)$, the density of the first  negative jump $t$  and (absolute) size $\beta$, provided that $0 < \beta < \theta+t< \pi$ is given by
  \begin{eqnarray*}
&&  \exp\left( - \int_{0}^{t}  \mathrm{d}s \int_{0}^{\theta+s} \mathrm{d}\alpha \, 2 F( \theta+s-\alpha) \frac{F_{\infty}(\alpha)}{ F_{\infty}(\theta+s)} \right) \times 2 F(\beta) \frac{F_{\infty}(\theta +t -\beta)}{ F_{\infty}(\theta+t)} \\
&=&  \exp\Big( - \int_{0}^{t}  \mathrm{d}s  \frac{1}{F_{\infty}(\theta+s)} \underbrace{2\int_{0}^{\theta+s} \mathrm{d}\alpha \, F( \theta+s-\alpha) F_{\infty}(\alpha)}_{ \normalsize \underset{ \eqref{eq:Finftyprime}}{=} - F'_{\infty}(\theta+s)} \Big) \times 2 F(\beta) \frac{F_{\infty}(\theta +t -\beta)}{ F_{\infty}(\theta+t)}\\
&=&     2 F(\beta) \frac{F_\infty(\theta+t-\beta)}{F_\infty(\theta)}. 
  \end{eqnarray*}
  Under the identification $t = \alpha + \beta - \theta$ this is precisely the law of the angles $\alpha$ and $\beta$ of Theorem~\ref{thm:localleaf} conditionally on the right tree being infinite. 
  
  The computation of the stationary distribution is similar to that of Proposition~\ref{prop:spine}. 
  We need to check that $\int_{0}^\pi \rmd\theta \tilde{\nu}(\theta) (L\phi)(\theta) = 0$ for any $\phi$ sufficiently regular and vanishing at $0$ and $\pi$.
  This follows from the identity 
  \begin{align*}
    &\int_0^\pi \rmd\theta F_\infty(\theta)F_\infty(\pi-\theta) \int_0^\theta \rmd\alpha 2F(\theta-\alpha)\frac{F_\infty(\alpha)}{F_\infty(\theta)}(\phi(\alpha)-\phi(\theta)) \\
    &= \int_0^\pi \rmd\theta \phi(\theta) \left(\int_\theta^\pi \rmd \beta 2 F(\beta-\theta) F_\infty(\theta)F_\infty(\pi-\beta) - \int_0^\theta \rmd\alpha 2 F(\theta-\alpha) F_\infty(\alpha)F_\infty(\pi-\theta)\right)\\
    &=\int_0^\pi \rmd\theta \phi(\theta) \left(- F_\infty(\theta)F_\infty'(\pi-\theta) + F_\infty(\pi-\theta) F_\infty'(\theta)\right) = -\int_0^\pi \rmd \theta F_\infty(\theta)F_\infty(\pi-\theta)\phi'(\theta), 
  \end{align*}
  where we used \eqref{eq:Finftyprime} twice in the second equality.
  For the normalization we proceed as in the proof of Proposition~\ref{prop:spine}, where this time we obtain
  \begin{align*}
    \int_0^\alpha F_\infty(\theta)F_\infty(\alpha-\theta)\rmd\theta = \frac{\pi}{c_0}\sin(\tfrac{\alpha}{\pi}c_0).
  \end{align*}
  Setting $\alpha = \pi$ verifies $\int_0^\pi \rmd \theta \,\tilde{\nu}(\theta) = 1$.
\end{proof}

\subsection{Blob-decomposition and a monotype Galton--Watson } 

\label{sec:piT} 
In this section we draw probabilistic consequences of the blob-decomposition of  trees of Euclidean triangles presented in Section \ref{sec:euclidean}.  We work under the unconditional probability measure $ \mathbb{P}_{0} \equiv \mathbb{P}_{x_{c},0}$ and connect this law to a mono-type Galton--Watson tree that will enable us to re-interpret Theorem \ref{thm:localleaf}, as well as establishing the forthcoming scaling limits results. 

Let $ \btau$ be a binary tree with allowed angle assignments which we will see as a red tree,  and recall from Section \ref{sec:euclidean} that we can decompose the tree of Euclidean triangles associated with $\btau$ into blobs, which yields  a new plane tree, called the \textbf{blob-tree}. This tree has two sorts  of vertices (see Figure \ref{fig:blob-3}): 
\begin{itemize}
\item each \textbf{black} vertex is associated to a blob, its degree corresponds to the number of sides of the blob.
 \item the {\color{red}red leaves} of $\tau$ are also the leaves of the blob-tree. More precisely, recall that each blob may carry at most one red leaf and that red leaves of $\tau$ that are trivial blob (with one side).
 \end{itemize}
 This bicolored blob-tree tree, denoted by $\pi( \boldsymbol{\tau})$, naturally inherits the planar ordering from $\tau$ and is rooted at the same red leaf as $\tau$. Let us emphasize that although the construction of $\pi( \boldsymbol{\tau})$ requires the angle assignments of $ \boldsymbol{\tau}$, the blob-tree $\pi( \boldsymbol{\tau})$ carries no label and is just a bicolored plane tree. If $u$ is a black vertex of $\pi( \boldsymbol{\tau})$ we denote by $ \mathrm{deg}^{b}(u)$ the number of black neighbors of $u$ in $\pi( \boldsymbol{\tau})$, which thus corresponds to the degree of the corresponding blob, i.e. its corresponding number of sides which are cut edges (the possible red dangling leaf does not count).
 
    \begin{figure}[!h]
    \begin{center}
    \includegraphics[width=14cm]{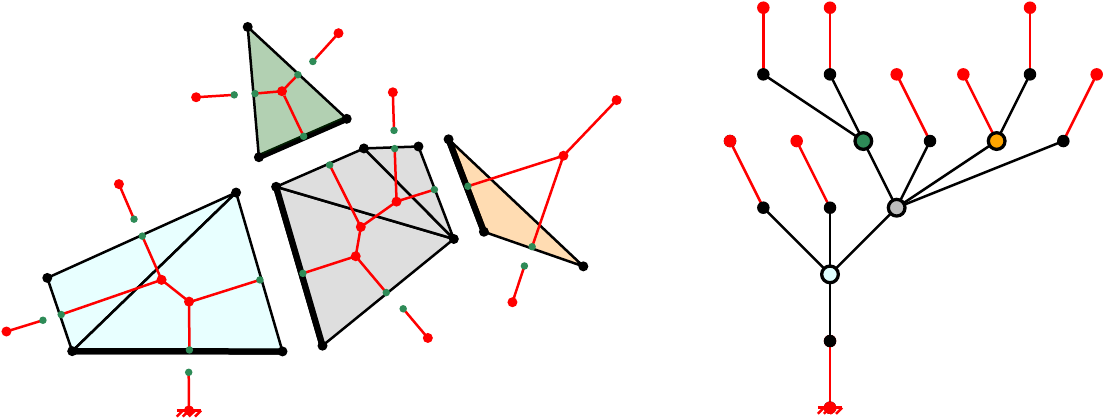}
    \caption{ Illustration of the definition of the bicolored blob-tree $\pi( \boldsymbol{\tau})$ (right) from the labeled red tree $ \boldsymbol{\tau}$ seen as a tree of Euclidean triangles (left). To ease the correspondence, the interior of some black vertices and of some blobs have been colored accordingly. In this decomposition, the blobs naturally carry a root cut-edge (in thick black).  \label{fig:blob-3}}
    \end{center}
    \end{figure}

   \subsubsection{Law of the blob-tree}
Recall from Section \ref{sec:enumeration} (Figure \ref{fig:blobenu1}) the generating functions $B_{x}(z)$ and $B^{\redl}(z)$. For $k \geq 0$, the total weight $[z^{k}] B_{x_{c}}(z)$ of blobs with degree $k+1$ and rooted on a cut edge, can be split into 
\begin{align}
[z^{k}] B_{x_{c}}(z) = ( \mathfrak{B}_{k+1} + x_{c} \mathfrak{B}_{k+1}^{\redl}),\label{eq:blobsplit}
\end{align} where $ \mathfrak{B}_{k+1} = [z^{k}][x^{0}]B_{x}(z)$ corresponds to the blobs having no red leaf and $ \mathfrak{B}_{k+1}^{\redl} = [z^{k}][x^{1}]B_{x}(z)= (k+1)[z^{k+1}] B^{\redl} (z;0)$  corresponds to those having a red leaf.  These weights are obtained by the integration over the angles of their dual red labeled trees, see Section \ref{sec:enumeration}. In particular we have $  \mathfrak{B}_{1} = \mathfrak{B}_{2} =0$ and recall that by convention we put $ \mathfrak{B}_{1}^{\redl} = 1$. From Proposition \ref{prop:blobredleaf} it follows in particular that 
$$ \mathfrak{B}_{k}^{\redl} = k[x^{k}] B^{\redl} (x;0) \underset{ \mathrm{Prop.} \ref{prop:blobredleaf}}{=} k 4^{1-k}[x^{k}]F(x;0) \underset{ \mathrm{Thm.} \ref{thm:zograf}}{=} \frac{8^{1-k}}{(k-1)!} V_{0,k+2}.$$ 

Now if $  \boldsymbol{\tau}$ has law $ \mathbb{P}_{0}$, it is straightforward to see that its blob-tree $\pi( \boldsymbol{\tau})$ has law prescribed by
\begin{align} \label{eq:lawsimplygenerated} \mathbb{P}_{0} (  \pi(  \boldsymbol{\tau})=\mathfrak{t}) = \frac{1}{x_{c}}\cdot\frac{1}{Z(x_c)} \prod_{ \begin{subarray}{c}u \in \mathfrak{t}\\ u \mathrm{ \ black} \end{subarray}} \left\{\begin{array}{ll} \mathfrak{B}_{ \mathrm{deg}^{ \mathrm{b}}(u)} & \mbox{ if } u \mbox{ has no red leaf}\\
x_c \mathfrak{B}_{ \mathrm{deg}^{ \mathrm{b}}(u)}^{\redl} \cdot \frac{1}{ \mathrm{deg}^{ \mathrm{b}}(u)} & \mbox{ if } u \mbox{ has a red leaf},
\end{array}\right. \end{align}
for any bicolored plane tree $ \mathfrak{t}$, with only red leaves, rooted on a red leaf and having at most one red leaf attached to any black vertex. In particular, the factor $\frac{1}{ \mathrm{deg}^{ \mathrm{b}}(u)}$ accounts for the fact that the position of the red leaf is fixed by the bicolored tree and the factor $ \frac{1}{x_{c}}$ comes from the fact that the root red leaf does not receive any weight in the function $Z$. Furthermore, conditionally on the combinatorial structure $\pi(  \boldsymbol{\tau})$, the blobs of $ \boldsymbol{\tau}$ are independent and of law uniform on blobs with combinatorics prescribed by $\pi(  \boldsymbol{\tau})$. Let us translate this result in a more probabilistic form using modified Galton--Watson trees. For this, we introduce the following:
\begin{definition}[Offspring distribution $p$] \label{def:pk} Let $q :=  F(x_c; \tfrac\pi2)= \frac{c_0}{\pi^2}J_1(c_0/2) = 0.12156\ldots$  and put 
\begin{align*}&p_{k-1} = (\mathfrak{B}_{k} + x_c  \, \mathfrak{B}_{k}^{\redl})\,q^{k-2}, \qquad \mbox{ for }k \geq 1, \end{align*}
also introduce
$$ r_{k-1} = \frac{x_c \,  \mathfrak{B}_{k}^{\redl}  q^{k-2}}{p_{k-1}} \in [0,1],$$
in particular $r_{0}=1$.
\end{definition}
Using Lemma \ref{lem:decomposition} and Proposition \ref{prop:blobredleaf} we can check:
\begin{proposition} \label{prop:offspringp} The measure $ p=(p_{k})_{k \geq 0}$ is a probability measure on $\{0,1,2,3,...\}$ which we will see as an offspring distribution. It is critical (it has mean $1$), has finite variance 
\begin{align*}
\sigma^{2} = 2 \frac{J_1(\tfrac{c_0}{2})J_1(c_0)}{J_0(\tfrac{c_0}{2})^3} = 1.72286...
\end{align*} and admits small exponential moments, more precisely $p_{k} = (q/(4 x_{c}))^{k+o(k)} = (0.4805...)^{k}$ as $k \to \infty$.
\end{proposition}

\begin{proof} 
Let us start by deriving an expression for $B_x(z)$.
From Proposition~\ref{prop:blobredleaf} we have that $[x^1]B_x(z) = \partial_z B^{\redl}(z;0) = Z'(\tfrac{z}{4})$. 
Then by Lemma~\ref{lem:decomposition}, $B_0(F(x;\pi/2)) = F(x;\pi/2) - x Z'(\tfrac{z}{4})$.
Recalling the function $S$ from \eqref{eq:Finvfunction} obeying $S(F(x_c,\pi)) = x$, for $z \in [0,q]$ we may use $F(S(4 Z(\tfrac{z}{4}));\pi/2) = z$, such that
\begin{align*}
 B_x(z) = B_0(z) + x Z'(\tfrac{z}{4}) = z + Z'(\tfrac{z}{4})\left(x - S(4 Z(\tfrac{z}{4}))\right). 
\end{align*}
In particular, since $S$ is entire, the function $B_{x_c}(z)$ has radius of convergence $4 x_c$.
According to Proposition~\ref{prop:blobredleaf} we have $4Z(q/4) = Z(x_c)$, so expanding around $z=q$ gives
\begin{align*}
S(4 Z(\tfrac{z}{4})) &= x_c - \frac{\pi^2 J_1(c_0)}{c_0} Z'(\tfrac{q}{4})^2 (z-q)^2 + O((z-q)^3),\\
B_{x_c}(z) &= z+ \frac{\pi^2 J_1(c_0)}{c_0} Z'(\tfrac{q}{4})^3 (z-q)^2 + O((z-q)^3).
\end{align*}
The (offspring) generating function for $(p_k)_{k\geq 0}$ thus obeys
\begin{align*}
  \sum_{k=1}^\infty p_k y^k = \tfrac{1}{q}B_{x_c}(y q) = y + \frac{\pi^2 J_1(c_0)}{c_0} q Z'(\tfrac{q}{4})^3 (y-1)^2 + O((y-1)^3).
\end{align*}
We see that it indeed is a critical probability distribution with variance,
\begin{align*}
  \sigma^2 = 2 \frac{\pi^2 J_1(c_0)}{c_0} q Z'(\tfrac{q}{4})^3.
\end{align*}
Using $Z'(q/4) = 1 / J_0(c_0/2)$ and $q = \tfrac{c_0}{\pi^2}J_1(c_0/2)$ this gives the claimed formula.
\end{proof}
We also define the slightly modified probability distribution for $k \geq 1$
  \begin{eqnarray} \label{eq:distribroot} {p}^{\redl}_{k} = \frac{\frac{1}{k} p_{k-1}r_{k-1}}{\sum_{k=1}^\infty \frac{1}{k} p_{k-1}r_{k-1}} = \frac{1}{Z(x_c)} \frac{1}{k}\mathfrak{B}_{k}^{\redl} q^k ,  \end{eqnarray} where the normalization is computed using Lemma \ref{lem:decomposition}. We will now give an equivalent description of the bicolored tree of law \eqref{eq:lawsimplygenerated}  based on a modified  Galton--Watson tree. For this, we need a random operation transforming a black tree into a bicolored tree:
  \begin{definition}[$\mathrm{Red}$ and $ \mathrm{Red}^{\redl}$]  \label{def:reds}Let $  \mathsf{t}$ be a plane (black) tree. We denote by $ \mathrm{Red}( \mathsf{t})$ the bicolored tree where independently for each vertex $u \in \mathsf{t}$ with $k$ (black) children, we attach to it a red leaf with probability $r_{k}$, and if so, the red leaf is uniformly placed among the children (i.e. there are $k+1$ choices). In particular $ \mathrm{Red}( \mathsf{t})$ is a random bicolored tree still rooted at the root of $ \mathsf{t}$.\\
  We also define the variant $ \mathrm{Red}^{\redl}( \mathsf{t})$ obtained by performing the same operation as above for vertices different from the root, but where a special treatment is applied to the root vertex of $ \mathfrak{t}$ to which a red leaf is always attached on the root corner and on which the bicolored tree is now rooted. 
  \end{definition}
  
   \begin{proposition}[The blob-tree as a modified GW tree] \label{prop:decomp} Let $ \mathfrak{T}$ be a $p$-Galton--Watson tree except that the ancestor vertex has distribution ${p}^{\redl}$. Since $ p$ is critical, $ \mathfrak{T}$ is almost surely finite and we see it as a black tree. Then 
   $$\mathrm{Red}^{\redl}( \mathfrak{T}) \quad \overset{(d)}{=} \quad  \pi( \boldsymbol{\tau}) \quad \mbox{ under } \quad \mathbb{P}_{x_{c},0}.$$ 
  \end{proposition}
  \proof 
  Let us check that for any bicolored plane tree $\mathfrak{t}$, the probability $\mathbb{P}_0(\mathrm{Red}^{\redl}( \mathfrak{T}) = \mathfrak{t})$ is identical to \eqref{eq:lawsimplygenerated}.
  Note that the randomness involved in building $ \mathrm{Red}^{\redl}( \mathfrak{T})$ from $ \mathfrak{T}$ is independent of $\mathfrak{T}$ itself.
  The probability receives independent contributions from each black vertex of $\mathfrak{t}$.
  For a given black vertex of degree $k$, if this vertex has no red leaf attached, the contribution is clearly
  $$ p_{k-1} (1-r_{k-1}) = \mathfrak{B}_{k} q^{k-2}.$$
  If the vertex has a red leaf attached that is not the root, it is 
  $$ p_{k-1} r_{k-1} \frac{1}{k} = x_{c}\mathfrak{B}_{k}^{\redl}  \frac{1}{k}q^{k-2},$$
  where the factor $1/k$ accounts for the probability that the red leaf is chosen in the right place.
  Finally, if the black vertex has the root red leaf attached to ti, it is simply
  $$ p_{k}^{\redl} = \frac{1}{x_{c}} \frac{q^{2}}{ Z(x_{c})} \cdot x_{c}  \mathfrak{B}_{k}^{\redl} \frac{1}{k} q^{k-2}.$$ 
  Multiplying all contributions gives precisely \eqref{eq:lawsimplygenerated}, because the factors $q$ cancel due to the relation $$\sum_{ \begin{subarray}{c}u \in \mathfrak{t}\\ u \mathrm{ \ black} \end{subarray}} ( \mathrm{deg}^{\mathrm{b}}(u)-2)=-2.$$
  This proves the equality in law.   
  \endproof

   \begin{remark} Notice that the law of $\pi( \btau)$ as described in \eqref{eq:lawsimplygenerated} does not depend on the location of the root red leaf. In a sense, such laws are more symmetric than standard Galton--Watson measures. In our case, this is the modification of the $p$-Galton--Watson measure at the origin which makes the law of $ \mathrm{Red}^{\redl}( \mathfrak{T})$ invariant under rerooting at a uniform red leaf. Such modification of Galton--Watson trees to ensure an invariance property under rerooting appears in many places in the literature, see the pioneer work of Aldous on fringe subtrees \cite{Ald91c} and e.g.~\cite{MP02} or \cite[Section 4]{BJ2020}. \end{remark}
   
   \subsubsection{Blob labeling} \label{sec:labels}

According to Proposition \ref{prop:decomp}, if we want to recover the labeled tree $ \btau$ under $  \mathbb{P}_{0}$ from $ \mathrm{Red}^{\redl}( \mathfrak{T})$ we need, independently for each black vertex of $ \mathrm{Red}^{\redl}( \mathfrak{T})$ with $k$ neighbors, to replace it with a binary red tree with allowed angle assignments dual to the blob. 
 All those trees are attached according to the combinatorial structure of $\mathrm{Red}^{\redl}( \mathfrak{T})$ and produce a global angle assignment because of the compatibility at the cut-edges. Let us record this operation in a definition:
 
 \begin{definition}[From bicolored to red labeled trees] \label{def:blob} If $ \mathfrak{t}$ is a bicolored blob tree we write 
  $$ \textbf{Blob}( \mathfrak{t})$$ for the random tree with allowed angle assignments obtained by blowing up each black vertex into a red  tree with angle assignments dual to a uniform blob of the corresponding combinatorics  (degree and position of possible red leaf), and gluing the resulting trees together to get a red labeled tree (still rooted at the root of $ \mathfrak{t}$).
 \end{definition} 
 
 It is then clear from the blob construction and the proof of Proposition \ref{prop:decomp} that 
 \begin{eqnarray} \label{eq:lawblobtree} \textbf{Blob}( \mathrm{Red}^{\redl}(  \mathfrak{T})) \quad \overset{(d)}{=} \quad   \boldsymbol{\tau} \quad \mbox{ under } \mathbb{P}_{0}.  \end{eqnarray} These constructions are summarized in the following diagram:
 \begin{figure}[!h]
  \begin{center}
  \includegraphics[width=15cm]{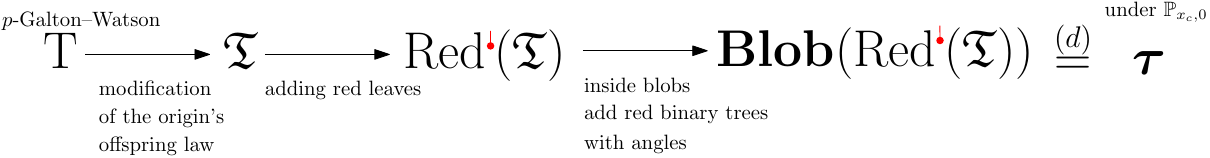}
  \end{center}
  \end{figure}
  
       \begin{remark} The construction of $ \btau$ from $ \mathfrak{T}$ by ``blowing up'' independently its black vertices is very similar to the notion of $ \mathcal{R}$-enriched trees of \cite{stufler2015random}. Notice however, that compared to those constructions, we shall also need to deal with the labeling.  \end{remark}
  
Recall from Section \ref{sec:euclidean}, that the angle assignments of $ \textbf{Blob}( \mathrm{Red}^{\redl}( \mathsf{t}))$ can equivalently be transformed into the $\ell$-labeling of its edges. We shall then project this labeling on the edges of the black tree $\mathsf{t}$ only (the reference edge labeled $0$ is carried by an imaginary edge attached to the root of $\mathsf{t}$). By construction, conditionally on the black tree $\mathsf{t}$, the increments of the $\ell$-labeling around each black vertex (different from the root) are independent\footnote{Of course, for any black vertex $v \in \mathsf{t}$, the existence of a red leaf attached to $v$ in $ \mathrm{Red}^{\redl}(\mathsf{t})$ and the label differences around that vertex are not independent} and we denote by $(\Delta_{i}(k) : 1 \leq i \leq k)$ their law. By reversal symmetry of blobs, these increments are \textbf{locally centered}: i.e. we have $\Delta_{i}(k) \overset{(d)}{=} - \Delta _{i}(k)$ for any $1 \leq i \leq k$, see Figure \ref{fig:deltaLik}.

\begin{figure}[!h]
 \begin{center}
 \includegraphics[width=16cm]{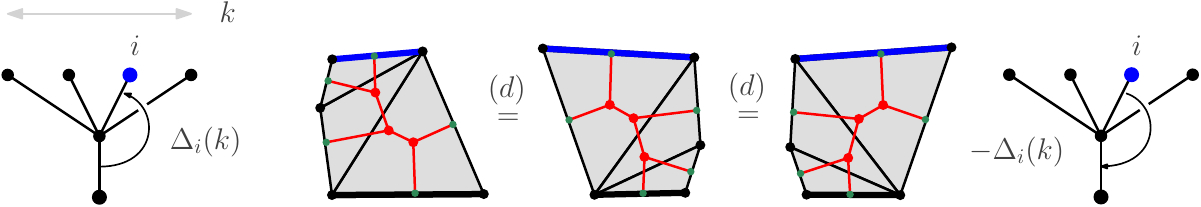}
 \caption{Illustration of the symmetry yielding to the local centering of the increments of $\ell$. The increments of $\ell$ correspond to (twice) the logarithm of the ratio between the corresponding side lengths in the tree of Euclidean triangles. \label{fig:deltaLik}}
 \end{center}
 \end{figure}
 
 The following lemma which builds upon Lemma \ref{lem:roughexpo} controls the maximal $\ell$-increments inside a blob:
\begin{lemma} \label{lem:momentslabels} Let $M_{k}$ be the maximal label in absolute value in a labeled red tree dual to a uniform blob of degree $k$ (with or without a red leaf). Then for some $ \mathrm{c}>0$ we have $$\sup_{k \geq 1} \mathbb{E}[ \exp( \mathrm{c} M_{k}/k)] < \infty.$$
\end{lemma}
\begin{proof} Under  $\mathbb{P}_{\pi/2}$  we can label the edges of the underlying red labeled tree by a subset of the full binary tree $ \mathbb{B}=\bigcup_{k \geq 0}\{0,1\}^{k}$. By Lemma \ref{lem:roughexpo} applied along any deterministic exploration the tree, the increments of the $\ell$-labels between two neighbor vertices in $ \mathbb{B}$ are stochastically dominated by independent r.v. having law $ \mathrm{Cst}(1+ \mathcal{E})$ where $ \mathcal{E}$ is a mean $1$ exponential variable. In particular, the probability that one of the $2^{h+1}-1$ vertices of combinatorial height less than $h \geq 0$ in $ \mathbb{B}$ gets a label larger than $\mathrm{Cst}(\lambda+1)h$ is upper bounded by 
  \begin{eqnarray*} (2^{h+1}-1) \cdot \mathbb{P}\left(\sum_{i=1}^{h} \mathrm{Cst}(1+\mathcal{E}_{i}) >  \mathrm{Cst}(\lambda+1)h\right) &\leq& 2^{h+1} \mathbb{P}\left(\sum_{i=1}^{h}\mathcal{E}_{i} > \lambda h\right)\\
  & \underset{ \mathrm{Markov. ineq.}}{\leq}  & 2^{h+1} \cdot \mathrm{e}^{-\lambda h/2}\mathbb{E}[  \mathrm{e}^{\mathcal{E}/2}]^{h} \\
   &=& 2(4 \mathrm{e}^{-\lambda/2})^{h}.  \end{eqnarray*}
In particular, if we condition on the event where the blob of the origin has degree $h$ (in particular all cut-edges of the blob are associated to vertices of the first $h$ generations of $ \mathbb{B}$) we get that 
$$ \mathbb{P}(  \mbox{max label in a uniform blob of degree $h$ is } \geq \mathrm{Cst}(\lambda+1)h) \leq \frac{2(4 \mathrm{e}^{-\lambda/2})^{h}}{ \mathbb{P}_{\pi/2}( \mbox{root blob has degree $h$})}.$$
Recalling from Proposition \ref{prop:offspringp} that $\mathbb{P}_{\pi/2}( \mbox{root blob has degree $h$})$ decreases like $(q/(4x_{c}))^{h+o(h)} \approx 0.4805^{h}$, one then takes $\lambda$ large enough so that $4 \mathrm{e}^{-\lambda} \ll 0.4805$. We deduce in particular that the maximal label in a blob of degree $h$, once renormalized by $h$ has some exponential moments as desired. 
\end{proof}

\subsection{Back to the local limits}
We now revisit Theorem \ref{thm:localleaf} through the lens of the underlying the monotype Galton--Watson tree given by the blob tree.

\subsubsection{Modified Kesten's tree}
 The $p$-Galton--Watson tree conditioned to be infinite, sometimes called the $p$-Kesten's tree, is the random plane tree $ \mathrm{T}_\infty$ obtained as the genealogical tree of a population made of two sorts of particles: standard particles and mutant particles. The standard particles reproduce independently with law $(p_{k})_{k\geq 0}$. Initially, the ancestor of the tree is a mutant particle. The mutant particles reproduce according to the size biased version of $p$, i.e.~$$ \hat{p}_{k} = k p_{k} \quad \mbox{ for } k \geq 1.$$ The measure $\hat{p}$ is indeed a probability distribution because $p$ is critical  (Proposition \ref{prop:offspringp}).  Among the children of each mutant, we pick one uniformly at random and declare it to be the mutant of the next generation (hence, all other children are standard particles). We shall denote by $ \mathfrak{T}_\infty$ the modified version of $ \mathrm{T}_\infty$ where the offspring distribution of the ancestor mutant, instead follows the size-biased version  of \eqref{eq:distribroot} 
 \begin{align*}
\hat{{p}}^{\redl}_k = \hat{c}\,k\, p_{k}^{\redl}, \quad \mbox{ for }k \geq 1 \quad \mbox{ where}\quad \hat{c} = \frac{1}{q} Z(x_{c})J_0(\pi \sqrt{Z(x_c)}) = \frac{c_0\,J_0(c_0/2)}{4J_1(c_0/2)}=0.80729\ldots
 \end{align*}
where $\hat{c}$ is the normalization constant so that it is a probability measure (the mutant at generation $1$ is then also picked uniformly among the children of the root). The normalization is computed using 
 \begin{align}
 Z(x_c) \sum_{k\geq 1} k p^{\redl}_k\underset{\eqref{eq:distribroot}}{=}\sum_{k \geq 1} \mathfrak{B}_{k}^{\redl} q^{k} &= z \partial_{z} B^{\redl}_{1}(z;0)\big|_{z=q}\underset{ \mathrm{Prop.}\ \ref{prop:blobredleaf}}{=} q Z'(q/4) \label{eq:phatrootcalc}\\
 & \underset{ \eqref{eq:ZJ}}{=}  \frac{q}{J_{0}( 2\pi \sqrt{Z(q/4)})} \underset{ \mathrm{Prop.}\ \ref{prop:blobredleaf}}{=} \frac{q}{J_{0}( \pi \sqrt{Z(x_{c})})}.\nonumber  \end{align}
 In particular $ \mathfrak{T}_{\infty}$ almost surely has one spine. Conditionally on $ \mathfrak{T}_\infty$, we can apply the same treatment as in Definitions  \ref{def:reds} and \ref{def:blob} (adding red leaves, then expanding black vertices into red labeled trees, recall the special treatment at the root vertex) and obtain $  \mathrm{Red}^{\redl}( \mathfrak{T}_\infty)$ as well as $ \textbf{Blob}( \mathrm{Red}^{\redl}( \mathfrak{T}_\infty))$.

\begin{proposition} \label{prop:local1} The law of $ \textbf{Blob}( \mathrm{Red}^{\redl}( \mathfrak{T}_\infty))$ is the law $ \mathbb{P}_{0}^{\infty}$ introduced in Theorem \ref{thm:localleaf}.
\end{proposition}

\begin{proof} We shall directly prove that $ \textbf{Blob}( \mathrm{Red}^{\redl}( \mathfrak{T}_\infty))$ is the local limit of $ \boldsymbol{\tau}$ under $ \mathbb{P}_{0}^{n}$ as $n \to \infty$. Since the labels are recovered by the \textit{same local procedure} independently for each vertex, by \eqref{eq:lawblobtree} it suffices to prove that 
$$  \big(\mathrm{Red}^{\redl}( \mathfrak{T}) \mid   \{ n \mbox{ leaves}\}\big) \xrightarrow[n\to\infty]{(loc)} \mathrm{Red}^{\redl}( \mathfrak{T}_{\infty}),$$ where the local convergence holds for bi-colored tree. This is almost a direct application of the literature \cite{AD14,St14,BJ2020}, unfortunately we cannot apply directly the results \footnote{If we see $ \mathrm{Red}^{\redl}( \mathfrak{T})$ as a multi-type Galton--Watson tree, then the irreducibility assumption of \cite{St14} fails (red vertices have no offspring) and we cannot directly apply the results of \cite{AD14} dealing with monotype Galton--Watson tree (our conditioning is on the number of red vertices).}  but we can apply the same proof and let us sketch the argument for the reader's convenience.

\begin{figure}[!h]
 \begin{center}
 \includegraphics[width=8cm]{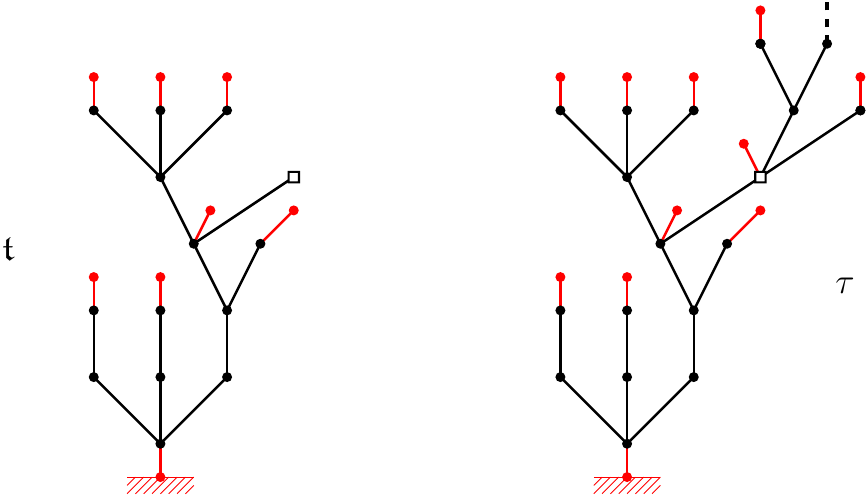}
 \caption{Illustration of the definition of $\mathfrak{t} \subset \tau$.}
 \end{center}
 \end{figure} Fix a finite bicolored tree $\mathfrak{t}$ rooted on a red leaf  satisfying the conditions described in the opening of  Section \ref{sec:piT}  and with a distinguished (non red) leaf $\smallsquare$. We write $\mathfrak{t} \subset  \mathfrak{t}'$ if the tree $\mathfrak{t}'$ can be obtained from $\mathfrak{t}$ by grafting a bicolored tree with a black root vertex on $\smallsquare$. Since all the random trees involved here are either finite or have one spine (one-ended tree), to prove the above convergence, it is sufficient to prove that for any $\mathfrak{t}$ as above we have 
  \begin{eqnarray} \label{eq:goallocal1} \mathbb{P}( \mathfrak{t} \subset \mathrm{Red}^{\redl}( \mathfrak{T}) \mid \{ n \mbox{ leaves }\}) \xrightarrow[n\to\infty]{} \mathbb{P}( \mathfrak{t} \subset \mathrm{Red}^{\redl}( \mathfrak{T}_{\infty})),  \end{eqnarray}
see \cite[Lemma 2.1]{AD14} which is straightforwardly extended to our multi-type case. We compute exactly the right-hand side of the last display by noticing that $\smallsquare$ must be the mutant particle in its generation and we find recalling the notation $ \mathrm{deg}^b(u)$ for the degree of $u$ in the black tree
$$\mathbb{P}( \mathfrak{t} \subset \mathrm{Red}^{\redl}( \mathfrak{T}_{\infty})) =\frac{\hat{p}^{\redl}_{\mathrm{deg}^b(1)}}{\mathrm{deg}^b(1)} \cdot \prod_{ \begin{subarray}{c} u \in \mathfrak{t} \setminus \{1, \smallsquare\}\\ \text{black} \end{subarray}} p_{\mathrm{deg}^b(u)-1} \cdot \left\{\begin{array}{ll} r_{\mathrm{deg}^b(u)-1} \frac{1}{\mathrm{deg}^b(u)}  & \mbox{ if } u \mbox{ has a red leaf},\\
1 -r_{\mathrm{deg}^b(u)-1} & \mbox{otherwise,}
\end{array}\right. $$ where the vertex $1$ is the unique black neighbor of the root leaf $\varnothing$. For $n$ large enough, in particular larger than the number $n_{0}$ of red leaves of $\mathfrak{t}$, we can compute similarly   \begin{eqnarray*}&& \mathbb{P}(  \{\mathfrak{t} \subset \mathrm{Red}^{\redl}( \mathfrak{T})\}  \cap \{ n \mbox{ leaves} \})\\ &=& \frac{\mathrm{deg}^b(1) p^{\redl}_{\mathrm{deg}^b(1)}}{{\hat{p}}^{\redl}_{\mathrm{deg}^b(1)}} \mathbb{P}( \mathfrak{t} \subset \mathrm{Red}^{\redl}( \mathfrak{T}_{\infty})) \cdot  \mathbf{P}( \mbox{the tree grafted on } \smallsquare \mbox{ has } n-n_{0} \mbox{ red leaves}),  \end{eqnarray*}
where under the probability $ \mathbf{P}$ the underlying random tree is a $p$-Galton--Watson tree starting with a black vertex to which the $ \mathrm{Red}$ procedure is applied (without special treatment at the root, recall Definition \ref{def:reds}). This actually corresponds to the law of a tree under the law  $ \mathbb{P}_{\pi/2}$ instead of $ \mathbb{P}_{0}$, starting with a cut-edge in the blob decomposition point of view.  According to \eqref{eq:asymptgeneral}
\begin{align*}
\frac{\mathbf{P}( \mbox{the tree grafted on } \smallsquare \mbox{ has } n-n_{0} \mbox{ red leaves})}{ \mathbb{P}(  n \mbox{ leaves})} \xrightarrow{n\to\infty} \frac{c_0}{4}\frac{J_0(c_0/2)}{J_1(c_0/2)} = \hat{c}.
\end{align*}
Using $\hat{p}_k^{\redl} = \hat{c}\, k\,p^{\redl}_k$, the ratio works out nicely and  \eqref{eq:goallocal1} holds. This finishes the proof of the theorem.  
\end{proof}

\subsubsection{Aldous' sin tree seen from a vertex}
To establish the Benjamini--Schramm convergence of the random punctured spheres (Theorem \ref{thm:BS}), it will also be necessary to prove local convergence of the tree $ \bT_{n}^{\star}$ obtained from $ \bT_{n}$ by rerooting at a uniform \textit{vertex}. To be precise, conditionally on the red labeled tree $ \bT_{n}$ (of law $ \mathbb{P}^{n}$) we sample a uniform red vertex $\reddot$, and then a uniform corner of $\reddot$ where to reroot the tree. We shift the labels so that the label of the root edge (the edge following the root corner in clockwise order) is equal to $0$ and write $ \bT^{\star}_{n}$ for the labeled tree we obtain that way. We show in this section (Corollary \ref{cor:aldous} below) that $\bT_{n}^{*}$ also converge locally. 

It is part of the folklore that (multitype) large critical Galton--Watson trees rerooted at a random uniform vertex converge towards an infinite random tree obtained as a slight variation on Kesten's infinite random tree, see the seminal work of Aldous \cite{Ald91a} or \cite{JA16,stufler2019local,stufler2019rerooting,BJ2020}. Actually, in the case of random binary trees with arbitrary laws, the local convergence in Corollary \ref{cor:aldous} is a direct consequence of Theorem \ref{thm:localleaf}, see  \cite[Proposition 86]{Yvette} which is easily extended to our labeled version. However, the following proof gives an explicit \textit{description} of the law of $\boldsymbol{T}_{\infty}^{*}$ which may be useful for later applications (although we are not convinced this is the most elegant description of $ \boldsymbol{T}_{\infty}^{*}$):

 \begin{corollary}[Local limit from a uniform vertex] \label{cor:aldous} We have the local convergence 
 $$  \boldsymbol{T}^{*}_{n} \xrightarrow[n\to\infty]{(d)} \boldsymbol{T}_{\infty}^{*},$$ where $\boldsymbol{T}_{\infty}^{*}$ is obtained using the following recipe:
 \begin{itemize}
 \item Introduce the offspring distribution obtained by normalizing $$ \tilde{p}_k \propto \Big( \mathfrak{B}_k + x_c \frac{k+1}{k} \mathfrak{B}^{\redl}_k \Big) q^{k-2}, \quad \mbox{ for } k \geq 1.$$
Then introduce $ \tilde{r}_k =  \frac{x_c \frac{k+1}{k} \mathfrak{B}^{\redl}(k)}{\mathfrak{B}_k +x_c \frac{k+1}{k} \mathfrak{B}^{\redl}(k) }$. Finally define $ \widetilde{ \mathrm{Red}}$ the modified $\mathrm{Red}$ procedure as in Definition \ref{def:reds} except that a red leaf is attached to the root vertex of a black tree with $k$ children with probability $ \tilde{r}_k$ and where the red leaf is place among the $k+1$ corners uniformly at random (the tree stays rooted on the initial black vertex). 
\item Let $ \tilde{\mathrm{T}}_\infty$ be the $p$-Galton--Watson tree conditioned to survive where the root vertex has normalized distribution $\tilde{p}$.
\item Apply then the $ \widetilde{\mathrm{Red}}$ procedure to get an infinite bicolored tree whose root vertex $\varnothing$ is a black vertex. 
\item Next, bias the law by the function 
 $$ \frac{1}{ \mathrm{deg}(\varnothing)} \left(  \left(\mathrm{deg}( \varnothing)-2\right) + \mathbf{1}_{  \varnothing \mbox{\footnotesize \   has a red leaf attached}}\right),$$ where the degree $ \mathrm{deg}$ is computed in the bicolored tree $\widetilde{\mathrm{Red}}( \tilde{\mathrm{T}}_\infty)$.
 \item Blow-up each black vertex into a Blob by applying $ \mathbf{Blob}$.
 \item Pick a red vertex inside the blob of $\varnothing$ or the unique leaf dangling from it (if any) uniformly at random and  root the tree at a uniform corner of it. You get a labeled tree having the law of $\bT_\infty^*$.
 \end{itemize}
 \end{corollary}

 \begin{proof} We first translate the effect of pointing $\bT_n$ at a uniform vertex in the associated blob-tree $\pi(\bT_n)$.   Conditionally on the labeled tree $  \bT_n$ sample a uniform red vertex $ \mathfrak{r}$ of it, and let us consider the black vertex $ \mathfrak{b}$ ``associated'' to $ \mathfrak{r}$ in the blob-tree $\pi(\bT_n)$: this is either the black vertex ``carrying'' that vertex in its associated blob, or the immediate neighboring black vertex if $ \mathfrak{r}$ is a red leaf. At the level of the black/red tree $ \pi( \bT_{n})$ the vertex $ \mathfrak{b}$ is thus chosen proportionally to 
 \begin{eqnarray} \label{def:phi} \phi( \mathfrak{b})=  \left(\mathrm{deg}( \mathfrak{b})-2\right) + \mathbf{1}_{  \mathfrak{b} \mbox{\footnotesize \   has a red leaf attached}} ,  \label{eq:weightsampling}\end{eqnarray} where  $ \mathrm{deg}( \mathfrak{b})$ is the number of \textit{black or red}  neighbors of $ \mathfrak{b}$. This is so because the blob of a black vertex with $d$ sides contains $d-2$ red inside vertices.  Once this black vertex is distinguished, we sample as above a uniform random corner of it and reroot the blob tree  there. We denote by $\pi( \bT_{n})^{\star}$ the random bicolored plane tree obtained this way, see Figure \ref{fig:rerootblobtree}.
 
 \begin{figure}[!h]
  \begin{center}
  \includegraphics[width=16cm]{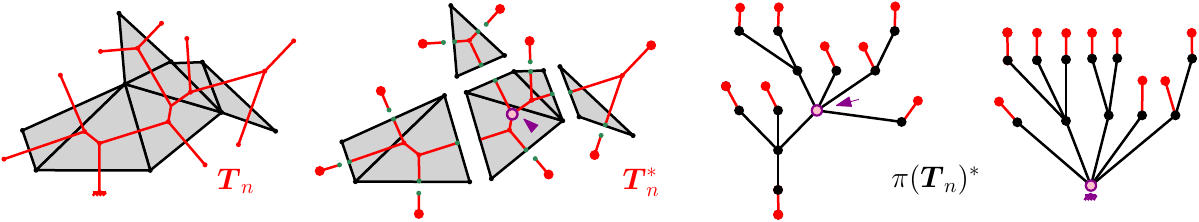}
  \caption{Illustration of the definition of $\pi( \bT_{n})^{\star}$: the sampled red vertex in $\bT_{n}^{*}$ is circled with purple as well as its associated black vertex in $\pi( \bT_{n})$. \label{fig:rerootblobtree}}
  \end{center}
  \end{figure}

 \begin{lemma} Consider a $p$-Galton--Watson tree $\tilde{\mathrm{T}}$ where the offspring distribution of the root vertex is changed to the normalized version of $\tilde{p}$. Then the law of $\pi( \bT_{n})^{\star}$ is absolutely continuous with respect to that of $ \widetilde{\mathrm{Red}}( \tilde{\mathrm{T}})$ conditioned on having $n$ leaves with Radon--Nikodym derivative proportional to 
 $$\frac{1}{ \mathrm{deg}(\varnothing)} \left(  \left(\mathrm{deg}( \varnothing)-2\right) + \mathbf{1}_{  \varnothing \mbox{\footnotesize \   has a red leaf attached}} \right),$$
where $\varnothing$ is the root vertex and where the degree is counted in the black/red tree.
 \end{lemma}
\begin{proof}[Proof of the lemma.] Fix a black/red tree $ \mathfrak{t}$ with $n$ red leaves rooted on a black vertex, and let us compute the probability that such a tree appears as  $\pi( \bT_{n})^{\star}$. There are exactly $n$ scenarios for this, coming from the $n$ possible initial locations  of the root red leaf of $ \bT_{n}$. Notice that the weight of the black/red tree is actually independent of the location of such a red leaf, and is given by  \eqref{eq:lawsimplygenerated} which we recall:
\begin{align*}\mathbb{P}_{0} (  \pi(  \boldsymbol{\tau})=\mathfrak{t}) = \frac{1}{x_{c}}\cdot\frac{1}{Z(x_c)} \prod_{ \begin{subarray}{c}u \in \mathfrak{t}\\ u \mathrm{ \ black} \end{subarray}} \left\{\begin{array}{ll} \mathfrak{B}_{ \mathrm{deg}^{ \mathrm{b}}(u)} & \mbox{ if } u \mbox{ has no red leaf}\\
x_c \mathfrak{B}_{ \mathrm{deg}^{ \mathrm{b}}(u)}^{\redl} \cdot \frac{1}{ \mathrm{deg}^{ \mathrm{b}}(u)} & \mbox{ if } u \mbox{ has a red leaf}.
\end{array}\right. \end{align*}
It is then the same exercise as in the proof of Proposition \ref{prop:decomp} to check that the above product is proportional to the probability that $ \widetilde{ \mathrm{Red}}( \tilde{T})$ lands on $ \mathfrak{t}$. Notice in particular the subtle difference between $\tilde{p}_k$ and $p_{k-1}$ as in Definition \ref{def:pk}: In the current setting, the blob corresponding to the root black vertex of $ \mathfrak{t}$ may be rooted on any side (and not necessarily on a cut-edge). In particular $x_c \frac{k+1}{k} \mathfrak{B}^{\redl}_k$ is the generating function of blobs having a red leaf, $k$ cut edges and where the root can be any of the $k+1$ sides of the blob. Once the initial root red leaf is picked, it remains to choosing the correct black vertex $ \mathfrak{b}$ with probability proportional to $\phi( \mathfrak{b})$ and the desired root corner of $ \mathfrak{b}$ with probability $ \frac{1}{ \mathrm{deg}( \mathfrak{b})}$. The lemma follows. \end{proof}

Using the same arguments as in the proof of Proposition \ref{prop:local1}, it follows that $ \widetilde{\mathrm{Red}}( \tilde{\mathrm{T}})$ conditioned on having $n$ red leaves and biased by $ \frac{ \phi( \varnothing)}{ \mathrm{deg}( \varnothing)}$ converges locally towards the random tree $ \widetilde{\mathrm{Red}}( \tilde{\mathrm{T}}_{\infty})$ biased by the random variable $ \frac{ \phi( \varnothing)}{ \mathrm{deg}( \varnothing)}$. To deduce the convergence of $ \bT_n^\star$ we just need to apply independent Blob constructions at each vertex, pick a uniform red vertex associated to the black root of ${\mathrm{Red}}(\tilde{\mathrm{T}}_{\infty})$, and reroot at a uniform random corner of it. \end{proof}

\section{Scaling limits}

 In this section we establish the scaling limits for the contour functions coding the labeled tree $ \boldsymbol{{T}}_{n}$ of law $ \mathbb{P}_{0}^{n}$ when $n$ is large. This will be the main ingredient for the proof of the convergence towards the Brownian sphere (Theorem \ref{thm:GH}).   Recall from Proposition \ref{prop:hypeucl} that the allowed angle assignment of $  \boldsymbol{{T}}_{n}$ enables us to see it as a labeled tree $(T_{n}, \ell)$ where $\ell$ is a function defined on the $2n-3$ edges of $T_{n}$. 
 
 \paragraph{Coding by functions.} Fix a labeled binary tree $\btau_n$ with $n$ leaves. We imagine that a particle traces the contour of the combinatorial (red) tree $\tau_{n}$ at unit speed starting from the left of the root vertex.  As it travels along the tree during the time interval $t \in [0,4n-6]$, we can record the height $\textbf{C}_{\btau_n}(t)$ of the particle as well as the label $\textbf{Z}_{\btau_n}(t)$ of the edge it is visiting (the definition is unambiguous  when the particle is not at a vertex, but we extend it by choosing once and for all for each vertex the value of one of its adjacent edge\footnote{With this choice, the functions may not be c\`adl\`ag, but this is no big deal since they will converge in the scaling towards continuous functions.}). 
 Let also $\textbf{R}_{\btau_n}$ be the function counting the number of leaves encountered so far in the contour. These functions are defined over $[0, 4n-6]$.

 \begin{figure}[!h]
  \begin{center}
  \includegraphics[width=15cm]{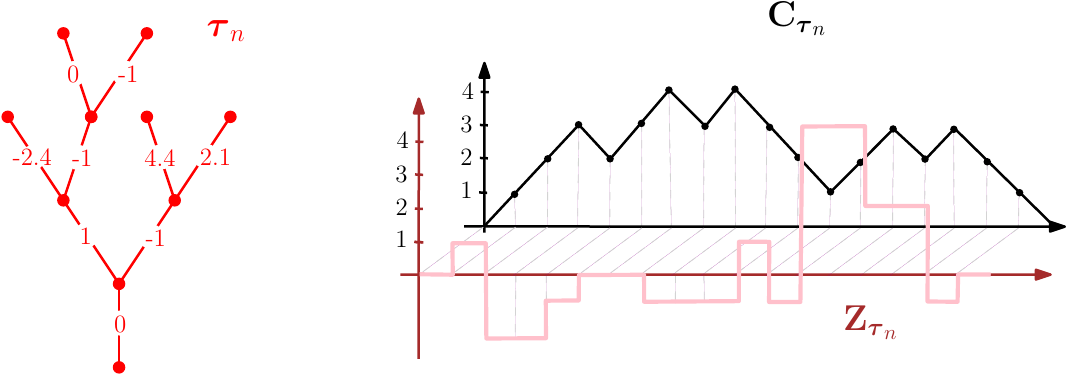}
  \caption{Illustration of the definition of $\textbf{C}_{\btau_n}$ and $\textbf{Z}_{\btau_n}$.}
  \end{center}
  \end{figure}
Our goal in this section is the following scaling limit result:
\begin{theorem}[Convergence of coding functions] \label{thm:scaling}
We have the following convergence in distribution for the uniform topology
 \begin{eqnarray} \label{eq:convergences} \left(\frac{\textbf{C}_{\bT_n}(t \cdot (4n-6))}{ \sqrt{n}},\frac{ \textbf{Z}_{\bT_n}(t \cdot (4n-6))}{  \sqrt{\sqrt{n}}},\frac{\textbf{R}_{\bT_n}(t \cdot (4n-6))}{n} \right)_{0 \leq t \leq 1} \xrightarrow[n\to\infty]{(d)} \left( c_{1} \mathbf{e}_{t}, c_{2}  Z_{t},  t\right)_{0 \leq t \leq 1},  \end{eqnarray}
where $( \mathbf{e}, Z)$ is a normalized Brownian excursion and the head of the Brownian snake driven by $ \mathbf{e}$ and 
$$c_{1}= \frac{2 |\cos c_0|}{J_1(c_0)}, \quad c_{2}= \frac{4\pi}{\sqrt{3 c_0}}.$$
\end{theorem}
The process $( \mathbf{e},Z)$ is the building block of the Brownian  sphere: the function $ \mathbf{e}$ is, heuristically, a Brownian motion conditioned to start and end at $0$ while staying non-negative over $[0,1]$. It can be seen as the ``continuous contour function'' of a random real tree $ \mathcal{T}_{\mathbf{e}}$, see \cite{DLG05}, while the process $Z$ can be interpreted as a Brownian motion indexed by the tree\footnote{The cautious reader might worry about confusing the head of the Brownian snake denoted $Z_{\cdot}$ with the function $Z(\cdot)$ appearing in Theorem \ref{thm:zograf}; however, since the latter will not be used in this section, there should be no risk of ambiguity.}. For precise definitions, we refer e.g.~to \cite{MM07,LG11,Mie11,le2019brownian}. Such convergences are the standard inputs for proving convergence to the Brownian sphere and are usually derived using invariance principles for labeled Galton--Watson trees, see \cite{MM07,Mie08c} and the recent work of \cite{addario2025discrete} for the state-of-the-art. Indeed, recall that thanks to Proposition \ref{prop:decomp} the labeled tree $ \boldsymbol{{T}}_{n}$ can be seen as the random tree $ \textbf{Blob}( \mathrm{Red}^{\redl}( \mathfrak{T}))$ conditioned on the event $\{   \mathrm{Red}^{\redl}(\mathfrak{T}) \mbox{ has }n \mbox{ red leaves}\},$ of probability of order $ n^{-3/2}$ by \eqref{eq:asymptgeneral}. As such, Theorem \ref{thm:scaling} should come as no surprise for the knowledgeable reader. However, in our case $ \bT_{n}$ is \textit{not} a monotype Galton--Watson tree and the increments of the labels are \textit{not} independent conditionally on the structure of the tree, furthermore,  the conditioning is on the number of leaves and not on the number of vertices. Those constraints impose us to produce an ad-hoc proof and take some bucolic path collecting  folklore results and methods from the theory of random trees.  The main steps of the proof of this convergence are given in the  next five subsections, using as key input the (modified) monotype Galton--Watson structure of the blob tree. To lighten the prose, we shall first prove a version of Theorem \ref{thm:scaling} with implicit scaling constants $c_{1},c_{2}$ and finally compute them from Proposition \ref{prop:spine} in Section \ref{sec:scalingcst}.

\subsection{Step 1: From contour to lexicographical exploration}
Although contour functions are convenient in the construction and proof of the convergence towards the Brownian sphere, they are not well suited from a probabilistic point of view because of their ``backtracking'' effect. In the literature on random trees one usually prefers the lexicographic exploration of a plane tree obtained by visiting the vertices of the tree in the depth first order. More precisely, if $  \btau_n= ( \tau_n, \ell)$ is a binary labeled tree with $n$ leaves, rooted on a red leaf, we shall denote by $\varnothing = u_{1} \prec u_{2} \prec... \prec u_{2n-2}$ the vertices of the tree listed in lexicographic order, we write $ \mathrm{ht}( u_{i})$ for the distance of $u_{i}$ to the root vertex and $ \ell( u_{i})$ to be the label of the edge just below $u_{i}$ with the convention that $\ell( \varnothing)=0$ (recall that in our setup, the labels are carried by the edges of the tree). Then the height and label processes are by definition 
$$ \mathcal{H}_{\btau_n}( k) = \mathrm{ht}(u_{k}), \qquad \mathcal{Z}_{\btau_n}(k)=  \ell( u_{k}),$$ and similarly we put $ \mathcal{R}_{\btau_n}( k)$ for the number of leaves encountered so far at the $k$th step in the lexicographical exploration.  Those functions are extended to continuous time by linear interpolation. In order to establish Theorem \ref{thm:scaling} we shall prove the corresponding version of  \eqref{eq:convergences} for the lexicographic exploration:
 \begin{eqnarray} \label{eq:goallexico} \left(\frac{\mathcal{H}_{\bT_n}( t \cdot (2n-2))}{ \sqrt{n}}, \frac{\mathcal{Z}_{\bT_n}( t \cdot (2n-2))}{ \sqrt{ \sqrt{n}}}, \frac{\mathcal{R}_{\bT_n}( t \cdot (2n-2))}{n} \right)_{0 \leq t \leq 1} \xrightarrow[n\to\infty]{(d)} \left( c_{1} \mathbf{e}_{t}, c_{2}  Z_{t}, t\right)_{0 \leq t \leq 1}.  \end{eqnarray}
Indeed, it is then classic that for random trees of  height of order $ o_{\mathbb{P}}(n)$, the contour exploration roughly follows the lexicographic process speed up by a factor $2$. More precisely, given \eqref{eq:goallexico}, the convergence of the first component in \eqref{eq:convergences} is implied by  Theorem 5.3 in \cite{marckert2019iterated}, and the proof adapts to the triplet of processes, see   \cite[Corollary 26]{MM07}. We shall now focus on proving the lexicographic convergences \eqref{eq:goallexico}.

\subsection{Step 2: Marckert--Miermont for the black tree}
To prove  \eqref{eq:goallexico} we first consider the random labeled tree obtained as follows: Let $ \mathrm{T}$ be a $p$-Galton--Watson tree. Conditionally on $ \mathrm{T}$, recall from Definition \ref{def:reds} the construction of $ {\mathrm{Red}}( \mathrm{T})$ and from Definition \ref{def:blob} that of $  \textbf{Blob}( {\mathrm{Red}}( \mathrm{T}))$.
We focus on the labeled tree 
 \begin{eqnarray} \label{eq:projectlabel}\boldsymbol{ \mathrm{T}} = ( \mathrm{T},\ell),  \end{eqnarray}
obtained by projecting the labels of $ \textbf{Blob}( \mathrm{Red}( \mathrm{T}))$ onto $ \mathrm{T}$ as explained in Section \ref{sec:labels}. To be precise, each edges of the black tree is canonically identified with an edge of the underlying red tree $\textbf{Blob}( \mathrm{Red}( \mathrm{T}))$ from which it receives its label (the reference label $0$ is carried by an imaginary root edge attached under the root of $ \mathrm{T}$).  Recalling Proposition \ref{prop:offspringp} and  Lemma \ref{lem:momentslabels}: 
\begin{itemize}
\item the tree $ \mathrm{T}$ is a monotype critical Galton--Watson with exponential moments,
\item conditionally on it, the $\ell$ increments around each vertex are locally centered  (beware, the labels increments are not locally centered if we condition on $ \mathrm{Red}( \mathrm{T})$ instead).
\end{itemize}

As in the preceding section, we perform the lexicographical exploration of $ \boldsymbol{ \mathrm{T}}$ and record the height function $ \mathcal{H}_{\boldsymbol{ \mathrm{T}}}$ and the label function $ \mathcal{Z}_{\boldsymbol{ \mathrm{T}}}$. The following proposition is a consequence of the work \cite{MM07} (notice the conditioning by the number of vertices of the black tree $ \mathrm{T}$):

\begin{proposition}  \label{prop:MM} We have the convergence under $\mathbb{P}( \cdot \mid \#  \mathrm{T} = n)$
 \begin{eqnarray} \label{eq:conv} \left( \frac{ \mathcal{H}_{ \boldsymbol{ \mathrm{T}}}(nt)}{ \sqrt{n}},\frac{ \mathcal{Z}_{\boldsymbol{ \mathrm{T}}}(nt)}{  n^{1/4}}%,\frac{ \mathcal{R}_{\boldsymbol{T}}(nt])}{  n}
 \right)_{0 \leq t \leq 1} &\xrightarrow[n\to\infty]{(d)}&  \left( c_{3}\cdot  \mathbf{e}_{t}, c_{4} \cdot Z_{t}\right)_{0 \leq t \leq 1} ,  \end{eqnarray}
for the uniform topology on $ \mathcal{C}([0,1],  \mathbb{R}^{2})$, where $c_{3},c_{4}$ are positive constants.\end{proposition}

\begin{proof} The convergence follows from (a monotype version of)  \cite[Theorem 8]{MM07}: The required hypotheses on the offspring distribution (exponential moments on $p$) and bounded $4+ \varepsilon$ moments for the label displacements are given respectively by Proposition \ref{prop:offspringp} and Lemma \ref{lem:momentslabels}.  \end{proof}

The constants $c_{3}, c_{4}$ will be related to $c_{1},c_{2}$ in Section \ref{sec:scalingcst}.

\subsection{Step 3: Stretching the black tree to the red one}

In this third step, we push the convergence of Proposition \ref{prop:MM} onto $ \textbf{Blob}( \mathrm{Red}( \mathrm{T}))$ which we will denote by $ {\boldsymbol{ \mathrm{B}}}$ for notational convenience. More specifically, let $ \mathcal{H}_{{\boldsymbol{ \mathrm{B}}}}, \mathcal{Z}_{{\boldsymbol{ \mathrm{B}}}}$ and $ \mathcal{R}_{{\boldsymbol{ \mathrm{B}}}}$ be the height, label and counting function of the red leaves during the lexicographical exploration of $ {\boldsymbol{ \mathrm{B}}}$. Recall that conditionally on $ \mathrm{T}$, each vertex of $ \mathrm{T}$ with $k$ children gets a red leaf with probability $r_{k}^{\redl}$ independently of each other so if we introduce 
 \begin{eqnarray} \label{eq:defalpha} \alpha = \sum_{k \geq 0}p_{k} \cdot r_{k} \underset{\mathrm{Def.}\ref{def:pk}}{=} x_c \sum_{k \geq 1} \mathfrak{B}_k^{\redl} q^{k-2} \underset{\eqref{eq:phatrootcalc}}{=} \frac{J_1(c_0)}{2 J_1(\tfrac{c_0}{2})J_0(\tfrac{c_0}{2})} = 0.776628\ldots,  \end{eqnarray} 
 for the probability that a typical vertex in a $p$-Galton--Watson tree receives a red leaf, then  when $\# \mathrm{T}=n$ we shall have $ |{{ \mathrm{B}}}| \approx \alpha n$ where we recall that $ |\tau|$ is the number of leaves of $\tau$  and so $ \# \mathrm{B} \approx 2 \alpha n$ since $ \mathrm{B}$ is binary. More precisely we will prove:
\begin{proposition} \label{prop:stretched} With the same notation as in Proposition \ref{prop:MM} we have the following convergence  under $ \mathbb{P}( \cdot \mid \#  \mathrm{T} = n)$: we have $n^{{-1}}\# \mathrm{B} \to 2\alpha$ in probability and 
 \begin{eqnarray} \label{eq:convB}  \left( \frac{ \mathcal{H}_{{\boldsymbol{ \mathrm{B}}}}(t \, \# \mathrm{B})}{ \sqrt{n}},\frac{ \mathcal{Z}_{{\boldsymbol{ \mathrm{B}}}}(t \, \# \mathrm{B}))}{  n^{1/4}},\frac{ \mathcal{R}_{{\boldsymbol{ \mathrm{B}}}}(t \, \# \mathrm{B}))}{ n}\right)_{0 \leq t \leq 1} &\xrightarrow[n\to\infty]{(d)}&  \left(  \kappa \cdot c_{3} \cdot  \mathbf{e}_{t}, c_{4} \cdot Z_{t}, \alpha \cdot t\right)_{0 \leq t \leq 1} ,  \end{eqnarray}
for the uniform topology on $ \mathcal{C}([0,1],  \mathbb{R}^{3})$, where $\alpha$ is given in \eqref{eq:defalpha} and $\kappa >0$ is a positive constant.
\end{proposition}
The precise value of $\kappa$ will be established in Section \ref{sec:scalingcst}.
\begin{proof} To go from $ \boldsymbol{ \mathrm{T}}$ to $ { \boldsymbol{ \mathrm{B}}}=\textbf{Blob}( \mathrm{Red}( \mathrm{T}))$, we need to stretch the tree by inserting in each black vertex a red labeled tree of the correct distribution. We shall prove that this operation amounts to a simple time dilation in the exploration process in the scaling limits. Such effects have already been observed in \cite{CHKdissections,caraceni2016scaling,stufler2015random} but since we were unable to locate a precise statement to use in our context, we prefer going through the arguments for the reader's convenience. 

If $\tilde{u}_{k}$ is the $k$th vertex visited during the lexicographical exploration of $  { \boldsymbol{ \mathrm{B}}}$, we shall write $t(k)$ for the lexicographical index of the corresponding vertex in $ \boldsymbol{ \mathrm{T}}$ (beware $t(1),t(2),...$ is not necessarily increasing).  The common method to prove the forthcoming estimates is to show that they hold with a very high probability (large deviation events) under the unconditional Galton--Watson measure, so that even after  conditioning on $\{\# \mathrm{T}=n\}$ of ``polynomial probability'', they still hold with high probability. More precisely, consider the $p$-Galton--Watson tree $ \mathrm{T}_{n}$ conditioned on having  $n$ vertices and list its vertices $\varnothing= u_{1}, u_{2}, ...,u_{n}$ in the lexicographical order and denote by $k_{u_{i}}$ the number of children of $u_{i}$ in the tree. Then it is well known that the process $(k_{u_{i}}-1 : 1 \leq i \leq n)$ has the same law as i.i.d.~variables $(X_{1}, ... , X_{n})$ of law $(p_{k+1}: k \geq -1)$ conditioned on the event 
 $$ \{ X_{1} + \cdots + X_{n} =-1 \mbox{ and } \forall j \leq n-1, \ \ X_{1} + ... + X_{j} \geq 0\}.$$ Under our assumption, the law $p$ is critical, aperiodic, and has finite variance $\sigma^2$ so the above event has probability   \begin{eqnarray} \mathbb{P}(\#  \mathrm{T}=n) \sim  \frac{1}{ \sqrt{2\pi \sigma^2}} n^{-3/2}, \quad \mbox{ as }n\to\infty,   \label{eq:polyprob}\end{eqnarray} by Kemperman's formula and the local limit theorem, see e.g.~\cite[Chapters 3 and 4]{curien2024random}. Consequently, if $oe(n)$ is a function tending to $0$ quicker than any $n^{-k}$ for $k >0$,  every property $  \mathcal{P}_{n}$ that holds true with probability  $1 - oe(n)$ for $ \mathrm{T}$ also holds with probability  $1 - oe(n)$ for  its conditioned version $ \mathrm{T}_{n}$. Let us apply this principle:
\paragraph{Maximal degree and label displacement.} The maximal degree in $ \mathrm{T}_{n}$ satisfies  $$ \mathbb{P}( \max_{1 \leq i \leq n} \mathrm{deg}(u_{i}) \geq k+1 \mid \#  \mathrm{T}=n) \leq  \frac{n \mathbb{P}( X \geq k)}{ \mathbb{P}(\#  \mathrm{T}=n)},$$ where $X$ has law  $(p_{k+1}: k \geq -1)$. Thanks to Proposition \ref{prop:offspringp} this tends to $0$ as soon as $k = C \log n$ for some large constant $C>0$. Similarly, if $\Delta(u_{i})$ is the largest label displacement in the blob associated to $u_{i}$ in $ \textbf{Blob}( \mathrm{Red}( \mathrm{T}_n))$ then we have   \begin{eqnarray*} \mathbb{P}( \max_{1 \leq i \leq n} \mathrm{\Delta}(u_{i}) \geq k \mid \#  \mathrm{T}=n) &\leq&  \frac{n \mathbb{P}(  \Delta( \varnothing) \geq k)}{ \mathbb{P}(\# \mathrm{T}=n)} \\   & \underset{ \mathrm{Markov}}{\leq} &\frac{n}{ \mathbb{P}( \#\mathrm{T} =n)} \sum_{j=0}^{\infty}p_{j}  \frac{1}{(1 \vee k)^{d}} \cdot \mathbb{E}[\Delta(\varnothing)^{d} \mid \mathrm{deg}(\varnothing)=j]\\ &\underset{ \mathrm{Lem.} \ref{lem:momentslabels}}{\leq} & c' n^{5/2} \sum_{j=1}^{\infty} p_{j} \left(\frac{j}{k}\right)^{d} \sim c'' n^{5/2} k^{-d},  \end{eqnarray*} where in the last line we used the fact that $p$ has some exponential moments and where the constant $c''>0$ depends on $d \geq 1$ as well. If $ k = n^{ \varepsilon}$, we can take $d$ large enough and thus deduce that the largest label displacement inside a blob of $ \mathbf{B}$ is $o(n^{ \varepsilon} )$ for any $ \varepsilon >0$. In particular, since the vertex $u_{t(k)} \in  \mathrm{T}$ corresponds to the blob to which $\tilde{u}_{k} \in {\mathrm{B}}$ belongs then we have 
under $\mathbb{P}( \cdot \mid \#  \mathrm{T} = n)$ we have   \begin{eqnarray}  
\sup_{1 \leq k \leq \#\boldsymbol{ \mathrm{B}}} n^{-1/4} \cdot  | \mathcal{Z}_{{\boldsymbol{ \mathrm{B}}}}(k) - \mathcal{Z}_{{\boldsymbol{  \mathrm{T}}}}(t(k))| \xrightarrow[n\to\infty]{( \mathbb{P})} 0. \label{eq:labels}  \end{eqnarray}

\paragraph{Convergence of types.}  Let us prove  that the red leaves are asymptotically uniformly distributed in the lexicographical exploration of $ \mathrm{T}$. This is a general phenomenon known as  ``convergence of types'' lemma in multi-type Galton--Watson trees, see \cite[Section 4.5]{MM07}. The idea is the same as above:  If $R_{i}$ (resp. $\tilde{R}_{i}$) is the number of red leaves (resp. red vertices) of $    \mathbf{B}=\mathbf{Blob}(\mathrm{Red}( \mathrm{T}))$ emanating from the first $i$ vertices in lexicographical order of $  \mathrm{ T}$, then under the unconditional measure, $R_{i}$ (resp $ \tilde{R}_{i}$) is a sum of   i.i.d.~random variables having some exponential moments by Proposition \ref{prop:offspringp}. By standard large deviations results, $R_{[tn]}$ (resp. $\tilde{R}_{[tn]}$) is exponentially concentrated around its mean $\alpha tn$ (resp. $ \beta t n$), and the result still holds after conditioning  on the event $ \{\#  \mathrm{T}=n\}$ of polynomial probability by \eqref{eq:polyprob}.  In particular, we deduce that $| \mathrm{B}| \approx \alpha n$ and $\# \mathrm{B} \approx \beta n$ with high probability and since $R$ and $ \tilde{R}$ are increasing, it follows that  under $\mathbb{P}( \cdot \mid \#  \mathrm{T} =n )$
  \begin{eqnarray} \label{conv:types}  \left(\frac{R_{[nt]}}{n}\right)_{0 \leq t \leq 1} \xrightarrow[n \to \infty]{ ( \mathbb{P})} (\alpha \cdot t)_{t \in [0,1]} \quad \mbox{ and } \quad  \left(\frac{\tilde{R}_{[nt]}}{n}\right)_{0 \leq t \leq 1} \xrightarrow[n \to \infty]{ ( \mathbb{P})} (\beta \cdot t)_{t \in [0,1]} ,  \end{eqnarray} in the Skorokhod, in particular we must have $\beta = 2 \alpha$ since $ \mathrm{B}$ is binary.

\paragraph{Dilation of the height process.} We now show that the distances in $ \mathrm{T}$ and $ \mathrm{B}$ are roughly proportional at large scales. The proof is similar to \cite[Lemma 9]{CHKdissections}. Again, the idea is to look at the situation under the unconditional $p$-Galton--Watson measure. If $ \tau$ is a plane tree and $u \in \tau$, let $\varphi( \tau,u)$ be a  function depending only on the tree pruned at the vertex $u$. By the spine decomposition of a critical Galton--Watson tree (see \cite[Proposition 5]{CHKdissections}) we have 
  \begin{eqnarray} \label{eq:spinedecomp} \mathbb{E}\left[ \sum_{u \in \mathrm{T}} \varphi( \mathrm{T},u) \right] = \sum_{h \geq 0} \mathbb{E}[\varphi( \mathrm{T}_{\infty}, S_{h})],  \end{eqnarray}
where $ \mathrm{T}_{\infty}$ is Kesten's $p$-GW conditioned to survive and $S_{h}$ is the $h$-th vertex along its spine. Notice in particular that the height in $\mathbf{Blob}( \mathrm{Red}( \mathrm{T}_{\infty}))$ of $S_{h}$, that is the number of red edges strictly below the blob associated with $S_{h}$ in $  \mathbf{B}_{\infty}=\mathbf{Blob}( \mathrm{Red}( \mathrm{T}_{\infty}))$ is a sum of $h$ i.i.d.~random variables with exponential moments thanks to Proposition \ref{prop:offspringp}. It is in particular exponentially concentrated around $\kappa h$ where $\kappa$ is the mean number of red edges linking the root to a uniform point in the blob dual to a vertex of law $ \hat{p}$ (we count $1/2$ for red edges in-between blobs). We then apply the previous display to the function $$\phi( \mathrm{T},u) =  \mathbf{1} \left\{\left| \mathrm{ht}_{ \mathbf{Blob}( \mathrm{Red}( \mathrm{T}))}(u) - \kappa\cdot  \mathrm{ht}_{  \mathrm{T}}(u)  \right| \geq \varepsilon (\mathrm{ht}_{  \mathrm{T}}(u) \vee n^{ \varepsilon})  \right\}.$$
By  large deivation estimates we thus get 
$$ \mathbb{P}( \exists u \in \mathrm{T} :  \left| \mathrm{ht}_{ \mathbf{Blob}( \mathrm{Red}( \mathrm{T}))}(u) - \kappa \mathrm{ht}_{  \mathrm{T}}(u)  \right | \geq \varepsilon (\mathrm{ht}_{  \mathrm{T}}(u) \vee n^{ \varepsilon})) \underset{ \eqref{eq:spinedecomp}}{\leq} \sum_{h \geq 0}  \mathrm{e}^{-c_{ \varepsilon} (h \vee n^{\varepsilon})} \leq \mathrm{e}^{-c'_{ \varepsilon} n^{ \varepsilon}},$$ for $c_{ \varepsilon}, c'_{ \varepsilon}>0$. Taking $  \varepsilon = 1/3$, after passing to the polynomial conditioning on $ \{ \# \mathrm{T}=n\}$ we deduce using \eqref{eq:polyprob} that under $\mathbb{P}( \cdot \mid \#  \mathrm{T} = n)$ we have 
 \begin{eqnarray}  
\sup_{1 \leq k \leq \#\boldsymbol{ \mathrm{B}}} n^{-1/2} \cdot \left| \mathcal{H}_{{\boldsymbol{ \mathrm{B}}}}(k) -  \kappa \cdot \mathcal{H}_{\boldsymbol{ \mathrm{T}}}(t(k))\right| \xrightarrow[n\to\infty]{( \mathbb{P})}0. \label{eq:height}
 \end{eqnarray}

\paragraph{Completing the proof.} Armed with the rough bounds on the degrees and label displacements, together with \eqref{eq:labels}, \eqref{conv:types},\eqref{eq:height} and with Proposition \ref{prop:MM}, we can finish the proof of the proposition. Indeed, recall that for $ 1 \leq k \leq \# \mathrm{B}$ we denoted by $t(k)$ the lexicographical index of the blob in $\pi ( \mathbf{B})$ (the blob containing it, or immediately adjacent in the case of a leaf) of the vertex $\tilde{u}_k \in \mathrm{B}$. We suppose furthermore that $k$ is first visit of the blob in the lexicographical exploration of $ \mathrm{B}$ that is 
$$ \forall 1 \leq j < k, \quad t(j) \ne t(k).$$
We then claim that in such case we have 
$$ t(k) \approx  \frac{k}{2\alpha}.$$
To see this, examine the number of vertices visited in the lexicographical exploration of $ \mathrm{B}$ up to $\tilde{u}_k$. This number is equal to $k$ by the definition of the lexicographical exploration of $ \mathrm{B}$, but using the fact that this is the first visit of the blob so far, this number can be approximated by the number $ \tilde{R}_{t(k)}$ of vertices of $ \mathrm{B}$ discovered in the lexicographical exploration of $ \mathrm{T}$ up to time $t(k)$: the difference  only comes from the vertices of $ \mathrm{B}$ that are attached to the black vertices of the ancestral path yielding to $u_{t(k)}$ in $ \mathrm{T}$. But by \eqref{eq:conv}, this number is of order $ \sqrt{n}$ and thus negligible compared to $n$. Using \eqref{conv:types} we thus deduce that 
$$t(k) =  \frac{k}{2\alpha} + o_\mathbb{P}(n),$$ where the $o_\mathbb{P}(n)$ is uniform for all such $k$'s.  For those $k$'s one can then use \eqref{eq:labels}, \eqref{eq:height} and the asymptotic continuity of the limiting processes to deduce that 
$$  \frac{ \mathcal{H}_{{\boldsymbol{ \mathrm{B}}}}(k)}{ \sqrt{n}} \approx \kappa\cdot \frac{ \mathcal{H}_{{\boldsymbol{ \mathrm{T}}}}([k/(2\alpha)])}{ \sqrt{n}}, \quad \frac{ \mathcal{Z}_{{\boldsymbol{ \mathrm{B}}}}(k)}{  n^{1/4}}\approx \frac{ \mathcal{Z}_{{\boldsymbol{ \mathrm{T}}}}([k/(2\alpha)])}{  n^{1/4}}, \quad \mbox{ and } \quad \frac{ \mathcal{R}_{{\boldsymbol{ \mathrm{B}}}}(k)}{ n} \approx \frac{R_{\tilde{R}^{-1}(k)}}{n}\approx \frac{k}{2n}.$$
The case when $k$ does not correspond to the first visit of the associated blob is similar and is safely left to the reader. \end{proof}

\subsection{Step 3: Changing the conditioning}
 With Proposition \ref{prop:stretched} in our hands, we want  to lift this convergence to the conditional measure $\mathbb{P}( \cdot \mid  |\mathrm{Red}( \mathrm{T})| = n)$ where we recall that $|\tau|$ is the number of leaves of the plane tree $\tau$. For this we compare $\mathbb{P}( \cdot \mid | \mathrm{Red}( \mathrm{T})| = [\alpha  n])$ with $\mathbb{P}( \cdot \mid \#  \mathrm{T} = n)$. More precisely, we shall prove that the first $m$ steps of the lexicographical exploration of $ \mathrm{Red}( \mathrm{T})$ under $ \mathbb{P}(\cdot \mid \#\mathrm{T} = n)$ is absolutely continuous with respect to the same exploration of $ \mathrm{Red}(\mathrm{T})$ under $\mathbb{P}( \cdot \mid | \mathrm{Red}(\mathrm{T})| = [\alpha  n])$ and with a Radon--Nikodym derivative going to $1$ in probability as $n,m \to \infty$ with $m/n \to 1- \varepsilon$.
 
 Specifically, fix a bicolored tree $ \mathfrak{t}$ with $m$ black vertices which can be obtained as the first $m$ steps of the exploration of the black tree of an underlying bicolored tree. In particular, $ \mathfrak{t}$ contains the information about the additional red leaves  of the explored black vertices. See Figure \ref{fig:explom}. We write $ \mathfrak{h}$ for the number of unexplored  black vertices of $  \mathfrak{t}$ which corresponds to one plus the value of the height process at time $m$ and write $ \mathfrak{r}$ for the number of red leaves revealed in $\mathfrak{t}$. 
 
 \begin{figure}[!h]
  \begin{center}
  \includegraphics[width=6cm]{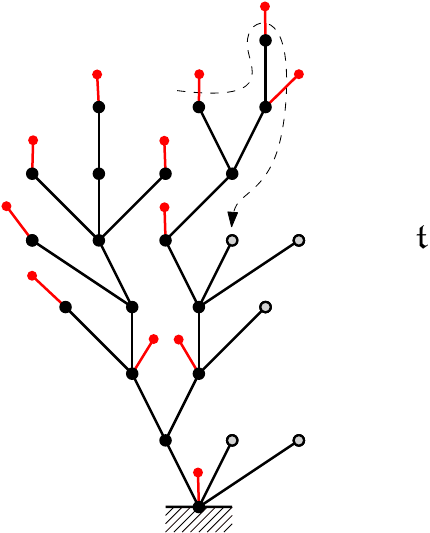}
  \caption{ \label{fig:explom}  We reveal the first $m=18$ black vertices during the lexicographical exploration of an underlying bicolored tree. The $ \mathfrak{h}=5$ gray vertices are still unexplored and we have already encountered $ \mathfrak{r}=12$ red leaves. }
  \end{center}
  \end{figure}
  In the following, we write $ \mathrm{E}_{m}( \mathfrak{t})$ for the result of the exploration of the first $m$ black vertices of the underlying bicolored tree $ \mathfrak{t}$. We also write $ \mathrm{T}_{1}, \mathrm{T}_{2}...$ for independent $p$-Galton--Watson trees. The following lemma should be combinatorially obvious:

\begin{lemma} With the above notation, for any compatible $ 1 \leq m < n$ and $r \geq  \mathfrak{r} \geq 0$ we have
 $$ \frac{\mathbb{P}\left(  \mathrm{E}_{m}( \mathrm{Red}( \mathrm{T})) =  \mathfrak{t} \mid | \mathrm{Red}( \mathrm{T})|= r\right)}{\mathbb{P}\left( \mathrm{E}_{m}( \mathrm{Red}( \mathrm{T})) =  \mathfrak{t} \mid \#  \mathrm{T}=n\right)} = \frac{ { \mathbb{P}(| \mathrm{Red}( \mathrm{T}_{1})| + ... + | \mathrm{Red}( \mathrm{T}_{ \mathfrak{h}})| =  r- \mathfrak{r})}}{ \mathbb{P}(\# \mathrm{T}_{1} + ... + \# \mathrm{T}_{ \mathfrak{h}} = n-m)} \cdot \frac{\mathbb{P}( \# \mathrm{T} =n)}{ \mathbb{P}( | \mathrm{Red}( \mathrm{T})| =r)}.$$
\end{lemma}
\proof Indeed, the event $\mathrm{E}_{m}( \mathrm{Red}( \mathrm{T})) =  \mathfrak{t}$ happens if and only the underlying tree $  \mathrm{Red}( \mathrm{T})$ is obtained by grafting trees on the $ \mathfrak{h}$ vertices remaining to explore $ \mathfrak{t}$. The constraints in the first case is that the total number of vertices be $n$ and in the second case that the total number of red leaves is $r$. \endproof 

Let us denote $f_{h}(n) = \mathbb{P}(\#  \mathrm{T}_{1} + \cdots + \#  \mathrm{T}_{h} =  n)$ and $g_{h}(n) = \mathbb{P}(| \mathrm{Red}( \mathrm{T}_{1})| + \cdots + |\mathrm{Red}( \mathrm{T}_{h})| =  n)$. Since $p$ is critical, aperiodic and has finite variance $\sigma^{2}>0$, we have the classic local-limit estimate 
$$ f_{h}(n) \underset{ \mathrm{Kemperman}}{=} \frac{h}{n} \mathbb{P}(S_{n} =-h) \underset{ \mathrm{local \ CLT}}{=} \frac{h}{n} \left(\frac{1}{ \sigma \sqrt{2\pi n}}  \mathrm{e}^{{- \frac{h^{2}}{2 n \sigma^{2}}}} +  o(1/\sqrt{n})\right)$$ where $S_{n}$ is a $n$-step random walk with increments $(p_{k+1})_{k \geq -1}$ and the $o(1/\sqrt{n})$ is independent of $h$. In particular, $f_{1}(n) \sim \frac{1}{  \sqrt{2\pi \sigma^2}} n^{{-3/2}}$ as $n \to \infty$ which recovers \eqref{eq:polyprob}. We have a similar estimate for $g_{h}(n)$: the case $h=1$ is given by 
$$ g_1(n) = \mathbb{P}_{\pi/2}(|\tau| =n-1)  \underset{ \eqref{eq:asymptgeneral}}{\sim}  \frac{1}{\sqrt{8 \pi}}\, \frac{J_0( c_0/2)}{J_1( c_0/2)}\, n^{-3/2}.$$
Comparing with the formulas for the constants $\sigma$ in Proposition~\ref{prop:offspringp} and $\alpha$ in \eqref{eq:defalpha}, this shows that 
$$ \alpha g_1([\alpha n]) \sim f_1(n)\quad \mbox{ as }n \to \infty.$$
The general case $h\geq 1$ is then a consequence of Gnedenko's local limit theorem \cite[Theorem 4.2.1]{IL71} for variables in the domain of attraction of the spectrally positive $1/2$-stable law.
In particular, for $h = O( \sqrt{n})$ we have 
  \begin{eqnarray} \label{eq:asymptofg} f_h(n) \sim \alpha g_h(m_n) \quad \mbox{ as soon as }  \frac{m_n}{n} \to \alpha.  \end{eqnarray}

We can now prove 
\begin{proposition} \label{prop:Ttildefinal} The convergence of Proposition \ref{prop:stretched} remains valid (with the same scaling constants) under the laws $\mathbb{P}( \cdot \mid |\mathrm{Red}( \mathrm{T})| = [\alpha n])$ instead of $ \mathbb{P}( \cdot \mid \# \mathrm{T} = n)$.
\end{proposition}
\begin{proof} In the previous lemma, we take $ m =[(1- \delta)n]$ for some small but fixed $\delta \in (0,1)$. Using Proposition \ref{prop:MM}, we see that under $ \mathbb{P}( \cdot \mid \# \mathrm{T}=n)$, the typical bicolored trees $ { \tau}$ explored after $m$ steps of the lexicographical visit are such that $ \mathfrak{r} \approx \alpha \cdot (1-\delta) n$ and $ \mathfrak{h} = O( \sqrt{n})$. So if we further take $r = [ \alpha n]$, we deduce thanks to \eqref{eq:asymptofg} that  the Radon-Nikodym derivative of the first $m$ steps of the exploration  of $ \mathrm{Red}( \mathrm{T})$ under $ \mathbb{P}( \cdot \mid | \mathrm{Red}( \mathrm{T})|= [\alpha n])$ with respect to  $ \mathbb{P}( \cdot \mid \#  \mathrm{T}= n)$ goes to $1$ in probability as $n \to \infty$ under the reference law  $ \mathbb{P}( \cdot \mid \#  \mathrm{T}= n)$. As a consequence of Scheff\'e's lemma, we deduce that the random tree $ \mathrm{Red}( \mathrm{T})$ of law  $ \mathbb{P}( \cdot \mid \# \mathrm{T}=n)$ and that of law $ \mathbb{P}( \cdot \mid \# \mathrm{Red}( \mathrm{T})= [\alpha n])$,  together with their labels\footnote{Recall that the labels are obtained from $ \mathrm{Red}( \mathrm{T})$ by applying $ \mathbf{Blob}$, so they can obviously be coupled.}, can be coupled so that they agree in the first $m= [(1-\delta)n]$ steps of the lexicographical exploration of $ \mathrm{T}$ with a probability tending to $1$ as $n \to \infty$. Using \eqref{conv:types} it follows that the binary labeled trees actually coincide for the first  $2 \alpha n (1- 2 \delta)$ steps of lexicographical exploration of $  \mathbf{B}$ as well. In particular, the convergence of Proposition \ref{prop:stretched} occurs at least if we restrict to $0 \leq t \leq 1- 2\delta$. Taking $\delta =1/3$,  by time reversal symmetry, we deduce that the law of the process over the full interval $[0,1]$ is tight (if a sequence of processes is tight over $[0,2/3]$ and over $[1/3,1]$ it is tight over $[0,1]$). This enables us  to send harmlessly $\delta \to 0$ and get our desired convergence. \end{proof}

\subsection{Step 5: Final modifications at the root}
To finally prove Theorem \ref{thm:scaling} we need to deal with the particular offspring distribution at the root:  By Proposition \ref{prop:local1}, conditionally on having $n$ leaves, after removing the origin vertex and its black neighbors in  $ \mathrm{Red}^{\redl}( \mathfrak{T})$, we have a single giant component containing all but $O_{\mathbb{P}}(1)$ of the leaves. Conditionally on the size $n'= n-O_{\mathbb{P}}(1)$ of this giant component, it has the law $ \mathrm{Red}( \mathrm{T})$ under $ \mathbb{P}( \cdot \mid | \mathrm{Red}( \mathrm{T})| = n')$ and so our previous result applies.
This settles Theorem \ref{thm:scaling} with the scaling constants related to those of Propostion~\ref{prop:stretched} via 
\begin{align}\label{eq:constantrelation}
  c_1 = \frac{\kappa\,c_3}{\sqrt{\alpha}}, \qquad c_2 = \frac{c_4}{\alpha^{1/4}}.
\end{align}
These will be explicitly computed in the next subsection.

\subsection{Step 6: Computing the scaling constants}\label{sec:scalingcst}

To determine the scaling constants we examine more closely the labels along the spine in the random infinite labeled tree  with law $\mathbb{P}_0^\infty$ described in Theorem~\ref{thm:localleaf}:

 \begin{proposition}  \label{prop:scalingspine} Let $ (  \mathcal{L}_{k} : k \geq 0)$ be the $\ell$-labeling along the unique spine of a random tree of law $ \mathbb{P}_{0}^{\infty}$. Then almost surely $( \mathcal{L})$ is recurrent and in particular   \begin{eqnarray} \liminf_{k \to \infty}  \mathcal{L}_k = - \infty \quad \mbox{ and } \quad \limsup_{k \to \infty}  \mathcal{L}_k = + \infty, \label{eq:lowerspine}  \end{eqnarray} and  
 $$\mathrm{Var}(\mathcal{L}_k)  \sim \lambda^2\, k \quad \mbox{ as }k \to \infty,\text{ where }\quad \lambda^2 = \frac{8 \pi^2 J_1(c_0)}{3 c_0 |\cos c_0|} =  7.67124\ldots.$$
Furthermore the number $ \mathcal{C}_k$ of cut-edges encountered along the first $k$ edges of the spine satisfies 
$$  \frac{\mathcal{C}_k}{k} \xrightarrow[k\to\infty]{\mathbb{P}_0^\infty} \frac{J_0(\tfrac{c_0}{2})^2}{|\cos c_0|} = 0.60596\ldots.$$
 \end{proposition}

\begin{proof} Recalling the proof of Lemma \ref{lem:roughexpo}, the $\ell$-label $\mathcal{L}_k$ of the $k$th edge of the spine is given by 
\begin{align*}
  \mathcal{L}_k &= \sum_{i=1}^k 2 \log \frac{\sin \alpha_i}{\sin (\alpha_i+\beta_i)}.
\end{align*}
This is an additive functional for the underlying Markov chain identified in Proposition \ref{prop:spine}. In particular it is easy to check that $\log \frac{\sin \alpha}{\sin (\alpha+\beta)}$ is integrable and of mean $0$ under the stationary distribution $\nu$. It follows by ergodicity that $( \mathcal{L})$ is recurrent see e.g. \cite{kesten1975sums}, and so \eqref{eq:lowerspine} is granted. To compute the asymptotic variance of this process, we shall use the angle-process of Proposition \ref{prop:angleprocess} which is again stationary under $ \tilde{\nu}$ and ergodic. Indeed, if $t_{k}$ is the $k$th jump time of $\theta_{\cdot}$ then 
\begin{align*}
  \int_0^{\infty} (\mathbf{1}_{t<t_k}\cot \theta_t - \mathbf{1}_{t<\pi/2} \cot t)\rmd t &= \log \sin \theta_{t_1-} + \int_{t_1}^{t_k} \cot \theta_t \rmd r \\
  &= \log \sin(\alpha_1+\beta_1) + \sum_{i=1}^{k-1} \log\frac{\sin (\alpha_{i+1}+\beta_{i+1})}{\sin \alpha_i}.
\end{align*}
Therefore,
\begin{align}
  \mathcal{L}_k &= 2 \log \sin \alpha_k - 2 \int_0^{\infty} (\mathbf{1}_{t<t_k}\cot \theta_t - \mathbf{1}_{t < \pi/2}\cot t)\rmd t.\label{eq:labelprocessfromtheta}
\end{align}
To estimate the variance of $\mathcal{L}_k$ as $k\to\infty$, let us consider the truncation $f_\varepsilon(\theta) = \cot \theta\, \mathbf{1}_{\varepsilon < \theta < \pi - \varepsilon}$ and note that it is bounded and has zero mean for the stationary distribution $\tilde{\nu}(\theta) \rmd \theta$.
We claim that 
\begin{align*}
  \phi_\varepsilon(\theta) := -\int_0^\theta \rmd\alpha F_\infty(\theta-\alpha) \frac{F_\infty(\alpha)}{F_\infty(\theta)} f_\varepsilon(\alpha)
\end{align*}
solves the corresponding Poisson equation
\begin{align}\label{eq:poissonsolution}
  (L \phi_\varepsilon)(\theta) = -f_\varepsilon(\theta),
\end{align}
where $L$ is the generator of the process $(\theta_t)_{t\geq 0}$ from Proposition~\ref{prop:angleprocess}.
To see this, we note using \eqref{eq:Finftyprime} that
\begin{align*}
  F_\infty(\theta) (L \phi)(\theta) = \frac{\rmd}{\rmd\theta}(F_\infty(\theta)\phi(\theta)) + \int_0^\theta \rmd \alpha\,2 F(\theta-\alpha)F_\infty(\alpha)\phi(\alpha). 
\end{align*}
Setting $\phi = \phi_\varepsilon$ we compute
\begin{align*}
  \frac{\rmd}{\rmd\theta}(F_\infty(\theta)\phi_\varepsilon(\theta)) &= - F_\infty(0) f_\varepsilon(\theta) - \int_0^\theta \rmd\alpha F_\infty'(\theta-\alpha) F_\infty(\alpha) f_\varepsilon(\alpha) \\
  &\underset{\eqref{eq:Finftyprime}}{=} -f_\varepsilon(\theta) + 2 \int_0^\theta \rmd\alpha \int_0^{\theta-\alpha}\rmd\beta F(\beta)F_\infty(\theta-\alpha-\beta) F_\infty(\alpha) f_\varepsilon(\alpha)\\ 
  &= -f_\varepsilon(\theta) - \int_0^\theta \rmd\alpha 2 F(\theta-\alpha) F_\infty(\alpha)\phi_\varepsilon(\alpha), 
\end{align*}
verifying \eqref{eq:poissonsolution}.
By \cite[Chapter~2]{Komorowski12}, the additive functional $\int_0^T f_\varepsilon(\theta_t)\rmd t$ has finite asymptotic variance $\tilde{\rho}_\varepsilon^2$, meaning
\begin{align*}
  \frac{1}{T} \operatorname{Var}\left(\int_0^T f_\varepsilon(\theta_t)\rmd t\right) \xrightarrow[T\to\infty]{} \tilde{\rho}_\varepsilon^2.
\end{align*}
Moreover, it is expressible in terms of $f_\varepsilon$ and $\phi_\varepsilon$ as
\begin{align*}
  &\tilde{\rho}_\varepsilon^2 = 2 \int_0^\pi \rmd \theta \phi_\varepsilon(\theta) f_\varepsilon(\theta)\tilde{\nu}(\theta) \rmd \theta \\
  &= -\frac{2 c_0}{\pi \sin c_0} \int_0^\pi \rmd\theta \int_0^\theta \rmd \alpha f_\varepsilon(\alpha)f_\varepsilon(\theta) F_\infty(\theta-\alpha)F_\infty(\alpha) F_\infty(\pi-\theta) \\
  &= \frac{4 c_0}{\pi \sin c_0}\int_0^{\pi/2} \rmd\theta \int_0^{\theta} \rmd \alpha f_\varepsilon(\alpha)f_\varepsilon(\theta) F_\infty(\alpha) (F_\infty(\theta)F_\infty(\pi-\theta-\alpha)-F_\infty(\pi-\theta)F_\infty(\theta-\alpha)),
\end{align*}
where in the last equality we made repeated use of the antisymmetry $f_\varepsilon(\pi-\theta) = - f_\varepsilon(\theta)$.
The latter is seen to be integrable even for $\varepsilon = 0$, so by dominated convergence $\lim_{\varepsilon \to 0} \tilde{\rho}_{\varepsilon}^2 = \tilde{\rho}_0^2$ is finite.
By Lemma~\ref{lem:besselcotintegral} below,
\begin{align*}
  \tilde{\rho}_0^2 = \frac{4 c_0}{\pi \sin c_0}\frac{\pi^2}{6 c_0} J_1(c_0) = \frac{2\pi}{3\sin c_0}J_1(c_0).
\end{align*}
On the other hand, the average time between jumps approaches the inverse expected jump rate
  \begin{align*}
    \frac{t_k}{k} \xrightarrow[k\to\infty]{\mathbb{P}_0^\infty-a.s.}  \left(  \int_0^\pi \rmd\theta \tilde{\nu}(\theta) \int_0^\theta \rmd\alpha 2 F(\theta-\alpha) \frac{F_\infty(\alpha)}{F_\infty(\theta)}\right)^{-1}
  \end{align*}
which is exactly computed using \eqref{eq:Finftyprime}
  \begin{align*}
    \int_0^\pi \rmd\theta \tilde{\nu}(\theta) \int_0^\theta \rmd\alpha 2 F(\theta-\alpha) \frac{F_\infty(\alpha)}{F_\infty(\theta)} &= - \int_0^\pi \rmd\theta \tilde{\nu}(\theta) \frac{F'_\infty(\theta)}{F_\infty(\theta)} \\
    &= - \frac{c_0 \cos c_0}{\pi \sin c_0} \int_0^\pi \rmd\theta \nu(\theta) = \frac{c_0}{\pi |\tan c_0|}.
  \end{align*}
Combining with the relation \eqref{eq:labelprocessfromtheta} and the law of large numbers for the jump times $t_k$ from Proposition~\ref{prop:angleprocess}, we conclude for the variance of $\mathcal{L}_k$ that
\begin{align*}
  \lim_{k\to \infty}\frac{1}{k}\operatorname{Var}\mathcal{L}_k = 4 \tilde{\rho}_0^2 \lim_{k\to\infty }\frac{t_k}{k} = \frac{4 \pi}{c_0} |\tan c_0| \tilde{\rho}_0^2 = \frac{8 \pi^2 J_1(c_0)}{3 c_0 |\cos c_0|} = \lambda^2.
\end{align*}
One can also establish a law of large number for the number of cut-edges along the spine. We first compute the expected density of  cut-edges for the $\theta$-process: From Figure~\ref{fig:splitting}, if $\theta \in (0,\pi/2)$, the first edge on the spine is a cut-edge in the event $\{\alpha + \beta > \pi/2\}$, which happens with probability
\begin{align*}
  \mathbb{P}_\theta^\infty( \alpha + \beta > \pi/2) = \frac{1}{F_\infty(\theta)} \int\int_0^\pi \mathbf{1}_{\tfrac{\pi}{2} \leq \alpha + \beta < \pi} 2F(\alpha)F_{\infty}(\beta) = \frac{F_\infty(\pi/2)}{F_\infty(\theta)}.
\end{align*}
By Proposition~\ref{prop:spine}, as $k\to\infty$ the probability that the $k$th edge is a cut-edge therefore approaches
\begin{align*}
  \int_0^{\pi/2}\rmd\theta\, \nu(\theta) \frac{F_\infty(\pi/2)}{F_\infty(\theta)} = \frac{F_\infty(\pi/2)}{\cos c_0}\int_0^{\pi/2} \rmd\theta F_\infty'(\pi-\theta) = \frac{-F_\infty(\pi/2)^2}{\cos c_0} = \frac{J_0(\tfrac{c_0}{2})^2}{|\cos c_0|}.
\end{align*}
By ergodicity of $(\alpha_k)_{k\geq 0}$, the claimed convergence in probability of $\mathcal{C}_k/k$ holds. \end{proof}

\begin{lemma}\label{lem:besselcotintegral}
  We have the integral identity
  \begin{align*}
    \int_0^{\pi/2} \rmd\theta \int_0^{\theta} \rmd \alpha \cot\theta \cot\alpha  F_\infty(\alpha) (F_\infty(\theta)F_\infty(\pi-\theta-\alpha)-F_\infty(\pi-\theta)F_\infty(\theta-\alpha)) = \frac{\pi^2}{6 c_0} J_1(c_0).
  \end{align*}
\end{lemma}
\begin{proof}See MathOverlow question \url{https://mathoverflow.net/questions/393947/}. 
\end{proof}

With the help of these computations, we can finally identify all scaling constants.

\begin{corollary}
  The scaling constants appearing in Theorem~\ref{thm:scaling}, Proposition~\ref{prop:MM} and Propostion~\ref{prop:stretched} are
  \begin{align*}
    &c_1 = \frac{2 |\cos c_0|}{J_1(c_0)}, \qquad\qquad c_2 = \frac{4\pi}{\sqrt{3 c_0}}\\
    &c_3 = \sqrt{\frac{2J_0(\tfrac{c_0}{2})^3}{J_1(c_0)J_1(\tfrac{c_0}{2})}}, \qquad c_4 = \frac{4\pi}{\sqrt{3 c_0}} \left(\frac{J_1(c_0)}{2 J_0(\tfrac{c_0}{2})J_1(\tfrac{c_0}{2})}\right)^{1/4}.
  \end{align*}
\end{corollary}

\begin{proof}
  The offspring variance $\sigma^2$ of the critical $p$-Galton-Watson tree $\mathrm{T}$ was computed in Proposition~\ref{prop:offspringp} and the invariance principle in Proposition~\ref{prop:MM} gives 
  \begin{align*}
    c_3 = \frac{2}{\sigma} = \sqrt{\frac{2J_0(\tfrac{c_0}{2})^3}{J_1(c_0)J_1(\tfrac{c_0}{2})}}.
  \end{align*}
  Moreover, by Proposition~\ref{prop:MM} the label normalized by $\sqrt{h}$ of a vertex at large height $h$ should approach a normal distribution of variance $c_4^2/c_3$.
  The latter must match the variance $\rho^2$ of the label increment along the spine of the tree $p$-Galton--Watson tree conditioned to survive $ \mathrm{T}_{\infty}$ labeled using the red and blob constructions and projecting back to $ \mathrm{T}_{\infty}$  as in \eqref{eq:projectlabel}, so that $c_4 = \rho \sqrt{c_3}$. To compute the later, we notice that with the notation of the previous proof we must have  $\operatorname{Var}(\mathcal{L}_k) \sim \rho^2 \mathcal{C}_k$ as $k\to\infty$. 
We conclude that \begin{align*}
  \rho^2 = \frac{8\pi^2}{3 c_0} \frac{J_1(c_0)}{J_0(c_0/2)^2}.
\end{align*}
  Similarly, the stretching factor $\kappa$ for the height of the blob tree in Proposition~\ref{prop:stretched} should balance the fraction $\mathcal{C}_k/k$ of cut-edges encountered along the spine, so that Proposition~\ref{prop:scalingspine} implies 
  \begin{align*}
    \kappa = \frac{|\cos c_0|}{J_0(\tfrac{c_0}{2})^2}.
  \end{align*} 
  The values for $c_1$ and $c_2$ then follow from \eqref{eq:constantrelation}.
\end{proof}

\part{Convergences of WP punctured spheres}

In this final part we use the limit theorems on the labeled trees encoding the random WP surfaces to establish their scaling and local limits announced in the introduction. To prove them, we develop methods to control the geometry of hyperbolic surfaces built from labeled trees, inspired by similar results in the theory of random planar maps.

\section{Benjamini--Schramm convergence}
We show in this section how to deduce the local convergence   of $ \mathcal{S}_n$ (Theorem \ref{thm:BS})  from the local convergence of the random trees $ \bT_n$ established in Corollary \ref{cor:aldous}. The idea of the proof is similar to the proof of the local convergence of uniform quadrangulations towards the UIPQ using the pointed Schaeffer mapping in \cite[Section 2]{CMMinfini}. We first need to extend the construction of random hyperbolic surfaces to infinite one-ended binary labeled trees. 

\subsection{Cactus bound and construction of infinite punctured surfaces}

 \label{sec:local}

Recall from Section \ref{sec:WP} the construction of hyperbolic surfaces $ \in \mathcal{M}_{0,n+1}$ with two distinguished punctures (the root and the origin punctures) from a  binary tree $\tau$ with $n$ leaves and an allowed angle assignment. More precisely, from the angle assignment on $\tau$, we produce using Proposition \ref{prop:hypeucl} the lambda-lengths $(\lambda_{e} : e \in \mathrm{Edges}(\tau))$ up to multiplicative scaling or equivalently the $\ell$-labeling up to an additive constant. Then for each $u \in \tau$, we consider a decorated ideal triangle $ \mathsf{t}_{u}$ whose lambda lengths are prescribed by the unique edge or the three edges incident to $u$, and then glue those triangles along the combinatorics of $\tau$ as in Figure \ref{fig:triangles-type}. The multiplicative scaling is fixed to ensure that the horocycle length around the origin puncture $p_{\mathrm{origin}}$  is equal to $ \epsilon_{n+1}=1$ and we distinguish the root puncture $p_{\mathrm{root}}$ corresponding to the root leaf of $\tau$ (with horocycle length $\epsilon_{1}=0$). We denote the resulting surface 
$$ \mathrm{Glue}(\btau) \in \mathcal{M}_{0,n+1},$$ where the information about the origin and root puncture is implicit in the notation\footnote{{Here, we slightly abuse notation: strictly speaking, a surface $X \in \mathcal{M}_{0,n+1}$ should be equipped with labeled punctures. However, the choice of origin and root puncture breaks all symmetries, so we may safely omit these labels for both the surface and the associated tree.}}. Recall that each edge $e \in \tau$ corresponds to a geodesic $\gamma_{e}$ linking $p_{\mathrm{origin}}$ to itself in $\mathrm{Glue}(\btau)$ and that the (signed) length of this loop with respect to the horocycle $ \epsilon_{n+1}=1$ of the origin puncture is precisely $\ell(e)$. We shall denote by $x_{e} \in \mathrm{Glue}(\btau)$ the furthest point of $\gamma_{e}$ from $p_{ \mathrm{origin}}$ and in particular $x_{e_{\varnothing}}$ is the point corresponding to the root edge of $\tau$.  We extend this labeling to all points $x \in \mathrm{Glue}(\btau)$ by putting 
$$ h(x) = \mathrm{d_{hyp}}(x, p_{\mathrm{origin}}) -  \mathrm{d_{hyp}}(x_{e_{\varnothing}}, p_{\mathrm{origin}}),$$  where the distances above have to be interpreted as the distance to the horocycle of length $1$ around $p_{\mathrm{origin}}$. This function, called below the (signed) distance to $p_{\mathrm{origin}}$, is unbounded from above and below and is pinned down to $0$ at $ x_{e_{\varnothing}}$. More generally we have $h(x_{e})= \ell(e)/2$ for each $e \in \mathrm{Edges}(\tau)$. We can state our first control on the distances in $\mathrm{Glue}(\btau)$ given the labeled tree $\btau$, which is similar in spirit to \cite{CLGMcactus}:

\begin{lemma}[Hyperbolic cactus bound]  \label{lem:cactus}For any  $u,v \in \tau$ and any edge $e$ separating $u$ from $v$ in $\tau$, for any $x \in  \mathsf{t}_{u} \subset  \mathrm{Glue}( \btau)$ and $y \in  \mathsf{t}_{v} \subset  \mathrm{Glue}( \btau)$ we have 
$$ \mathrm{d}_{\mathrm{hyp}}^{ \mathrm{Glue}(  \boldsymbol{\tau})}(x,y) \geq \big(h(x)- \ell(e)/2\big)_+ + \big(h(y) - \ell(e)/2\big)_+.$$	
\end{lemma}
 \begin{proof}
 The proof is based on a topological argument: by our assumption on $e$, the geodesic loop $\gamma_e \subset  \mathrm{Glue}( \btau)$ separates the surface in two connected components, one containing $x$ and the other containing $y$, see Figure \ref{fig:cactus}.
 
 \begin{figure}[!h]
  \begin{center}
  \includegraphics[width=13cm]{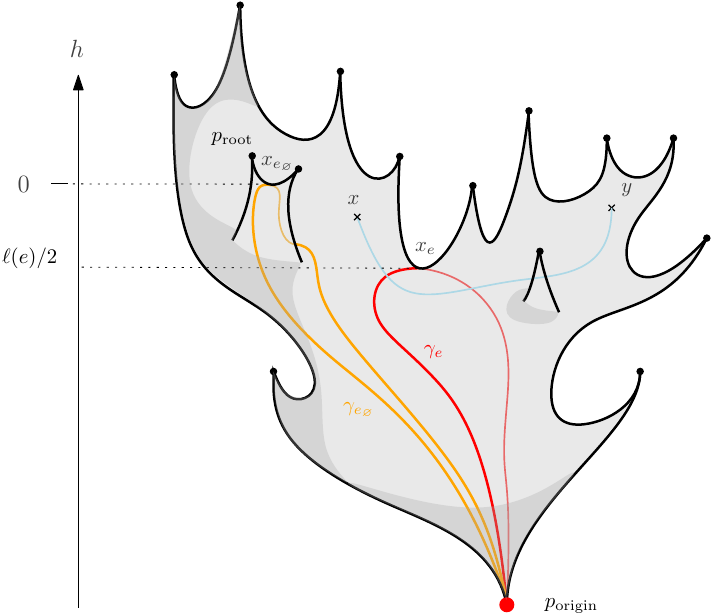}
  \caption{Illustration of the cactus bound: the geodesic going from $x$ to $y$ must cross $\gamma_{e}$ and so must come ``close'' to $ p_{ \mathrm{origin}}$ forcing the distance between $x$ and $y$ to be larger than the $h$-variation. \label{fig:cactus}}
  \end{center}
  \end{figure} By Jordan's theorem, any continuous path going from $x$ to $y$ on the surface must cross $\gamma_e$ and in particular reaches a point whose $h$-value is below $\ell(e)/2$. Since $h$ is obviously $1$-Lipschitz along the geodesic from $x$ to $y$ we deduce that the $h$-variation is bounded above by the hyperbolic distance between $x$ and $y$, in particular 
 $$ \mathrm{d_{hyp}}(x,y) \geq (h(x)- \ell(e)/2)_+ + (h(y) - \ell(e)/2)_+.$$
 \end{proof}

If $\btau_{\infty}$ is an infinite tree with an allowed angle configuration, using the same gluing procedure, one can still define $\mathrm{Glue}(\btau_{\infty})$ but the resulting hyperbolic surface may not be complete due to local accumulation of gluings (the gluings of course accumulate at the origin puncture). We shall see that this is not the case if 
\begin{itemize}
\item $\tau_{\infty}$ has a unique spine
\item the $\ell$-labels are unbounded from below along the spine.
\end{itemize}	
\begin{lemma}[Continuity of the gluing]  \label{lem:continuity} Let $\btau_n$ be a sequence of plane binary labeled trees with allowed angle assignments and $n$ leaves converging locally towards an infinite labeled tree $\btau_\infty$ having a unique spine and such that the labels along the spine are unbounded from below. Then we have 
$$ ( \mathrm{Glue}(\btau_n),x_{e_\varnothing}) \xrightarrow[n\to\infty]{(loc)} ( \mathrm{Glue}(\btau_\infty), x_{e_\varnothing}),$$  
in the sense that for every $r>0$ the ball $ B_{r}( \mathrm{Glue}(\btau_n),x_{e_\varnothing})$ is eventually constant and equal to $B_{r}( \mathrm{Glue}(\btau_\infty),x_{e_\varnothing})$.
\end{lemma}
\begin{proof} Fix $r >0$. By hypothesis on $\btau_\infty$, we can find an edge $e$ on its spine such that $\ell(e) < -2r$. Since $\tau_{\infty}$ has one spine, the tree $\tau_{\infty}$ pruned at $e$ has finite height $k$ and $[\btau_n]_{k+1}= [\btau_\infty]_{k+1}$ for all $n$ large enough (in particular $e$ is meaningfully defined in $\tau_{n}$ as well). Recalling that $\ell_{e_{\varnothing}} =0$, by the hyperbolic cactus inequality (Lemma \ref{lem:cactus}), the distance between any point $y \in \mathsf{t}_v$ where $v$ is separated from $\varnothing$ by $e$ must satisfy $ \mathrm{d_{hyp}}( x_{e_\varnothing},y) > r$. Hence the ball of radius $r$ around $x_{e_\varnothing}$ in $ \mathrm{Glue}( \btau_{n})$ only depends on $\btau_{n}$ pruned at $e$ and is eventually constant. The rooted surface $(\mathrm{Glue}(\btau_\infty),x_{e_\varnothing})$ is then defined by its sequence of coherent balls around $x_{e_\varnothing}$. It has a countable number of cusps (which do not accumulate), no boundary and has the topology of $ \mathbb{R}^2\backslash \mathbb{Z}^2$.\end{proof}

\subsection{Infinite random WP surfaces}
Let $ \mathcal{S}_n$ be a random WP sphere with $n+1$ punctures. By Theorem \ref{thm:treeencoding} one can suppose that $ \mathcal{S}_n = \mathrm{Glue}( \bT_n)$ where $ \bT_n$ is a uniform binary tree with allowed angle assignment and $n$ leaves (and rooted on a leaf). This construction enables us to root the surface  $ \mathcal{S}_n$ at $x_{e_\varnothing}$. However, this rooting might seem ad-hoc and we shall also  sample another distinguished point $ \circ_n$, which conditionally on $ \mathcal{S}_n$, is picked according to the normalized hyperbolic measure $\mu_{ \mathrm{hyp}}$. Recall the definitions of the local limits $ \bT_\infty$ and $ \bT^*_{\infty}$ from Theorem \ref{thm:localleaf} and Corollary \ref{cor:aldous}:

\begin{theorem} \label{thm:localprecise} We have the following local convergence of random pointed hyperbolic surfaces 
\begin{align*}
( \mathcal{S}_{n} , x_{e_{\varnothing}}) \xrightarrow[n\to\infty]{(d)} &( \mathcal{S}_{\infty},  {x}) \overset{(d)}{=} (\mathrm{Glue}(  \bT_\infty),x_{e_{\varnothing}}),\\
( \mathcal{S}_{n} , \circ_n) \xrightarrow[n\to\infty]{(d)} &( \mathcal{S}^{*}_{\infty},  \circ_\infty) \overset{(d)}{=}  (\mathrm{Glue}(  \bT^*_\infty), \circ_\infty),
\end{align*}
in the sense that for every $r>0$ we have $ \mathrm{d_{TV}}( B_{r}( \mathcal{S}_{n}, \circ_n); B_{r}( \mathcal{S}^*_{\infty}, \circ_\infty)) \to 0$ as $n \to \infty$, where  $ \mathrm{d_{TV}}$ is the total variation distance and $B_{r}(X,\rho)$ is the hyperbolic metric ball of radius $r$ around the point $\rho \in X$ (and similarly for the first convergence). The choice of the point $\circ_\infty$ is, conditionally on $\mathrm{Glue}(  \bT^*_\infty)$, uniform in the triangle $ \mathsf{t}_\varnothing$.
\end{theorem}

\begin{remark} In Theorem \ref{thm:BS}, the limit of $( \mathcal{S}_{n}, \circ_{n})$ was denoted $ ( \mathcal{S}_{\infty}, \circ_{\infty})$ without the star, but it is indeed the above $( \mathcal{S}_{\infty}^{*}, \circ_{\infty})$.
\end{remark}

\begin{proof} Recalling Theorem \ref{thm:localleaf}, we may assume by Skorokhod representation theorem that $ \bT_n  \xrightarrow[]{(loc)} \bT_\infty$ almost surely as $n \to \infty$ where $\bT_\infty$ of law $ \mathbb{P}_{0}^{\infty}$ has almost surely one-spine along which labels are unbounded from below by \eqref{eq:lowerspine}. The first convergence is then a consequence of Lemma \ref{lem:continuity}.\\
Let us move on to the second convergence seen from a uniform point.  Since all ideal triangles have the same hyperbolic area $\pi$, picking $ \circ_{n} \in \mathcal{S}_{n}$ according to the hyperbolic measure  amounts to first picking a vertex $\reddot$ of $ \bT_n$ uniformly at random and then picking a point in the associated hyperbolic triangle $ \mathsf{t}_{\reddot} \subset \mathrm{Glue}( \bT_n)$ uniformly according to the hyperbolic measure. If we further distinguish an oriented edge $e^*_n$ emanating from $\reddot$ and reroot there, we proved in Corollary~\ref{cor:aldous} that the plane tree obtained $ \bT^*_n$ converges locally towards the infinite binary labeled tree $\bT^*_\infty$. By Corollary~\ref{cor:aldous} and \eqref{eq:lowerspine} the tree $\bT^*_\infty$ again has labels unbounded from below along the spine. By the same argument as above, and applying Lemma~\ref{lem:continuity} we  have the convergence $$( \mathcal{S}_{n} , x_{e_n^*} ) \xrightarrow[n\to\infty]{(d)} (  \mathrm{Glue}( \bT_\infty^*),  x_{e^*_\infty})$$ where $e^*_\infty$ is the root edge of $\bT_\infty^*$.  Now if $\circ_n$ (resp. $\circ_\infty$) is, conditionally on $\bT_n^*$ (resp. $\bT_\infty^*$) a uniform point in the hyperbolic triangle $ \mathsf{t}_\varnothing$, since $\mathrm{d_{hyp}}(\circ_n,x_{e^*_n})$ is almost surely finite, we deduce the desired convergence of the surfaces pointed at $\circ_n$ and $\circ_\infty$ instead.    \end{proof}

 \subsection{A few questions}
 
 The random infinite surface \((\mathcal{S}^{*}_{\infty}, \circ_\infty)\) is the hyperbolic analog of the Uniform Infinite Planar Triangulation (UIPT) of Angel \& Schramm~\cite{AS03}. It is a unimodular random hyperbolic surface~\cite{abert2016unimodular} that describes the local scenery around a typical point in a random planar Weil--Petersson surface with many cusps. A natural and important goal is to compute exact distributional quantities for this object. Inspired by the extensive body of work on the UIPT, several interesting questions arise regarding \(\mathcal{S}^{*}_{\infty}\):

\begin{itemize}
    \item Can the reverse Penner construction be extended to \(\mathcal{S}_{\infty}^{*}\), as studied in~\cite{CMMinfini,CMKrikun}?
    \item Is the Brownian plane~\cite{CLGplane} the scaling limit of \(\mathcal{S}_{\infty}^{*}\)?
    \item Can we define a ``hyperbolic'' (sic!) version of \(\mathcal{S}_{\infty}^{*}\) as in \cite{CurPSHIT} and study its scaling limit as in \cite{budzinski2018hyperbolic}?
    \item What is the conformal structure of \(\mathcal{S}_{\infty}^{*}\), as explored in~\cite{CurKPZ} for the UIPT?
    
\end{itemize}

A more modest but foundational question is to determine the conformal type of \(\mathcal{S}_{\infty}^{*}\), which is the continuous analog of the recurrence/transience dichotomy in the discrete setting. Indeed, \(\mathcal{S}_{\infty}^{*}\) is naturally equipped with the structure of a simply connected infinite Riemann surface (by completing the punctures), and thus must be conformally equivalent to either \(\mathbb{C}\) (parabolic case) or \(\mathbb{D}\) (hyperbolic case). This classification is known to be equivalent to the recurrence/transience of Brownian motion on \(\mathcal{S}_{\infty}\) and to the recurrence (and ergodicity) of the uniform geodesic flow, see the references in \cite{pandazis2024parabolic}. It is natural to conjecture here that
the random infinite hyperbolic surface \(\mathcal{S}_{\infty}^{*}\) is parabolic (conformally equivalent to \(\mathbb{C}\)), and thus Brownian motion on \(\mathcal{S}_{\infty}^{*}\) is recurrent, and the uniform geodesic flow is ergodic. This conjecture is motivated by the analogy with the UIPT and the general principle that ``large'' finite  surfaces seen from a typical point tend to be parabolic, see the pioneer work of Benjamini--Schramm \cite{BS01} later adapted in different contexts \cite{GGN12,namazi2014distributional,GR10}.

\section{Convergence to the Brownian sphere}
We finally deduce the scaling limits    of $ \mathcal{S}_n$ (Theorem \ref{thm:GH}) from the scaling limits of $ \bT_n$ established in Theorem \ref{thm:scaling}. The proof borrows  from Le Gall's approach \cite{LG11}.
However, some non-trivial adaptations to the hyperbolic setting are necessary, most notably finding the analog of ``Schaeffer's'' upper bound  (see Proposition \ref{prop:Doupper}). We start by recalling the basic definition of Gromov--Hausdorff and Gromov--Hausdorff--Prokhorov metrics as well as the  definition of the Brownian sphere  \cite{LG11,Mie13} for the reader's convenience.

\subsection{Gromov-type topologies} \label{sec:gromov}

Consider the set $\mathbb{K}$ of all isometry classes of compact metric spaces $(K,d_K,\mu)$ equipped with a finite Borel measure $\mu$, meaning that $(K,d_K,\mu)$ is equivalent to $(K^{\prime},d_{K^{\prime}},\mu^{\prime})$ if there is a bijective isometry between $K$ and $K^\prime$ that takes $\mu$ to $\mu^\prime$. To lighten notation, we shall often identify a compact measured metric space with its equivalence class, and  leave to the reader to check that the definitions and results provided in these pages actually do not depend on the chosen representative. We equip $\mathbb{K}$ with the Gromov--Hausdorff--Prokhorov (GHP) metric, namely for every $\textbf{K}:=(K,d_K,\mu)$ and  $\textbf{K}^{\prime}:=(K^{\prime},d_{K^{\prime}},\mu^{\prime})$ in $\mathbb{K}$:
$$d_{\mathrm{GHP}}\big(\textbf{K},\textbf{K}^{\prime}\big):=\inf\limits_{\phi,\phi^{\prime}}\Big( \delta_{\text{H}}\big(\phi(K),\phi^{\prime}(K^{\prime})\big)\vee \delta_{\text{P}}\big(\phi_* \mu,\phi^{\prime}_* \mu^{\prime}\big)\Big)~,$$
where the infimum is taken over all isometries $\phi$, $\phi^{\prime}$ from $K$, $K^{\prime}$ into a third metric space and $\delta_{\text{H}}$ (resp.~$\delta_{\text{P}}$) stands for the classical Hausdorff distance (resp.~the Prokhorov distance). The space $(\mathbb{K},d_{\mathrm{GHP}})$ is a Polish space, see \cite{Mie09,abraham2013note} for more details. If one equips the compact metric spaces with the trivial $0$ measure, we get the Gromov--Hausdorff distance.  \medskip 

The GH/GHP distances are usually estimated using correspondences. A \textbf{correspondence} $ \mathcal{C}$ between two compact metric spaces $K$ and $K^{\prime}$ is a compact subset of $ K \times K^{\prime}$ such that for any $x_{1} \in M$ there exists $y_{1} \in K^{\prime}$ such that $(x_{1},y_{1}) \in \mathcal{C}$ and reciprocally for any $y_{2} \in K^{\prime}$ there exists $x_{2} \in K$ with $(x_{2},y_{2}) \in \mathcal{C}$. The distortion of the correspondence is 
  \begin{eqnarray} \label{def:distortion}  \mathrm{dis}( \mathcal{C}) = \sup \left\{ |d_{K}(x_{1},x_{2}) - d_{K^{\prime}}(y_{1},y_{2})| : (x_{1},y_{1}), (x_{2}, y_{2}) \in \mathcal{C} \right\}.  \end{eqnarray}
The Gromov--Hausdorff distance between $(K,d_K)$ and $(K^{\prime},d_{K^{\prime}})$ is then given by half of the infimum of the distortion of correspondences between $K$ and $K^{\prime}$, see \cite[Lemma 2.3]{burago2001course}. For the Gromov--Hausdorff--Prokhorov distance, we further require the existence of a coupling of $\mu, \mu^{\prime}$, that is a probability measure $\nu$ on $K \times K^{\prime}$ which projects onto $\mu$ and $\mu^{\prime}$ respectively. Then by \cite[Proposition 6]{Mie09}
  \begin{eqnarray} \label{def:GHPcor}  d_{\mathrm{GHP}}\big(\textbf{K},\textbf{K}^{\prime}\big) =  \inf \left\{ \varepsilon>0 :  \exists (  \mathcal{C}, \nu), \mbox{ with } \mathrm{dis}( \mathcal{C}) \leq 2 \varepsilon \mbox{ and } \nu( \mathcal{C}) \geq 1 - \varepsilon \right\},  \end{eqnarray} where $ \mathcal{C}$ is a correspondence and $\nu$ a coupling  of $\mu, \mu^{\prime}$. 

The Gromov--Hausdorff(--Prokhorov) topology requires one to deal with \textit{compact} spaces, while our hyperbolic surfaces with cusps are not compact. The Gromov--Prokhorov topology  elegantly deals with this type of objects:  If $(P,d_P,\mu)$ and $(P^{\prime},d_{P^{\prime}},\mu^{\prime})$ are two Polish spaces (complete separable, but non necessarily compact)  equipped with a probability measure of full support, the Gromov--Prokhorov distance is given by relaxing \eqref{def:GHPcor} 
\begin{eqnarray} \label{def:GPcor}  d_{\mathrm{GP}}\big(\textbf{P},\textbf{P}^{\prime}\big) =  \inf \left\{ \varepsilon>0 :  \exists (  \mathcal{R}, \nu), \mbox{ with } \mathrm{dis}( \mathcal{R}) \leq 2 \varepsilon \mbox{ and } \nu( \mathcal{R}) \geq 1 - \varepsilon \right\},  \end{eqnarray}
where the infimum is taken over all couplings $ \nu$ and all \textbf{relations} $ \mathcal{R}$ which is simply a Borel subset of $  P \times P^{\prime}$ (as opposed to a correspondence which must project onto $P$ and $P^{\prime}$) and where the distortion of $ \mathcal{R}$ is computed as in \eqref{def:distortion}. Then $ d_{ \mathrm{GP}}$ is a metric on the set $ \mathbb{P}$ of equivalence classes of Polish spaces equipped with a probability measure of full support and $ ( \mathbb{P}, \mathrm{d}_{ \mathrm{GP}})$ is itself Polish. See \cite{janson2020gromov} for details, equivalent definitions and references.

\subsection{Construction of the metric of the Brownian sphere}

 \label{sec:Bmap}
We recall now the construction of the Brownian sphere from \cite{LG11}. The Brownian sphere is built from a random pair  $( \mathbf{e},Z)$ made of a normalized Brownian excursion $\mathbf{e}$ and the head of the Brownian snake $Z$ driven by $\mathbf{e}$. Let us recall the basic definitions, thinking first of $(\mathbf{e},Z)$ as an abstract deterministic pair of functions.

Given a \textit{continuous} excursion $ \mathbf{e} : [0,1] \to \mathbb{R}_{+}$ satisfying $ \mathbf{e}(0)= \mathbf{e}(1)=0$, we say that two times $s,t \in [0,1]$ are identified by $ \mathbf{e}$ and write $ s \sim t$ if $  \mathbf{e}(s)= \mathbf{e}(t)= \min_{[s,t]} \mathbf{e}$. This identification actually corresponds to the times identified by the pseudo-metric defined in \cite[Section 2]{DLG05} and which turns $[0,1]$ into a real tree, but since we will not use it in the following we shall not detail it. Next, we give ourselves a continuous function $ Z = [0,1] \to \mathbb{R}$, which is \textbf{compatible} with $ \mathbf{e}$ in the sense that if $s \sim t$ then we have $Z(s)=Z(t)$. We define a pseudo-distance $ D^{\circ}_{Z}$ on $[0,1]$ by
$$ D^{\circ}_{Z}(s,t) = Z_{s}+ Z_{t} - 2 \cdot m_{Z}(s,t),$$ where we have put 
$$ m_{Z}(s,t)= \max \left\{ \min_{[s \wedge t, s \vee t]} Z ; \min_{[0, s \wedge t] \cup [s \vee t,1]} Z \right\}.$$
The definition above in particular shows that  \begin{eqnarray} \label{eq:moduluscontinuity}\omega_{D^\circ_Z} \leq C \cdot \omega_Z  \end{eqnarray} where $\omega_f$ is the modulus of continuity of the function $f$ and $C>0$ is a constant depending on the norm we choose on $[0,1]^2$. Allowing identifications in $ \mathbf{e}$, we define a pseudo-distance $D^*_{ (\mathbf{e},Z)}$ on $[0,1]$ as the largest pseudo-distance smaller than $D^\circ_Z$ and compatible with $ \sim$, that is 
$$ D^{*}_{( \mathbf{e},Z)}(s,t) =  \inf \left\{\sum_{i=0}^{n-1} D_Z^\circ(s_i,t_{i+1}) : s=s_0,t_1\sim s_1,t_2\sim s_2, \ldots,t_{n-1} \sim s_{n-1},t_n=t\right\}.$$
We denote by $ \mathbf{p}_{D^*}$ the canonical projection $[0,1] \to [0,1] / \{D^*_{( \mathbf{e},Z)}=0\}$ and $ \mathrm{Leb}_{[0,1]}$ the Lebesgue measure on $[0,1]$.
\begin{definition}[Brownian Sphere \cite{MM06,LG07}] The Brownian Sphere is the random compact measured metric space obtained as the quotient metric space $$( \mathbf{m}_\infty, D^{*} , \mu) := \big([0,1] / \{D^*_{( \mathbf{e},Z)}=0\}, D^*_{( \mathbf{e},Z)},\mathbf{p}_{D^*} \circ \mathrm{Leb}_{[0,1]}\big)$$  when $( \mathbf{e},Z)$ is a normalized Brownian excursion $ \mathbf{e}$ together with the head $Z$ of the Brownian snake driven by $ \mathbf{e}$, see  \cite{MM07,LG11,Mie11,le2019brownian,LGMi10} for precise construction of those random processes. \end{definition}

\subsection{Developing the surface and a finite metric space as a proxy}
In order to establish an upper bound on the distance in the $ \mathcal{S}_{n}$ using the contour functions of the labeled tree $ \bT_{n}$, we shall use a proxy surface with a boundary obtained by cutting along the spine\footnote{or cut-locus, hence justifying a posteriori the terminology  ;)} in the construction of Section \ref{sec:bowditch-epstein-penner}. 

Let $X \in \mathcal{M}_{0,n+1}$ be a generic surface with $n+1$ punctures, two of them being distinguished (the root and origin puncture) and recall that $\Sigma(X) \subset X$ is the spine with respect to the origin puncture, which, in generic situations, is an  embedded binary tree whose leaves are at the punctures different from the origin puncture. We shall now cut the surface $X$ along $\Sigma(X)$ to get \textbf{its development} 
$$ \mathrm{Dev}(X).$$
This surface is a hyperbolic polygon with one cusp inside (the origin puncture) and $n$ cusps on its boundary. Since the angles along $\Sigma(X)$ are all less than $ \pi$, the domain $  \mathrm{Dev}(X)$ is convex and in particular star-shaped seen from the origin puncture. Taking the universal cover, we shall represent $\mathrm{Dev}(X)$ in the upper-half plane $ \mathbb{H}$ by putting its inside puncture at $\infty$ and the other punctures become cusp domains touching the $x$-axis periodically, see Figure \ref{fig:devX}. The original hyperbolic surface $X$ can thus be obtained by gluing back the sides of this domain. Notice that this gluing only identifies points at the same height in the upper half-plane since they must be at the same distance from the origin puncture (they belong to the spine). The image of the spine $\Sigma(X)$ inside $ \mathrm{Dev}(X)$ will be called the ``red bottom boundary''. It is made of  $4n-6$ segments of half-circles orthogonal to the $x$-axis: this is just because $\Sigma(X)$ is made of $2n-3$ pieces of hyperbolic geodesics.  In the following, we shall identify the contour of the tree $\bT_{n}$ with the red bottom boundary of $ \mathrm{Dev}(X)$ and in particular, speak of a point visited at time $s \in [0,4n-6]$ or at time $s \in [0,1]$ in the sped-up version.

\begin{figure}[!h]
 \begin{center}
 \includegraphics[width=15cm]{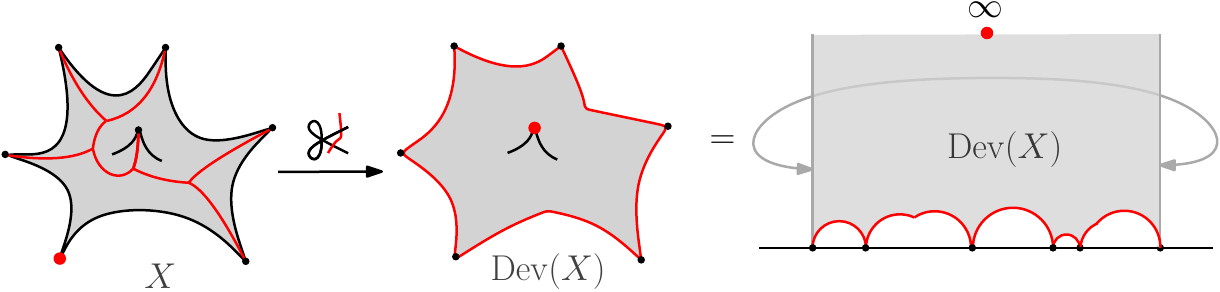}
 \caption{Constructing $ \mathrm{Dev}(X)$ by cutting along the spine seen from $p_{ \mathrm{origin}}$ in the generic hyperbolic surface $X$. The red bottom boundary of $\mathrm{Dev}(X)$ is made of  $4n-6$ pieces of red circles orthogonal to the $x$-axis, which form the graph of a function since $\mathrm{Dev}(X)$ is convex. \label{fig:devX}}
 \end{center}
 \end{figure}

The development $ \mathrm{Dev}(X)$ is seen up to positive dilations and horizontal translations of the upper half-plane. We shall fix a scale below, but before that, let us import some more information. Suppose from now on that $X = \mathrm{Glue}( \bT)$ and recall that each edge $e$ of $ \bT$ corresponds to a geodesic segment of the spine $\Sigma(X)$ and to a geodesic loop $\gamma_{e}$ starting and ending at the origin puncture of signed length $\ell(e)$ when measured from the origin horocycle of length $1$.
The edge $e$ also corresponds to $2$ symmetric geodesic segments (or half lines) in the red bottom boundary of $ \mathrm{Dev}(X)$ which are pieces of (red) circles orthogonal to the $x$-axis and also orthogonal to a side of one of the decorated triangle used to construct $X = \mathrm{Glue}( \bT)$. The radii $ \mathrm{ht}(e)$ those circles is given by 
  \begin{eqnarray} \label{eq:labeldistance} -\log \mathrm{ht}(e) = \frac{\ell(e)}{2} + \mathrm{additive \ constant},   \end{eqnarray} and is attained at the point which in $X$ corresponds to $x_e$ introduced in Section \ref{sec:local}. 

In the sequel, we shall normalize $ \mathrm{Dev}(X)$ so that the maximal radii of the red circles is equal to $1$. If we denote by $ \mathrm{Width}( \mathrm{Dev}(X))$ the horizontal Euclidean width of $ \mathrm{Dev}(X)$;  in this normalization it  satisfies the following deterministic bounds:
  \begin{eqnarray} \label{eq:width}  \mathrm{Width}( \mathrm{Dev}(X)) \leq   2 \cdot (4n-6),  \end{eqnarray}
  just because the red bottom boundary is  composed  of $4n-6$ pieces of circles orthogonal to the $x$-axis, the Euclidean width of each of these pieces is bounded above by twice the radius of the circle, hence by $2$. In this normalization, the piece-wise constant function representing the radii of the current red circle is the function $$  \exp\left(-\frac{1}{2} \left(\ell(e) -\min \ell\right)\right),$$
  with the obvious identification of the pieces of the red bottom boundary with the edges of $\bT$.

\paragraph{Canonical horocycles.}  Recall that both $X$ and $ \mathrm{Dev}(X)$ are of infinite diameter because of the cusps, so we shall consider a finite metric space by cutting off those cusps. More precisely, by a variant of the collar lemma, see \cite[Chap.4 \textsection 4]{BuserBook}, we know that the canonical horocycle neighborhoods of length $1$ in $X$ are all disjoint and are denoted $ \mathfrak{c}_{1}, \mathfrak{c}_{2}, ... ,  \mathfrak{c}_{n+1}$ where $ \mathfrak{c}_{1}$ is associated to the root puncture and $ \mathfrak{c}_{n+1}$ to the origin puncture.

\begin{figure}[!h]
 \begin{center}
 \includegraphics[width=15cm]{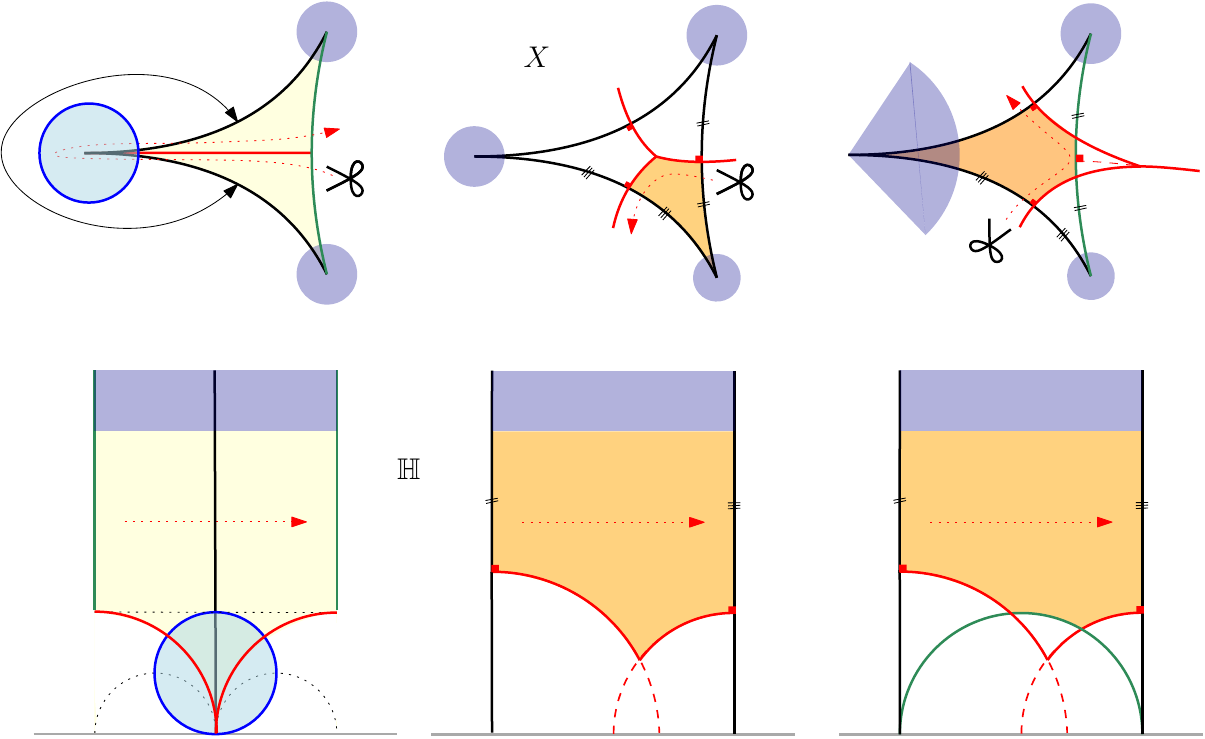}
 \caption{Illustration of the construction of $ \mathrm{Dev}(X)$ for each decorated triangle composing $ X = \mathbf{Glue}( \mathbf{T})$ as in Figure \ref{fig:triangles-type}. In particular, the canonical horocycles attached to the punctures different from the origin puncture are recovered by placing blue circles tangent to the $x$-axis and whose radii is half of the corresponding red circles. The origin horocycle $ \mathfrak{c}_{n+1}$ is displayed in purple and does not intersect those blue circles. \label{fig:horodev}}
 \end{center}
 \end{figure}

\begin{figure}[!h]
 \begin{center}

 \includegraphics[width=15cm]{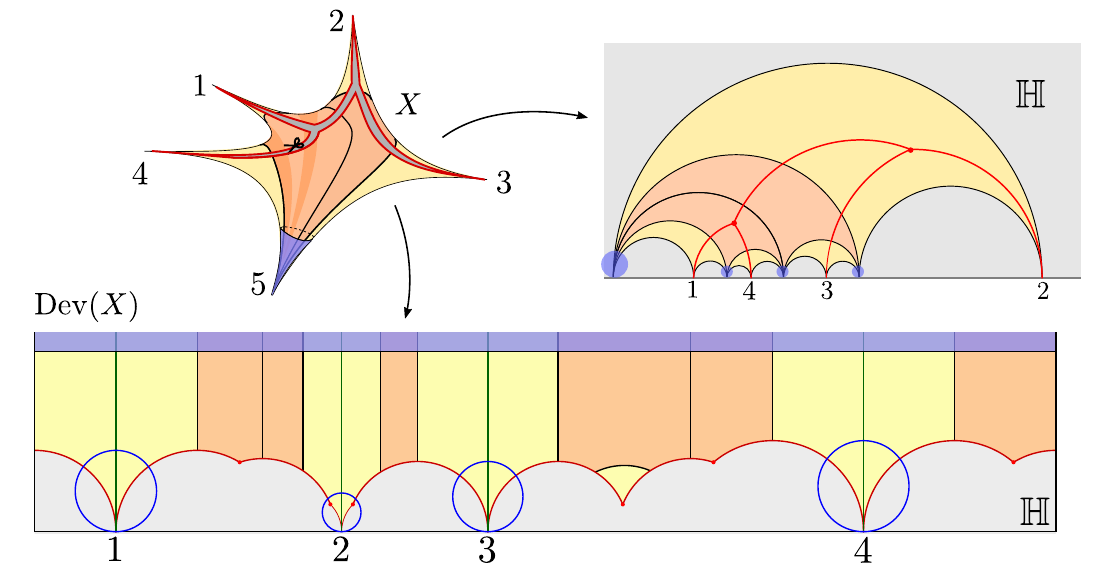}
 \caption{Illustration of the construction of $ \mathrm{Dev}(X)$ where the canonical horocycles (in blue for the non-origin punctures, and in purple for the origin puncture) are displayed.}
 \end{center}
 \end{figure}

  With the above normalization of $ \mathrm{Dev}(X)$, the canonical horocycle $  \mathfrak{c}_{n+1}$ is the horizontal line at height $\mathrm{Width}( \mathrm{Dev}(X))$. The other canonical horocycles $ \mathfrak{c}_{i}$ may come in ``several parts'' in $ \mathrm{Dev}(X)$ but we shall concentrate on the part of the horocycle intersecting the vertical geodesic above the puncture. Since the horocycles are disjoints we deduce with our normalization that $\mathrm{Width}( \mathrm{Dev}(X)) \geq 1$. The red bottom boundary is actually a very good approximation of $\mathrm{Dev}(X)$ since the horizontal projection of any point on it is only within $\log n$ distance:
  \begin{proposition}[Spine approximation]  \label{prop:distancetospine} Let $ h$ be a horizontal segment contained in the closure of $ \mathrm{Dev}(X)$. Then the hyperbolic diameter of $h$ inside  $ \mathrm{Dev}(X)$ is bounded by 
  $$ \mathrm{Diam}_{ \mathrm{hyp}}^{ \mathrm{Dev}(X)}(h) \leq 2 \log n + 10.$$
  \end{proposition}
  \begin{proof} Consider first the case when the segment $h$ has vertical coordinate larger than $1$, so that even if we extend it, it does not intersect the red bottom boundary. In this case, since $\mathfrak{c}_{n+1}$ is made of the horizontal line at height $\mathrm{Width}( \mathrm{Dev}(X))$, and since $ 1 \leq \mathrm{Width}( \mathrm{Dev}(X)) \leq   2 \cdot (4n-6)$ by \eqref{eq:width}, the hyperbolic distance of $  \mathfrak{c}_{n+1}$ to any point of $h$ is less than $\log (2 \cdot (4n-6))$ so that the diameter of $h$ is bounded above by
  $$ 1 +2 \cdot \log (2 \cdot (4n-6)).$$
  The general case can be worked out similarly by adapting \eqref{eq:width}, but we will deduce it from a more general lemma:
 \begin{lemma}
If $P \subset \mathbb{H}$ is a simple hyperbolic $k$-gon whose sides are geodesic or horocyclic, with $P$ being on the outside of the corresponding horocycles, then the total length $w$ of the horocyclic sides is bounded by
\begin{align*}
  w < (k-2)\pi.
\end{align*}
\end{lemma}
\begin{proof}
  Since horocyclic sides have signed geodesic curvature $-1$, Gauss--Bonnet tells us the that the area of $P$ is  \begin{align*}
    \sum_{i=1}^{k} (\pi - \alpha_i) - 2\pi - w,
  \end{align*}   where $\alpha_1,\ldots,\alpha_k \geq 0$ are the angles of $P$.
  The bound then follows from the fact that the area of $P$ is positive.
\end{proof}
If $h \subset \mathrm{Dev}(X)$ is a horizontal segment with vertical coordinate less than $1$, up to extending it, we can consider the polygon $P$ formed by $h$ and the red bottom boundary of $ \mathrm{Dev}(X)$. Since the latter is a horocyclic segment seen from the origin puncture and since the polygon is at most $(4n-6 + 4)$-sided, the previous lemma entails that its hyperbolic length $w$ is less than $w \leq 4 n \pi$. By classical hyperbolic trigonometry, the hyperbolic distance  between any of its points  is then bounded above by 
  \begin{eqnarray} \label{eq:horocyclegeo}  2 \operatorname{arcsinh} \frac{w}{2} \leq 2 \log(w+1) \leq 2 \log ( 4 n\pi +1),  \end{eqnarray}
and the geodesics stay in $ \mathrm{Dev}(X)$ since the latter is star-shaped seen from $ \infty$.
\end{proof}

  Recalling the definition made in the introduction, we denote by $ \mathcal{S}_{n}^{\circ}$ the surface $ \mathcal{S}_{n}$ deprived of its horocycle neighborhoods corresponding to $ \{ \mathfrak{c}_{1}, ... , \mathfrak{c}_{n+1}\}$.  
  Our proxy metric space for $\mathcal{S}_{n}^{\circ}$ will be the finite metric space made of the horocycle neighborhoods $\{ \mathfrak{c}_{1}, ... , \mathfrak{c}_{n+1}\}$ together with the induced hyperbolic distance between them in $ \mathcal{S}_n$ (the distance between two sets $A,B$ is the infimum of the distance between $a \in A$ and $b \in B$). Before tracking precisely the distances in this finite metric space using the contour functions of $ \bT_n$, let us prove that it is a good approximation:

\begin{proposition}[Density of horocycles] \label{prop:density}
  The maximal hyperbolic length of an inner edge of the spine in $\mathcal{S}_n$ is less than $6\log n$  with high probability as $n \to \infty$. 
  The same is true for the segment of leaf edges restricted to $\mathcal{S}_n^\circ$. Also, the maximal length of a a discrete contour interval of $T_{n}$ containing no leaves is bounded above by $3 \log n$ with high probability as $n \to \infty$. In particular,   $$\mathrm{d}_{ \mathrm{Haus}}( \mathcal{S}_{n}^{\circ}, \{ \mathfrak{c}_{1}, ... , \mathfrak{c}_{n+1}\}) < 19 (\log n)^2$$ with high probability as $n \to \infty$
\end{proposition}
\begin{proof} The final display easily follows from the first claims combined with Proposition~\ref{prop:distancetospine}. 

To estimate the length of the edges in the spine, we use the last item in Proposition  \ref{prop:hypeucl}: If $\alpha_{\mathrm{min}}$ is the minimal angle in $ \bT_n$ then the length of every segment is at most $2 \operatorname{arctanh} \cos(\alpha_{\mathrm{min}}) < 2 \log (\pi / \alpha_{\mathrm{min}})$. We shall prove  that $ \alpha_{\mathrm{min}} > n^{-2}$ with high probability so that this length is at most $2 \operatorname{arctanh} \cos(\alpha_{\mathrm{min}}) < 2 \log (\pi / \alpha_{\mathrm{min}}) < 6 \log n$. The length of the portion of a leaf edge in $\mathcal{S}_n^\circ$ with opposite angle $\alpha$ is $\operatorname{arctanh} \cos(\alpha) + \log 2$, so the same bound holds. Let us now prove that $\alpha_{\mathrm{min}} > n^{-2}$ with high probability: We saw in during the proof of Lemma \ref{lem:roughexpo} (just after \eqref{eq:densitybound}) that the density of the first two angles of $\bT$ under $ \mathbb{P}_{\theta}$ is bounded above by some absolute constant (independent of $\theta$). By exploring the tree in the lexicographical order, it follows by a union bound  that 
$$ \mathbb{P}\left(\min_{ c \in \mathrm{Corners}(  \bT_{n}) } \theta_{c} \leq \varepsilon \right) \leq  \frac{\mathrm{C} n \mathbb{P}(U \leq \varepsilon)}{ \mathbb{P}_{0} ( \btau \mbox{ has $n$ leaves})},$$ for some $ \mathrm{C}>0$. Taking $ \varepsilon = n^{-2}$ we deduce thanks to \eqref{eq:asymptgeneral} that the minimal angle in $ \bT_{n}$ is larger than $n^{-2}$ with high probability as $n \to \infty$. 
The estimate on contour intervals with no leaf is similar. Notice that by \eqref{eq:GWsplit} under $  \mathbb{P}_{\theta}$ the probability that the red tree is made a single leaf is bounded from below by 
$$ \frac{x_{c}}{ \sup_{\theta \in [0, \pi]} F(\theta)} = \frac{x_{c}}{F(0)}  \approx 0.43... $$
In particular, the probability that the contour exploration of $T_{n}$ contains a segment of $k$ steps without leaf is bounded  from above by 
$$ \frac{n ( 0.43...)^{k}}{\mathbb{P}_{0} ( \btau \mbox{ has $n$ leaves})},$$ and the previous display tends to $0$ as $n \to \infty$ for $k = 3 \log n$ using \eqref{eq:asymptgeneral}.
\end{proof}
  
\subsection{Upper bound on distances from contour functions}

\label{sec:schaefferhyp}
Suppose that $X\equiv \mathcal{S}_n = \mathrm{Glue}( \bT_n)$ is encoded by the labeled binary red tree $\bT_n$.
Recall from  Theorem \ref{thm:scaling} the definition of $\textbf{C}_{\bT_n}, \mathbf{Z}_{\bT_n}, \mathbf{R}_{ \bT_n}$ for the contour, label  and leaf-counting processes when following the contour of the labeled tree $ \bT_{n}$. We consider here their rescaled and sped-up analogs defined on $[0,1]$:
$$ \textbf{C}^{(n)}(t) = \frac{\mathbf{C}_{\bT_n}(t \cdot (4n-6))}{ \sqrt{n}}, \quad \textbf{Z}^{(n)}(t) = \frac{ \mathbf{Z}_{\bT_n}(t \cdot (4n-6))}{  \sqrt{\sqrt{n}}} \quad \mbox{ and } \quad \textbf{R}^{(n)}(t) = \frac{ \mathbf{R}_{\bT_n}(t \cdot (4n-6))}{n}.$$

 If  $0 \leq s,t  \leq 4n-6$ are two times during which the contour of $  \bT_{n}$ visits a leaf (including the root leaf), then we define $\tilde{D}^{(n)}(s,t)$ as the hyperbolic distance in $ \mathcal{S}_{n}$ between the two associated canonical horocycles $\mathfrak{c}_{\cdot}$. We extend it by putting $ \tilde{D}^{(n)}(s,t) := \tilde{D}^{(n)}(s',t')$ where $s'$ and $t'$ are respectively the visit times of red leaves closest to $s$ and $t$ in $[0,1]$ (we break ties arbitrarily when they appear). We also renormalize as above by setting $$D^{(n)}(s,t) = \frac{ \tilde{D}^{(n)}((4n-6)s, (4n-6)t)}{ \sqrt{ \sqrt{n}}}.$$
 
 In particular, $D^{(n)}$ is a pseudo-distance over $[0,1]$ and we have equality in terms of metric spaces
$$\left( [0,1] / \{D^{(n)}=0\} ; D^{(n)} \right) =\left(\{  \mathfrak{c}_{1},\mathfrak{c}_2, ... , \mathfrak{c}_{n}\},  n^{-1/4} \cdot \mathrm{d_{hyp}}\right),$$
 (note the absence of the origin horocycle $ \mathfrak{c}_{n+1}$). Our essential step in proving Theorem \ref{thm:GH} will be to show that $(\{ \mathfrak{c}_1, ... , \mathfrak{c}_n\},  n^{-1/4} \cdot \mathrm{d_{hyp}})$ converges towards a multiple of the Brownian sphere:
\begin{proposition}[Convergence of the horocycles metric space] \label{prop:horoBM} We have the convergence in the Gromov--Hausdorff topology
$$\left(\{  \mathfrak{c}_{1}, ... , \mathfrak{c}_{n}\},  n^{-1/4} \cdot \mathrm{d_{hyp}}\right) \xrightarrow[n\to\infty]{(d)} (  \mathbf{m}_\infty, c_{ \mathrm{wp}} \cdot D^*),$$
with $c_{ \mathrm{wp}} = c_2 / 2 = \frac{2\pi}{\sqrt{3 c_0}}$.
\end{proposition}
 To do so we shall estimate the hyperbolic distance between the horocycles $ \mathfrak{c}_i$ using the functions $  \textbf{C}^{(n)}$ and $  \textbf{Z}^{(n)}$. More precisely, we shall establish an \textit{upper bound} on the distances $D^{{(n)}}$ which is then shown to be sharp in the scaling limit using  the rerooting trick of Le Gall \cite{LG11}.

\paragraph{Distances to the origin horocycle.} Let us start by estimating the distance between $ \mathfrak{c}_i$ and the origin canonical horocycle $ \mathfrak{c}_{n+1}$. This is clearly done inside $ \mathrm{Dev}(X)$: one should go from the highest point of the horocycle $  \mathfrak{c}_{i}$ straight above to $ \mathfrak{c}_{n+1}$ located at height $ \mathrm{Width}( \mathrm{Dev}( \mathcal{S}_n))$. In particular, using \eqref{eq:labeldistance}, if $e_{i}$ is the edge incident to the leaf of $\bT_{n}$ corresponding to $ \mathfrak{c}_{i}$ we have the \textit{equality}
 \begin{eqnarray}  \mathrm{d_{hyp}}( \mathfrak{c}_{n+1}, \mathfrak{c}_{i}) =  \frac{1}{2}\left(\ell_{e_i} - \min \ell \right)+ \log( \mathrm{Width}( \mathrm{Dev}( \mathcal{S}_n))).   \label{eq:distformula}\end{eqnarray}
This formula is the analog of the well-known formula which recovers the distance to the origin from the labels in Schaeffer-type encodings, see \cite[Eq (3)]{LG11}.

\paragraph{$D^\circ$ upper bound.} 

Recall from Section \ref{sec:Bmap} the notation $D^{\circ}_{Z}$ for a continuous function $[0,1] \to \mathbb{R}$.  We shall now establish the upper bound on the distances that is analogous to the classic ``Schaeffer bound'' in the setting of random planar maps, see  \cite[Eq (4)]{LG11}:

\begin{proposition}[$D^\circ$-type upper bound on distances] \label{prop:Doupper}Let $s,t \in [0,1]$ such that $(2n-6)s$ and $(2n-6)t$ are two visit times of leaves of $\bT_n$, then we have 
 \begin{eqnarray*}  {D}^{(n)}(s,t) &\leq&   \frac{1}{2}D^{\circ}_{{ \mathbf{Z}}^{(n)}}(s,t)  + \frac{2 \log n + 10}{n^{1/4}}.  \end{eqnarray*}
\end{proposition}
\proof To lighten notation  we put $\hat{\textbf{Z}}^{(n)}(s) =  \frac{1}{2} (\textbf{Z}^{(n)}(s) - \min \textbf{Z}^{(n)}(s))$. The proof is best explained on a drawing, see Figure \ref{fig:Doupper} where we suppose that $0 < s <t < 1$ are two visits of leaves in the sped-up contour process. Suppose that the punctures visited are number $i$ and $j$ respectively. Let us look at  the corresponding horocycles $ \mathfrak{c}_{i}$ and $ \mathfrak{c}_{j}$ in the development of the surface in the upper half-plane after cutting along the spine. Even if these horocycles may correspond to several ``pieces'' in the development, we shall only look at the distance between the points on $ \mathfrak{c}_{i}$ and $ \mathfrak{c}_{j}$ at the vertical above the puncture in the development. By \eqref{eq:labeldistance}, the highest point of those horocycles are at height precisely  
$$ \mathrm{exp}\left( -  n^{1/4} \hat{\textbf{Z}}^{(n)}(s)\right) \quad \mbox{ and } \quad  \mathrm{exp}\left( - n^{1/4} \hat{\textbf{Z}}^{(n)}(t)\right).$$

\begin{figure}[!h]
 \begin{center}
 \includegraphics[width=15cm]{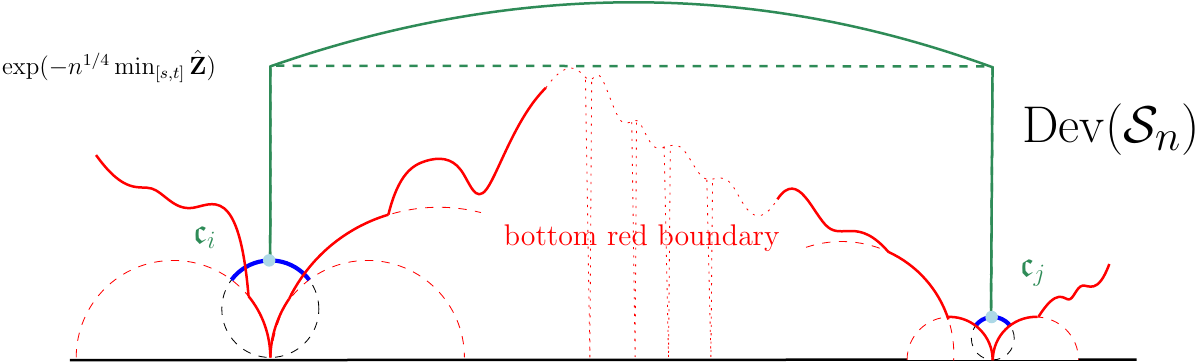}
 \caption{The $D^\circ$-type upper bound on distances between horocycles. At time $s,t$ the sped-up contour process visit leaf-eges corresponding to the two horocycles drawn in blue on the figure. We then create a path staying in the development that connects those two horocycles in dark green. \label{fig:Doupper}}
 \end{center}
 \end{figure}
A simple way to build a path inside $ \mathrm{Dev}( \mathcal{S}_{n})$ between those points is to go straight up to a vertical height equal to $$ h_{ \mathrm{max}} = \exp\left(- n^{1/4}\min_{[s,t]}  \hat{\textbf{Z}}^{(n)}\right)$$ and connect using a horizontal segment as in Figure \ref{fig:Doupper}.  
The hyperbolic length of such a path is bounded above using Proposition \ref{prop:distancetospine} 
$$ \underbrace{n^{1/4} \left(\hat{\textbf{Z}}^{(n)}(s) + \hat{\textbf{Z}}^{(n)}(t) - 2 \min_{[s,t]} \hat{\textbf{Z}}^{(n)} \right)}_{ \mathrm{ vertical \ parts}} \quad + \quad  \underbrace{2 \log n + 10}_{ \mathrm{``horizontal'' \ part}}.$$
We then repeat the argument for paths going from $t$ to $s$ crossing over time $1$ in the sped-up contour to get the desired upper bound. \endproof 

Beware, the upper bound of the last proposition is not directly true for times $s,t$ which are not visit times of leaves but since visit times of leaves are not more than $3 \log n $ apart in the contour process by Proposition \ref{prop:density}, combined with Theorem \ref{thm:scaling} we deduce    \begin{eqnarray} \label{eq:Doupbis} \forall s,t \in [0,1], \qquad 
 {D}^{(n)}(s,t) &\leq&   \frac{1}{2}D^{\circ}_{{ \mathbf{Z}}^{(n)}}(s,t) + o_ \mathbb{P}(1), \end{eqnarray} where $o_ \mathbb{P}(1)$ is a random variable independent of $s,t$ which converges to $0$ in probability. 

\subsection{$D^{*}$-upper bound}
We shall now refine our upper bound on the distances by using the identification made in the surface after gluing back the development along the spine. It is simpler to establish these bounds directly in the scaling limit and so we shall first take sub sequential limits of our pseudo-metrics $ D^{(n)}$.

\paragraph{Subsequential scaling limits.} Recall that $D^{(n)}$ is a pseudo-metric on $[0,1]$ such that the quotient $[0,1]/ D^{(n)}$ is just the metric space made over $ \{ \mathfrak{c}_{1}, ... , \mathfrak{c}_{n}\}$ with the renormalized distances $ n^{-1/4}\cdot \mathrm{d_{hyp}}$. Since visit times of leaves become uniformly spread in $[0,1]$ thanks to convergence of the third coordinate in Theorem \ref{thm:scaling}, it is easy to deduce from \eqref{eq:Doupbis} and the tightness of the process $ \textbf{Z}^{(n)}$ in Theorem \ref{thm:scaling} that the sequence of processes
$$ \left(D^{(n)}(s,t)\right)_{s, t \in [0,1]},$$ is tight in the space of probability measures on $ C([0,1]^{2}, \mathbb{R})$, see \cite[Proposition 3.2]{LG07} for similar arguments. Using Skorokhod representation theorem, from any monotone increasing sequence of positive integers, we can extract a subsequence $(n_k)_{k \geq 1}$ such that the preceding convergences are almost sure i.e. 
 \begin{eqnarray} \label{eq:skorokhod} \left(\textbf{C}^{(n)},\textbf{Z}^{(n)}, \textbf{R}^{(n)},D^{(n)}\right)  \xrightarrow[n\to\infty]{a.s.} ( c_{1} \cdot \mathbf{e}, c_{2} \cdot Z, \frac{t}{4}, \mathfrak{D}),  \end{eqnarray} where we restrict to values of $n$ in the extraction and where $ \mathfrak{D}$ is a random continuous pseudo-metric  over $[0,1]^{2}$. 
In terms of metric spaces,  (see e.g.~\cite[p645]{LG07} or \cite[Proposition 7.1]{CRM22} for details) this shows that with this coupling and along the sequence $(n_{k})$ we have 
 
   \begin{eqnarray}\left(\{  \mathfrak{c}_{1}, ... , \mathfrak{c}_{n}\},  n^{-1/4} \cdot \mathrm{d_{hyp}}\right) \xrightarrow[n_{k}\to\infty]{a.s.} ( [0,1] \backslash\{ \mathfrak{D}=0\} ; \mathfrak{D}).   \label{eq:consequencesko}\end{eqnarray}
Hence to prove Proposition \ref{prop:horoBM}, it remains to show that $\mathfrak{D}$ is actually equal to $c_{2}/2 \cdot D^{*}_{( \mathbf{e},Z)}$ regardless of the extraction chosen. One first piece of information is obtained by passing to the limit in  \eqref{eq:Doupbis} and \eqref{eq:distformula} (using Theorem \ref{thm:scaling} as always) and deduce 
\begin{eqnarray} \label{eq:lowertrivial}\forall s, t \in [0,1], \quad  \frac{c_2}{2}\cdot |Z_{s}-Z_{t}| \leq \mathfrak{D}(s,t)\leq  \frac{c_2}{2}\cdot D^{\circ}_{Z}(s,t)     \end{eqnarray}
The second step is to show that $ \mathfrak{D}$ is compatible with $ \mathbf{e}$ in the following sense:

\begin{proposition}[Identifications in $ \mathbf{e}$]In the above notation if $s,t \in [0,1]$ are identified by $ \mathbf{e}$ (i.e.~if $ \mathbf{e}(s)= \mathbf{e}(t) = \min_{[s\wedge t, s \vee t]} \mathbf{e}$) then we have $$ \mathfrak{D}(s,t)=0.$$
\end{proposition}
\begin{proof} Suppose that $s < t$ to fix ideas. By standard properties of the Brownian excursion (local minima are distinct almost surely), we have one of the following scenarii:
\begin{itemize}
\item either $  \mathbf{e}(u) > \mathbf{e}(s)= \mathbf{e}(t)$ for every $u \in (s ,t)$
\item or there exist a unique $u \in (s, t)$ such that $ \mathbf{e}(u)= \mathbf{e}(s)= \mathbf{e}(t)$. In this case, $s$ and $t$ are no local minima for the Brownian excursion.
\end{itemize} In both cases, using the convergence $ \textbf{C}^{(n)} \to \mathbf{e}$, a moment's though shows that one can find sequences $s_{n} \to s$ and $t_{n} \to t$ of (rescaled) times such that $s_{n}$ and $t_{n}$ are identified via $ \textbf{C}^{(n)}$, i.e.~they correspond to the visit of a same vertex of $\bT_{n}$ (which is not a leaf). This means that the point $x_n$ visited at time $s_n$ in the red bottom boundary of $ \mathrm{Dev}( \mathcal{S}_{n})$ is identified to the point $y_n \in \mathrm{Dev}( \mathcal{S}_{n})$ visited at time $t_n$. 
Since $s\ne t$, the times $s_n$ and $t_n$ correspond to visit times of interior edges, not edges adjacent to leaves. 
If we consider the times $\tilde{s}_{n}$ and $ \tilde{t}_{n}$ corresponding to the respective closest rescaled visit times of leaves with associated horocycles $ \mathfrak{c}_{i_{n}}$ (resp. $ \mathfrak{c}_{j_{n}}$ ), by Proposition \ref{prop:density} we have $|\tilde{s}_{n}-s_{n}| \leq 3 \log n /(4n-6)$ with high probability  (similarly $ |\tilde{t}_{n}-t_{n}| \leq 3 \log n /(4n-6)$ and 
$$ \max \left(\mathrm{d}_{ \mathrm{hyp}}^{\mathrm{Dev}( \mathcal{S}_{n})}(x_{n},  \mathfrak{c}_{i_{n}}),\mathrm{d}_{ \mathrm{hyp}}^{\mathrm{Dev}( \mathcal{S}_{n})}(y_{n},  \mathfrak{c}_{j_{n}})\right)  \leq 18 (\log n)^{2} \quad \mbox{ with high probability}.$$ Gathering up the pieces we have  \begin{eqnarray*}  \mathfrak{D}(s, t) = \lim_{n \to \infty}D^{(n)}(s_{n},t_{n}) &=&  \lim n^{-1/4} \mathrm{d}_{ \mathrm{hyp}}^{ \mathcal{S}_{n}}(  \mathfrak{c}_{i_{n}},\mathfrak{c}_{i_{n}} )\\ & \leq & \lim n^{-1/4}\left(\mathrm{d}_{ \mathrm{hyp}}^{\mathrm{Dev}( \mathcal{S}_{n})}( \mathfrak{c}_{i_{n}},x_{n}) + \underbrace{\mathrm{d}_{ \mathrm{hyp}}^{\mathcal{S}_{n}}(  x_{n},y_{n})}_{=0}+ \mathrm{d}_{ \mathrm{hyp}}^{\mathrm{Dev}( \mathcal{S}_{n})}(  y_{n}, \mathfrak{c}_{j_{n}}) \right)\\
& =& 0.  \end{eqnarray*}
\end{proof}

Since $ \mathfrak{D}$ is almost surely a pseudo-metric, is it smaller than the largest pseudo-metric compatible with $ \mathbf{e}$ and smaller than $c_{2}/2 \cdot D^{\circ}_{Z}$, which is precisely $c_{2}/2\cdot D^{*}_{( \mathbf{e},Z)}$, we then have
  \begin{eqnarray} \label{eq:upperboundDstar}\mathfrak{D} \leq c_{2}/2\cdot D^{*}_{( \mathbf{e},Z)}. \label{eq:D<D*}  \end{eqnarray}

\subsection{Le Gall's rerooting trick}
We now show that $ \mathfrak{D} = c_{2}/2\cdot D^{*}_{( \mathbf{e},Z)}$ using Le Gall rerooting trick \cite[Section 9]{LG11}. The essence of the idea is to look at distances between two uniform points and to use invariance under rerooting. Let us introduce $t^{*} = \mathrm{argmin} Z$ which almost surely unique by \cite[Proposition 2.5]{LGW06}. 
\begin{proposition}[Rerooting]  \label{prop:rerooting}With the above notation we have 
\begin{itemize}
\item For every $s\in [0,1]$ we have $  \mathfrak{D}(t^{*},s) = c_2/2 \cdot D^*_{( \mathbf{e},Z)}(t^{*},s) = c_2/2 \cdot ( Z_{s} - \min Z)$,
\item If $U,V$ are two independent variables uniformly distributed over $[0,1]$ also independent of $\mathfrak{D}$  then we have 
$$ \mathfrak{D}(U,V) \overset{(d)}{=} \mathfrak{D}(t^{*},U).$$
\end{itemize}\end{proposition} 
\begin{proof}[Proof of Proposition \ref{prop:horoBM} given Proposition \ref{prop:rerooting}] This is Le Gall's rerooting trick, see \cite{AA13} or \cite[Section 7.4.2]{CRM22} for examples of application of this technique: The invariance under rerooting of the Brownian sphere  \cite[Section 8.3]{LG11} shows that   \begin{eqnarray} \label{eq:rerootBS} D^*_{( \mathbf{e},Z)}(t^*,U)  \overset{(d)}{=} D^*_{( \mathbf{e},Z)}(U,V),  \end{eqnarray}where $U,V$ are also independent of $( \mathbf{e},Z)$. In particular, by Proposition \ref{prop:rerooting} the random variables $c_2/2 \cdot D^*_{( \mathbf{e},Z)}(U,V)$ and $ \mathfrak{D}(U,V)$ have the same law . We deduce that the almost sure bound $  \mathfrak{D}(U,V) \leq c_2/2 \cdot D^*_{( \mathbf{e},Z)}(U,V)$  obtained in \eqref{eq:upperboundDstar} is saturated (it is an almost sure equality) and by continuity that $ \mathfrak{D}= c_2/2 \cdot D^*_{( \mathbf{e},Z)}$ a.s. In particular \eqref{eq:skorokhod} holds without passing to a subsequence with $ \mathfrak{D} = c_{2}/2 D^*_{( \mathbf{e},Z)}$ and so does \eqref{eq:consequencesko}. 
\end{proof}

\begin{proof}[Proof of Proposition \ref{prop:rerooting}.] The items are proved by passing to the limit in the discrete setting as in \cite{LG07,LG11}. More precisely, recall from \eqref{eq:distformula} that the shifted labels represent, up to an additive factor $ \log ( \mathrm{Width}( \mathrm{Dev}( \mathcal{S}_{n}))= o(n^{{1/4}})$, half of the distances from the horocycles $ \mathfrak{c}_{i}$  for $1 \leq i \leq n$ to  $ \mathfrak{c}_{n+1}$.  If we introduce $i^{*}_{n}$ and $t^{*}_{n}$ the index and visit time (in the sped-up contour) of the leaf with the minimal label then since $t^{*}$ is the only minimum time of $Z$, by Theorem \ref{thm:scaling} we have $t_{n}^{*}\to t^*$ and $i^{*}_{n}/n \to t^*$. In particular  $ \mathfrak{c}_{i_n^*}$ is close to the origin horocycle $ \mathfrak{c}_{n+1}$ since \begin{eqnarray} \label{eq:presqueorigin} \mathrm{d_{hyp}}( \mathfrak{c}_{i^{*}_{n}},  \mathfrak{c}_{n+1}) = \log( \mathrm{Width}( \mathrm{Dev}( \mathcal{S}_{n}))  \underset{ \eqref{eq:width}}{=} o(n^{1/4}).  \end{eqnarray} We deduce that for any sequence of visit times $s_{n}$ of leaves number $j_{n}$ such that $s_{n} \to s$ we have
 \begin{eqnarray*}\mathfrak{D}(t^{*},s)= \lim_{n \to \infty} D^{(n)}(t^{*}_{n},s_n) &=& \lim_{n \to \infty} n^{{-1/4}}\mathrm{d_{hyp}}( \mathfrak{c}_{i_{n}^{*}}, \mathfrak{c}_{j_{n}})\\
 & \underset{ \eqref{eq:presqueorigin}}{=} & \lim_{n \to \infty} n^{{-1/4}}\mathrm{d_{hyp}}( \mathfrak{c}_{n+1}, \mathfrak{c}_{j_{n}})\\
  & \underset{ \eqref{eq:distformula}, \eqref{eq:skorokhod}}{=} & c_{2}/2\cdot  (Z_{s}- \min Z) = c_{2}/2 \cdot  (Z_{s}-Z_{t^{*}}). \end{eqnarray*} This proves the first statement of the proposition. For the second one, we use symmetry in the discrete model and consider the $n+1$ horocycles equivalently. Let $i_n = \lceil (n+1) U \rceil$ and $j_n = \lceil (n+1) V \rceil$ where $U,V$ are independent uniform variables also independent of all surfaces $ (\mathcal{S}_n)_{n \geq 1}$ and of $( \mathbf{e},Z, \mathfrak{D})$. Then we clearly have by symmetry
  $$ \mathrm{d_{hyp}}( \mathfrak{c}_{n+1},  \mathfrak{c}_{i_{n}}) \overset{(d)}{=} \mathrm{d_{hyp}}( \mathfrak{c}_{i_n},  \mathfrak{c}_{j_{n}}),$$ whereas in the coupling \eqref{eq:skorokhod} we have by the above arguments
  $$\lim_{n \to \infty} n^{-1/4} \cdot \mathrm{d_{hyp}}( \mathfrak{c}_{n+1},  \mathfrak{c}_{i_{n}}) =  \mathfrak{D}(t^{*},U)  \quad \mbox{ and } \quad \lim_{n \to \infty} n^{-1/4} \cdot \mathrm{d_{hyp}}( \mathfrak{c}_{i_n},  \mathfrak{c}_{j_{n}}) = \mathfrak{D}(U,V).$$
The proof is complete.\end{proof} 

\subsection{Proof of Theorem \ref{thm:GH}}
 \label{sec:finalproof}
 
 Given Proposition \ref{prop:horoBM}, Equation \eqref{eq:presqueorigin} and Proposition \ref{prop:density}, we directly get the Gromov--Hausdorff convergence of $( \mathcal{S}_{n}^{\circ}, n^{-1/4} \mathrm{d}_{ \mathrm{hyp}})$ towards the same multiple of the Brownian sphere. Notice that if we change the definition of $ \mathcal{S}_{n}^{\circ}$ to cut at smaller horocycle neighborhoods say of length $ 0 < \varepsilon < 1$ and get $ \mathcal{S}_{n}^{\circ, \varepsilon}$,  since $ \mathrm{d}_{ \mathrm{Haus}} (\mathcal{S}_{n}^{\circ},\mathcal{S}_{n}^{\circ, \varepsilon}) \leq  - \log \varepsilon$ the same result holds for $\mathcal{S}_{n}^{\circ, \varepsilon}$. We shall now prove the Gromov--\textit{Prokhorov} convergence  $$\big(\mathcal{S}_{n}, n^{-1/4} \cdot \mathrm{d_{hyp}},  (2 \pi (n-1))^{-1}\mu_{ \mathrm{hyp}}\big) \xrightarrow[n\to\infty]{(d)} (  \mathbf{m}_\infty, c_2/2 \cdot D^*, \mu)$$ where $\mu_{ \mathrm{hyp}}$ is the unnormalized hyperbolic measure on $ \mathcal{S}_n$. To see this, conditionally on $ (\mathcal{S}_{n})$ and on $( \mathbf{e},Z, \mathfrak{D})$, let us sample $ \circ_{n} \in \mathcal{S}_{n}$ uniformly at random according to the normalized version of $ \mu_{ \mathrm{hyp}}$. After projecting $\circ_{n}$ in $ \mathrm{Dev}( \mathcal{S}_{n})$ (it has probability $0$ to land on the spine), we denote by $H_{n}$ the time in the sped-up contour when the process visits the abscissa of $\circ_{n}$ and denote by $ \tilde{H}_{n}$ its closest leaf-time on $[0,1]$ associated with the horocycle $ \mathfrak{c}_{ I_{n}}$. We will show that 
  \begin{eqnarray} n^{-1/4} \cdot \mathrm{d}_{ \mathrm{hyp}}^{ \mathrm{Dev}(X)} ( \circ_{n}, \mathfrak{c}_{I_{n}}) &\xrightarrow[n\to\infty]{( \mathbb{P})}& 0,  \label{eq:goalfin1}\\
   \mathrm{Law}( \tilde{H}_{n}) & \xrightarrow[n\to\infty]{ ( \mathbb{P})}& \mathrm{Unif}[0,1],  \label{eq:goalfin2} \end{eqnarray}
   from which it is easy to deduce the Gromov--Prokhorov convergence: just consider the relation $ \mathcal{R}_{n}$ between $ (\mathcal{S}_{n}, n^{-1/4} \cdot \mathrm{d}_{ \mathrm{hyp}},(2 \pi (n-1))^{-1}\mu_{ \mathrm{hyp}}) $ and $ ([0,1] ,  c_{2}/2\cdot D^{*}, \mu)$ made of all $(\circ_{n}, \tilde{H}_{n})$ with the above notation. The previous two displays together with  the convergence \eqref{eq:skorokhod} with $ \mathfrak{D} = c_{2}/2 \cdot D^{*}$ shows using \eqref{def:GPcor} that the desired Gromov--Prokhorov distance tends to $0$. Let us now show \eqref{eq:goalfin1} and \eqref{eq:goalfin2}.
To treat the first one, let us make a couple of simple observations. First of all, conditionally on the horizontal coordinate not being that of a cusp --which appears with  probability one--, the hyperbolic distance of $\circ_{n}$ to its vertical projection on the red bottom boundary is exponentially distributed with mean $1$. Similarly, conditionally on $\circ_{n}$ belonging  to one of the cusps $\{ \mathfrak{c}_{1}, ... ,  \mathfrak{c}_{n}\}$ its distance to the boundary of the cusp is again exponentially distributed. Combined with Proposition \ref{prop:density}, we have that $$\mathrm{d}_{ \mathrm{hyp}}^{ \mathrm{Dev}(X)} ( \circ_{n}, \mathfrak{c}_{I_{n}}) \leq   \mathcal{E} + \mathcal{E}' + 18 \log^{2} n,$$ with high probability as $n \to \infty$ where $ \mathcal{E}, \mathcal{E}$ are now exponential variables. This grants \eqref{eq:goalfin1}.

Let us now move to the second point. Although each vertical strip of the development may not have the same hyperbolic area, we can cut and regroup them by pairs or triplets so that they form ideal hyperbolic triangles of constant hyperbolic area $\pi$. More precisely, any ideal triangle in the decomposition of $ \mathcal{S}_{n}$ is traversed by at most three segments or one-half line and two segments of the spine. They correspond to the three edges adjacent to an inside vertex of $T_{n}$ or to the leaf edge as well as its two neighboring edges in the case of a leaf. When all sides of theses edges have been visited in the contour, we are sure that the associated triangle is contained in the corresponding part of the development. In particular, the probability that $H_{n} \in [s,t]$ is larger than $\pi N_{[s,t]} /   (2 \pi (n-1))$ where $N_{[s,t]}$ is the number of vertices of $T_{n}$ for which all edges within distance $2$ have been fully visited during the time interval $[s,t]$. The convergence \eqref{eq:skorokhod} together with the fact that the Brownian excursion admits no local minima at any fixed times $s,t$ then implies that 
$$ \forall 0 \leq s, t \leq 1, \quad  \liminf_{n \to \infty} \mathbf{P}( H_{n} \in [s,t]) \geq |t-s|,$$ almost surely. Since both sides are probability measures, the previous inequality is saturated and we get that the law of $H_{n}$ converges in probability towards the uniform law on $[0,1]$. The asymptotic equidistribution of the leaves shows that the same is true for $ \tilde{H}_{n}$  thus establishing \eqref{eq:goalfin2}. \qed 
 
\bibliographystyle{siam}

\bibliography{bibli}
\end{document}